\documentclass[twoside,11pt]{article}

\usepackage{blindtext}

\usepackage{amsthm} 
\usepackage[abbrvbib,preprint]{jmlr2e}
\usepackage{amsmath}
\usepackage[capitalize]{cleveref}

\makeatletter
\newtheorem*{rep@theorem}{\rep@title}
\newcommand{\newreptheorem}[2]{%
\newenvironment{rep#1}[1]{%
 \def\rep@title{#2 \ref{##1}}%
 \begin{rep@theorem}}%
 {\end{rep@theorem}}}
\makeatother

\newtheorem{theorem}{Theorem}
\newreptheorem{theorem}{Theorem}
\newtheorem{lemma}[theorem]{Lemma}
\newreptheorem{lemma}{Lemma}
\newtheorem{proposition}[theorem]{Proposition}
\newreptheorem{proposition}{Proposition}
\newtheorem{remark}[theorem]{Remark}

\newreptheorem{corollary}{Corollary}
\newtheorem{definition}[theorem]{Definition}

\renewenvironment{proof}[1][\proofname]{\par\noindent{\bf #1.\ }}{\hfill\BlackBox\\\vspace{2mm}}

\usepackage{amsfonts}%
\usepackage{amssymb}%
\usepackage{mathrsfs}
\usepackage{overpic}
\usepackage{mathtools}
\usepackage{enumitem}
\usepackage{bm}

\newcommand{\R}{\mathbb{R}}
\newcommand{\F}{\textup{F}}
\newcommand{\C}{\mathbb{C}}
\renewcommand{\H}{\mathcal{H}}
\newcommand{\M}{\mathcal{M}}

\DeclareMathOperator{\Tr}{Tr}

\renewcommand{\L}{\mathcal{L}}
\newcommand{\diam}{\textup{diam}}

\renewcommand{\d}{\,\textup{d}}

\newcommand{\vertiii}[1]{{\left\vert\kern-0.25ex\left\vert\kern-0.25ex\left\vert #1 
    \right\vert\kern-0.25ex\right\vert\kern-0.25ex\right\vert}}
\newcommand{\innerii}[1]{{\left\langle\kern-0.25ex\left\langle #1 
    \right\rangle\kern-0.25ex\right\rangle}}

\newcommand{\mcal}[1]{{\mathcal{#1}}}

\DeclareMathOperator{\rank}{rank}

\DeclareMathOperator{\vspan}{span}
\DeclareMathOperator{\Range}{Range}
\DeclareMathOperator{\Null}{Null}
\DeclareMathOperator{\grad}{\nabla}
\DeclareMathOperator{\vdiv}{div}
\DeclareMathOperator{\D}{\mathrm{D}}
\renewcommand{\div}{\mathrm{div}}

\usepackage{algorithm}
\usepackage{algpseudocode}

\usepackage{lastpage}
\jmlrheading{25}{2024}{1-\pageref{LastPage}}{1/24}{9/24}{24-0162}{Boull\'e, Halikias, Otto, and Townsend}

\ShortHeadings{Operator learning without the adjoint}{Boull\'e, Halikias, Otto, and Townsend}
\firstpageno{1}

\usepackage{xcolor}

\begin{document}

\title{Operator learning without the adjoint}

\author{\name Nicolas Boull\'{e} \email n.boulle@imperial.ac.uk \\
\addr Department of Mathematics\\
Imperial College London\\
London, SW7 2AZ, UK
\AND
\name Diana Halikias \email dh736@cornell.edu \\
\addr Department of Mathematics\\
Cornell University\\
Ithaca, NY 14853, USA
\AND
\name Samuel E. Otto \email s.otto@cornell.edu \\
\addr Sibley School of Mechanical and Aerospace Engineering \\
Cornell University\\
Ithaca, NY 14853, USA
\AND
\name Alex Townsend \email townsend@cornell.edu \\
\addr Department of Mathematics\\
Cornell University\\
Ithaca, NY 14853, USA
}

\editor{Mehryar Mohri}

\maketitle

\begin{abstract}
    There is a mystery at the heart of operator learning: how can one recover a non-self-adjoint operator from data without probing the adjoint? Current practical approaches suggest that one can accurately recover an operator while only using data generated by the forward action of the operator without access to the adjoint. However, naively, it seems essential to sample the action of the adjoint. In this paper, we partially explain this mystery by proving that without querying the adjoint, one can approximate a family of non-self-adjoint infinite-dimensional compact operators via projection onto a Fourier basis.  We then apply the result to recovering Green's functions of elliptic partial differential operators and derive an adjoint-free sample complexity bound. While existing theory justifies low sample complexity in operator learning, ours is the first adjoint-free analysis that attempts to close the gap between theory and practice.
\end{abstract}


\section{Introduction}
Let $\mathcal A: \mathcal{H} \to \mathcal{H}'$ be an operator between Hilbert spaces. Suppose that one can only access $\mathcal A$ via the forward and adjoint queries $f \mapsto \mathcal Af$ and $g \mapsto \mathcal A^\ast g$, where $f \in \mathcal{H}$ and $g \in \mathcal{H}'$ are inputs. We consider the problem of approximating $\mathcal A$ efficiently from data, using as few inputs $\{f_i\}_{i = 1}^N$ as possible. In this paper, we ask the following question:\\
\centerline{
    Is it possible to recover $\mathcal A$ when one can only query $\mathcal A$, and not $\mathcal A^\ast$?
}
As a toy example, consider the discrete problem of recovering an $N \times N$ rank-one matrix $A = uv^\top$, where $u, v \in \R^{N}$. The randomized SVD~\citep{halko2011finding} and generalized Nystr\"om~\citep{nakatsukasa2020fast,tropp2017practical} methods recover $A$ in just two queries: one with $A$, and one with $A^\top$. However, if one cannot query $A^\top$, one can only observe $uv^\top x = (v^\top x) u$. Therefore, to recover $A$, one needs $N$ matrix-vector products with $A$. Thus, in general, the action of the adjoint is essential to efficient low-rank matrix recovery.

The situation is more complicated for other classes of structured matrices. Consider the recovery of an $N \times N$ Toeplitz matrix $T$ from matrix-vector products. Unlike in the low-rank case, this can be done using just two matrix-vector products with $T$: $Te_1$ and $Te_N$, where $e_i \in \R^n$ is the $i$th elementary basis vector. In this case, access to the action of $T^\top$ is not required, even though $T$ is not symmetric. These examples suggest that depending on what prior information is known about the matrix, the adjoint may or may not be needed in a recovery algorithm~\citep{halikias2022matrix}.

The infinite-dimensional generalization of the matrix recovery problem arises naturally in operator learning~\citep{boulle2023mathematical}. Learning mappings between function spaces has widespread applications in science and engineering, as one can use data to either efficiently approximate existing scientific models or even discover new ones entirely. Moreover, just as low-rank matrices arise naturally in data science~\citep{townsend2019bigdata}, operators that occur in physics also have known mathematical properties. Thus, to be as efficient and accurate as possible, operator learning techniques seek to exploit prior knowledge about the operator.

Scientists across many disciplines use machine learning techniques to discover dynamical systems and partial differential equations (PDEs)  from data. These methods also benefit by leveraging prior information about the underlying physical operators. For example, a new interdisciplinary approach leverages machine learning and physical laws to uncover PDEs from experimental data or numerical simulations~\citep{karniadakis2021physics}. In particular, this approach uses input-output data pairs to approximate a PDE's solution operator, which maps source terms to their corresponding solutions. Within this context, neural operators, such as the Fourier Neural Operator (FNO)~\citep{li2020fourier,kovachki2023neural} and the Deep Operator Network (DeepONet)~\citep{lu2021learning}, have been introduced. These are extensions of traditional neural networks designed to learn mappings between infinite-dimensional function spaces. Their application to PDEs has shown significant promise, allowing them to act as rapid solvers. Once trained, neural operators can be seamlessly integrated into optimization loops for parameter estimation or used to make predictions about previously unseen data.

In the context of linear PDEs, the adjoint operator is essentially a dual operator that can sometimes be interpreted as the operator that arises when changing the direction of time or reversing the direction of space in the original PDE. The adjoint operator arises frequently in linear sensitivity analysis. In practice, acquiring data from the adjoint operator can be impossible when the underlying PDE is unknown. From a theoretical point-of-view, previous studies mostly focused on learning the solution operator of self-adjoint elliptic PDEs~\citep{boulle2021learning,boulle2023elliptic} in divergence form defined as
\begin{equation} \label{eq_elliptic}
    L u\coloneqq -\div(\bm{A}(x)\nabla u) = f, \quad x\in \Omega\subset\R^d,
\end{equation}
where the coefficient matrix $\bm{A}$ is symmetric and satisfies the uniform ellipticity condition. However, non-self-adjoint PDEs arise naturally when considering time-dependent problems such as the heat equation or advection-diffusion equation, i.e., $L u = -\div(\bm{A}(x) \grad u) + \bm{c}(x)\cdot \grad u$. In this context, it seems essential to require a solver for the adjoint to obtain information about the left and right singular functions of the PDE~\citep{boulle2022learning}. Curiously, there is a lack of emphasis in practical works on the need for adjoint equation solvers, as many methods seem to succeed without them.

This paper aims to bridge the gap between the theoretical requirement for the adjoint in PDE learning and its omission in practice by providing theoretical guarantees on operator learning in the adjoint-free setting. Our goal is to understand when and why neural network models can recover operators without access to data about the adjoint operator. To this end, we provide a thorough characterization of various contexts where one can leverage additional assumptions to quantify the accuracy of the adjoint-free reconstruction.

In the finite-dimensional setting of low-rank matrix recovery, we prove that the quality of the reconstructed matrix is fundamentally limited without access to the adjoint. However, we show that the quality of the approximation improves when we have more information about the left and right singular vectors. This suggests that no clever technique from linear algebra can be leveraged in the analogous adjoint-free operator learning problem unless we have prior information. 

\renewcommand{\algorithmicrequire}{\textbf{Input:}}
\renewcommand{\algorithmicensure}{\textbf{Output:}}
\begin{algorithm}[htbp]
\caption{Adjoint-free approximation algorithm}\label{alg_adjoint}
\begin{algorithmic}[1]
\Require Bounded linear operator $A:\mathcal{H}\to \mathcal{H}'$, self-adjoint operator $L:D(L)\subset \mathcal{H}\to\mathcal{H}$ such that $\Range(A^*) \subset D(L)$, integer $n\geq 1$.
\State Compute the first $n$ eigenfunctions $\varphi_1,\ldots,\varphi_n$ of $L$ and eigenvalues $\lambda_1,\ldots,\lambda_n$, with $|\lambda_1|\leq |\lambda_2|\leq \cdots |\lambda_n|$.
\State Sample the operator $A$ $n$ times at the eigenfunctions of $L$ to obtain \[u_1=A(\varphi_1),\quad \ldots,\quad u_n=A(\varphi_n).\]
\State Define the rank-$n$ projected operator $A P_n:\mathcal{H}\to \mathcal{H}'$ as
\[A P_n(f) \coloneqq \sum_{k=1}^n u_k\langle \varphi_k,\ f\rangle_{\mcal{H}} = \sum_{k=1}^n A(\varphi_k)\langle \varphi_k,\ f\rangle_{\mcal{H}}, \quad f\in \H\]
\Ensure Approximation $A P_n$ of $A$ satisfying
\[\| A - A P_n \|_{\mcal{H}\to\mcal{H}'} \leq \frac{1}{| \lambda_{n+1} |} \| L A^* \|_{\mcal{H}'\to\mcal{H}}.\]
\end{algorithmic}
\end{algorithm}

To approximate a non-self-adjoint compact operator $A$ without the adjoint, we exploit favorable properties of a carefully chosen preconditioner $L$ and reduce the problem to bounding $\|LA^\ast\|$. In this case, prior knowledge of $A$ is used to select a suitable $L$ guaranteeing the approximation's quality. This analysis gives a simple algorithm (see \cref{alg_adjoint}) for approximating $A$ by projecting it on the eigenfunctions of $L$. In the particular case where we seek to learn the solution operator of a uniformly elliptic PDE, we show that the Laplace-Beltrami operator can be used as a preconditioner, and derive explicit bounds on $\|LA^\ast\|$ by exploiting elliptic regularity. A general method using the Laplacian preconditioner is presented in \cref{sec:extension_method} (see \cref{alg_extension}). We perform numerical experiments showing that the convergence rate of our approximation is in close agreement with the theoretical predictions. Finally, we analyze our bound for solution operators of elliptic PDEs perturbed away from self-adjointness by lower-order terms. The linear degradation in performance with increasing non-self-adjointness predicted by our analysis is in agreement with deep learning experiments performed using standard operator learning methods. The close agreement between our theoretical and numerical results suggests that our framework explains the success of adjoint-free deep learning models for PDEs. Moreover, our work highlights operator preconditioning as a tool that can further improve performance in the adjoint-free setting.

\subsection{Related Works}

The predominant focus of theoretical research in operator learning is approximation theory results. \cite{chen1995universal} and \cite{lu2021learning} generalized the universal approximation theorem for neural networks~\citep{cybenko1989approximation} to shallow and deep neural operators. Over the past few years, significant progress has been made to derive approximation error bounds for neural operator techniques such as Fourier Neural Operators~\citep{kovachki2021universal,kovachki2023neural,lanthaler2021error} and DeepONets~\citep{lanthaler2021error,lu2021learning,schwab2023deep}. These results show that neural operators can approximate a large class of operators, and they characterize approximation error in terms of the network's width and depth.

Aside from approximation theory, other approaches aim to derive sample complexity bounds for solution operators associated with elliptic PDEs, i.e., determining the size of the training dataset needed to achieve a target error. These methods exploit prior knowledge of the structure of the solution operator, such as sparsity patterns~\citep{schafer2021sparse} or the hierarchical low-rank structure of the Green's function~\citep{boulle2021learning,boulle2023elliptic}. Convergence rates for more general linear self-adjoint operators have been derived by \cite{de2021convergence}, whose analysis assumes that the target operator is diagonalizable in a known basis. To our knowledge, the recent sample complexity analysis of the solution operator of parabolic (time-dependent) PDEs is the only theoretical extension to non-self-adjoint operators~\citep{boulle2022learning}. However, the proof technique assumes one can evaluate the adjoint operator, which is unrealistic in many applications.

On the practical side, many studies are available in the literature that successfully apply neural operators to a wide range of physical problems. For example, \cite{wang2021learning} extend DeepONets to incorporate prior knowledge of the PDE and consider applications in parametric differential equations such as diffusion-reaction equations. Then, \cite{lu2022comprehensive} compare Fourier Neural Operators and DeepONets on fluid dynamics benchmarks (Darcy flow and Navier--Stokes equations) and report relative testing errors of $1\%-5\%$. \cite{wen2022u} combine the popular U-NET architecture~\citep{ronneberger2015u} with FNO for solving multiphase flow problems in geosciences. Finally, \cite{goswami2022physics} applied DeepONets to predict crack locations in materials. However, to the best of our knowledge, none of the existing works in the literature studied or evaluated the impact of the non-self-adjointness of the operator on the performance of the model.

\subsection{Summary of Contributions}
We address the intriguing question of how to recover non-self-adjoint operators from data without accessing the action of the adjoint operator. By providing the first adjoint-free analysis, we attempt to close the existing gap between theoretical understanding and practical applications (see~\cref{sec_motivation}). We have three main contributions:

\paragraph{Limits of adjoint-free low-rank matrix recovery.} We start in the fundamental setting of recovering a low-rank matrix by querying the map $x\mapsto Ax$ but without access to $x\mapsto A^*x$. We show that querying $x\mapsto A^*x$ is essential for recovering $A$ and prove rigorous guarantees on the quality of the reconstruction in terms of how close $A$ is to a symmetric matrix (see~\cref{thm:lower_bound_on_Omega_eps,thm:upper_bound_on_Omega_eps}). Thus, we conclude that without prior knowledge of the properties of the adjoint, one must have access to its action.

\paragraph{An adjoint-free operator learning approach.} To provide an operator learning approach that does not need access to the adjoint, we exploit regularity results from PDE theory to estimate the range of the adjoint of the solution operator.  This allows us to prove the first guarantees on the accuracy of adjoint-free approximations (see~\cref{thm:operator_projection}). Our key insight is to leverage the favorable properties of a prior self-adjoint operator, such as the Laplace--Beltrami operator, to use as an operator preconditioner in the approximation problem. In particular, we query the action of the solution operator on the eigenfunctions of the prior self-adjoint operator, yielding an approximation with an error that decays at a rate determined by the eigenvalues of the prior. This is remarkable because common operator learning techniques (see~\cref{fig_exp}(a)) always seem to plateau; yet, we construct a simple algorithm that provably converges.

\paragraph{The effect of non-self-adjointness on sample complexity.}
We derive a sample complexity bound for our algorithm when applied to second-order uniformly-elliptic PDEs that are perturbed away from self-adjointness by lower-order terms. We show that for small perturbations, our bound on the approximation error grows linearly with the size of the perturbation (see~\cref{thm_approx_L2}), and we conjecture that this linear growth continues for large perturbations as well. This aspect of the error growth is also present in common operator learning techniques, as our numerical experiments illustrate (see~\cref{fig_exp}(d)). With respect to our operator learning algorithm, this means that the number of samples required to achieve a fixed error tolerance grows algebraically with the perturbation size.

\subsection{Organization of the Paper}

The paper is organized as follows. We begin in \cref{sec_motivation} with motivational examples for analyzing the sample complexity of non-self-adjoint operator learning. Then, in subsequent sections,  we gradually strengthen the assumptions about the operator we wish to learn and analyze the quality of our reconstruction in each case. In~\cref{sec:LowRankRecovery}, we consider the finite-dimensional case of a low-rank operator recovery problem. Given additional information about how close an unknown low-rank matrix is to symmetric, we provide lower and upper bounds (see~\cref{thm:lower_bound_on_Omega_eps,thm:upper_bound_on_Omega_eps}) on the size of the set of possible matrices satisfying given sketching constraints. In \cref{sec:FourierSampling}, we consider the recovery problem for general compact operators with prior information encoded by preconditioners (see~\cref{thm:operator_projection}). We prove that operators with Sobolev regularity properties, such as elliptic PDEs, can be successfully approximated by projection onto Fourier bases (see~\cref{thm:domain_with_boundary_case}), and perform numerical experiments to confirm the predicted rate of convergence. Finally, in \cref{sec:Perturbed}, we consider a concrete operator learning problem of approximating the Green's function of a 3D elliptic PDE with lower order perturbations and derive a sample complexity bound for reconstructing the solution operator (see~\cref{sec_approx_H,thm_approx_L2}). We then conclude with discussions and remarks in \cref{sec_conclusions}.

\section{Motivational Examples of Non-Self-Adjoint Operator Learning} \label{sec_motivation}

As a first motivational example for the analysis of non-self-adjoint operators, we consider the problem of estimating the sample complexity of learning parabolic PDEs (generalizing the heat equation) in the following form:
\[\mathcal{P}u\coloneqq u_t-\div(\bm{A}(x,t)\nabla u)=f(x,t),\quad x\in\Omega,\, t\in[0,T],\quad 0<T<\infty,\]
where $\Omega\subset\R^d$ is a bounded spatial domain with Lipschitz smooth boundary and $A(x,t)\in \R^{d\times d}$ is a symmetric positive definite matrix with bounded coefficient functions satisfying the uniform parabolicity condition. Here, we are interested in estimating the number of training pairs $\{(f,u)\}$ needed to learn the solution operator associated with $\mathcal{P}$, i.e., the Green's function~\citep{evans10}, to within a target tolerance $\epsilon>0$. \cite{boulle2022learning} construct an algorithm that provably converges to the solution operator at an algebraic rate with respect to the number of training pairs. However, a key assumption required to approximate the solution operator is that one can evaluate the adjoint $\mathcal{P}^*$ of the parabolic operator defined as
\[\mathcal{P}^*u = -u_t-\div(\bm{A}(x,t)^\top \nabla u).\]
In this work, we ask whether the requirement for the adjoint is an essential assumption and why neural operators do not require the adjoint in practical applications. A second example emerges from the stationary convection-diffusion equation with variable coefficients in the form:
\begin{equation} \label{eq_conv_diff_2}
    \L u\coloneqq -\div (\bm{A}(x)\nabla u)+\bm{c}\cdot \nabla u,\quad x\in \Omega\subset \R^d,
\end{equation}
where the lower order coefficient vector $\bm{c}$ contains functions in $L^p(\Omega)$ for some $p>d$~\citep{kim2019green}. Here, one can interpret $\L$ as a perturbation of the self-adjoint partial differential operator $L$ defined in \cref{eq_elliptic}. In particular, the magnitude of $\bm{c}$ influences the difference between the solution operator and its adjoint, and our ability to approximate it from training pairs of source terms and solutions (see~\cref{sec:Perturbed}).

\begin{figure}[htbp]
    \centering
    \begin{overpic}[width=\textwidth]{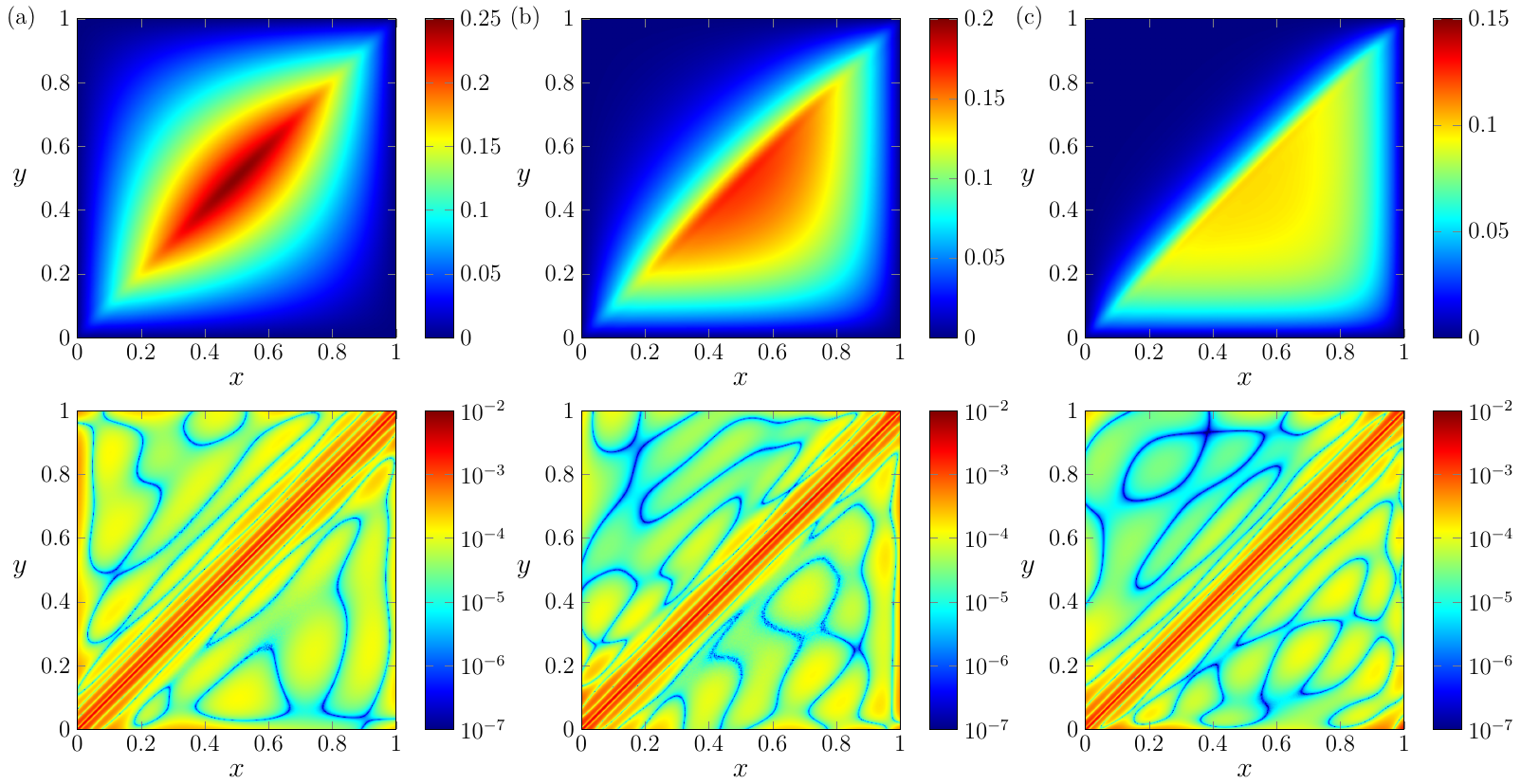}
    \end{overpic}
    \caption{Green's functions learned by a rational neural network (top row) along with the absolute error with the exact Green's function (bottom row) for the stationary convection-diffusion equation, with coefficients (a) $c=0$, (b) $c=5$, and (c) $c=10$.}
    \label{fig_error}
\end{figure}

We perform a deep learning experiment to approximate Green's function associated with the one-dimensional stationary convection-diffusion equation with homogeneous Dirichlet boundary conditions on $\Omega=[0,1]$:
\begin{equation} \label{eq_conv_diff}
    -\frac{d^2 u}{d x^2} + c\frac{d u}{d x}=f,\quad u(0)=u(1)=0,\quad x\in [0,1].
\end{equation}
We employ a rational neural network~\citep{boulle2020rational} to approximate the Green's function associated with \cref{eq_conv_diff} using the Green's function learning technique introduced by~\cite{boulle2021data}. We sample $25$ random functions from a Gaussian process with squared-exponential kernel and length-scale parameter $\ell=0.03$ and solve \cref{eq_conv_diff} using a Chebyshev spectral collocation method implemented in the Chebfun software system~\citep{driscoll2014chebfun}. The source terms $f$ and solutions $u$ are then sampled on a uniform grid with $200$ points, and the neural network is trained in the TensorFlow library~\citep{tensorflow2015} using a combination of Adam~\citep{kingma2015adam} and L-BFGS~\citep{byrd1995limited} optimization algorithms. The learned Green's function is then evaluated at a higher resolution on a $400\times 400$ grid and compared with the analytical expression for the exact Green's function given by:
\[G(x,y) =
    \frac{(1-e^{c(x-1)})(1-e^{-cy})}{c(1-e^{-c})}H(x-y)+
    \frac{(-1+e^{cx})(e^{-cy}-e^{-c})}{c(1-e^{-c})}H(y-x),
\]
where $H$ is the Heaviside step function.

\begin{figure}[htbp]
    \centering
    \begin{overpic}[width=0.9\textwidth]{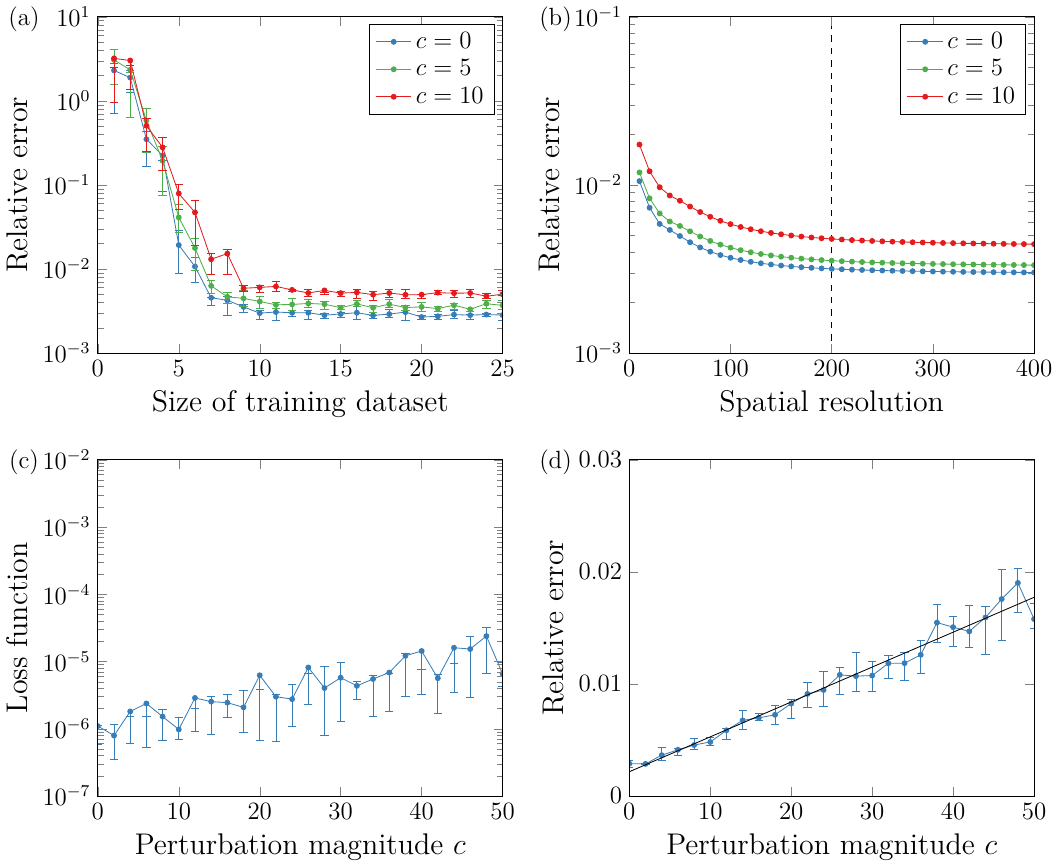}
    \end{overpic}
    \caption{(a) Relative errors for learning the Green's function of the advection-diffusion equation. The graph displays the mean error over ten runs, along with error bars representing the first and third quartiles. (b) Relative errors of the Green's function after training using input-output pairs sampled on a grid with resolution $s=200$ (dashed line) and evaluated at lower and higher resolutions. (c)-(d) Evolution of the loss function after training and relative error as the magnitude of the perturbation increases. The black line in (d) represents the linear least squares approximation and achieves $R^2=0.8$.}
    \label{fig_exp}
\end{figure}

We report in~\cref{fig_error}(a-c) the learned Green's functions of the stationary convection-diffusion equation \eqref{eq_conv_diff} with respective convection parameters $c=0$, $c=5$, and $c=10$, along with the absolute error with the exact Green's function in the bottom row. We observe that even though the difference between the Green's function and its adjoint increases between $c=0$ and $c=10$, the rational neural network can approximate the Green's function within $0.3-0.5\%$ relative error. Additionally, the approximation errors are mainly located around the diagonal $x=y$ of the domain $[0,1]^2$ (see the bottom row in \cref{fig_error}), where the Green's function has derivative discontinuity.

In \cref{fig_exp}, we study the dependence of the error on the size of the training dataset, the spatial grid resolution, and the convection coefficient $c$. To do this, we approximate the Green's function associated with \cref{eq_conv_diff} using $25$ training pairs sampled on a grid with resolution $s=200$ and report the relative error when evaluating the Green's function at different resolutions from $s=10$ to $s=400$. Similar to the Fourier neural operator~\citep{kovachki2023neural}, the rational neural network is capable of zero-shot super-resolution, even when learning highly non-self-adjoint operators, as the approximation error does not degrade when testing the network at a higher resolution.

The relative error plotted in \cref{fig_exp}(a) for different convection coefficients shows two distinct regimes as the size of the training dataset is increased.
There is an initial exponential decay of the error up to $10$ training pairs, followed by a stagnation at small relative error.
These results agree with previous experiments performed in~\citep{boulle2023elliptic} and suggest that one can approximate non-self-adjoint operators with few training data using deep learning. However, our experiments show that the plateau in relative error increases with the magnitude of the convection coefficient $c$, suggesting that there is a systematic component of the error that increases with the non-self-adjointness of the operator we seek to learn.

To study the systematic error introduced by non-self-adjoiness, we progressively increase the magnitude of the perturbation in \cref{eq_conv_diff} and report the corresponding loss function after training and relative errors for the learned Green's function in \cref{fig_exp}(c)-(d). Using a linear least squares regression ($R^2=0.78$), we observe that the relative error increases linearly with the magnitude of the perturbation. At the same time, the loss function after training remains relatively small between $10^{-6}$ and $10^{-5}$. The discrepancy between the magnitude of the loss function and the relative error is because the loss function is a relative mean-squared error, while the error reported in \cref{fig_exp}(d) is measured as a relative $L^2$-error between the exact and learned Green's functions. Finally, we observe in \cref{fig_exp}(d) that the variance of the errors also increases with the perturbation magnitude. These numerical experiments motivate our theoretical analysis of learning non-self-adjoint solution operators associated with PDEs in the rest of the paper. Moreover, these numerical results lead us to introduce an adjoint-free operator learning method that provably converges with the size of the training dataset, and therefore does not suffer from the plateau observed in \cref{fig_exp}(a).

\section{Reconstruction of Low-Rank Matrices under Sketching Constraints}\label{sec:LowRankRecovery}

In this section, we analyze a finite-dimensional variant of the main problem. Existing sample complexity bounds for learning Green's functions of linear PDEs, which consider self-adjoint elliptic PDEs~\citep{boulle2021learning,boulle2023elliptic} or parabolic PDEs~\citep{boulle2022learning}, assume access to the adjoint operator. The proofs of these results exploit randomized numerical linear algebra techniques, such as the randomized singular value decomposition (SVD)~\citep{halko2011finding,martinsson2020randomized,boulle2022generalization}, to construct low-rank approximants of the Green's function on well-separated subdomains. This motivates our investigation into the analogous problem of adjoint-free recovery of low-rank matrices from matrix-vector products.

The randomized SVD is an algorithm that computes a near-best low-rank approximant to a matrix $F\in \R^{n\times n}$ from matrix-vector products with a random input matrix $X\in \R^{n\times s}$ using a two-stage procedure. First, one sketches the matrix $F$ at $X$ to obtain $Y=FX$ and orthonormalizes $Y$ to form a basis $Q$, which approximates the range of $F$. Then, one constructs the approximant $\hat{F} = QQ^*F = Q(F^*Q)^*$ by sketching the adjoint, $F^*$, of $F$. In a landmark paper, \cite{halko2011finding} showed that the approximant $\hat{F}$ is a near-best approximant to $F$ with high probability.

If $F$ is too large to be stored or given as a streaming model~\citep{muthukrishnan2005data,clarkson2009numerical,woodruff2014sketching}, $F = H_1+H_2+H_3+\cdots$, one might not be able to view the matrix twice as in the randomized SVD~\citep{martinsson2020randomized}. Several single-view algorithms have been proposed to compute an approximate SVD of $F$, which visit the matrix only once, such as the Nystr\"om method~\citep{nystrom1930praktische,gittens2013topics,li2017algorithm,tropp2017practical}. However, to our knowledge, every low-rank approximation algorithm based on sketching requires access to $F^\ast$~\citep{martinsson2020randomized, halikias2022matrix}. This leads to a natural question: is there an algorithm that constructs a low-rank approximant to $F$ without its adjoint?

Recently, this question was answered negatively. It was proven that there are infinitely many rank-$k$ matrices $F$ satisfying the same matrix-vector products $FX = Y$ and $F^\top W = Z$, where $X \in \R^{n \times k_1}, W \in \R^{n \times k_2}$, if $\min(k_1, k_2) < k$ and $\max(k_1, k_2) < n$ \citep{halikias2022matrix}. This result extends to complex-valued matrices, so one needs $k$ queries to $F$ and $F^\ast$ each for the matrix recovery problem to have a unique solution. Even if one has as many as $n-1$ matrix-vector products with $F$ and none with $F^\ast$, $F$ is not uniquely determined.  Thus, this section considers the space of these infinitely many possible rank-$k$ matrices when the recovery problem is underspecified and does not have a unique solution.

We aim to understand how close one can get to recovering an unknown low-rank matrix $F\in \C^{n\times n}$ from matrix-vector products with an input matrix $X\in \C^{n\times s}$, i.e., without access to its adjoint. It has already been shown that when one does not have access to $F^\ast$, the possible row spaces of a matrix satisfying the same sketch constraints as $F$ can be arbitrarily far apart in the Riemannian metric on the Grassmannian manifold~\citep{otto2022model}. Thus, we assume some additional structure on our space of possible matrices, using the following notion of near-symmetry.
\begin{definition}[Near-symmetry] \label{def_near_sym}
    Let $F\in \C^{n\times n}$ be a rank-$k$ matrix with singular value decomposition $F=U_FS_FV_F^\ast$. We say that $F$ is $\delta$-near-symmetric if its left and right singular subspaces are $\delta$-close, i.e., there exists a $k \times k$ orthogonal matrix $Q$ such that
    \[
        \|U_F^\ast V_F-Q\|_2\leq \delta.
    \]
\end{definition}
We show that one cannot recover an accurate low-rank approximant to $F$ unless $F$ is near-symmetric. This analysis indicates that the adjoint is essential for low-rank recovery algorithms.

\subsection{Reconstruction of Near-Symmetric Matrices}

We consider an unknown rank-$k$ matrix $F \in M_n(\C)$ with singular value decomposition $F = U_F S_F V_F^\ast$ and aim to construct an approximant $A$ to $F$ satisfying the sketch constraint $AX=FX$, where the test matrix $X \in \C^{n \times s}$ has linearly independent columns and $\rank(FX) = k$. By construction, we have $k \leq s\leq n$. Almost every matrix $X \in \C^{n \times s}$ with respect to the Lebesgue measure satisfies the condition $\rank(FX) = k$~\citep[Lem.~2.4]{otto2022model}, meaning that the queries almost surely reveal the rank of $F$ and its range. In randomized numerical linear algebra, the test matrix $X$ is typically chosen to be a random matrix following a standard Gaussian distribution~\citep{martinsson2020randomized}, but other random embeddings, such as subsampled trigonometric transforms~\citep{woolfe2008fast} or coordinate samplings~\citep{williams2000using,tropp2011improved,kumar2012sampling,gittens2013topics}, may also be used.

We assume that $F$ is $\delta$-near-symmetric (see~\cref{def_near_sym}), but we only have access to partial information regarding the symmetry of $F$, namely that $F$ is $\epsilon$-near-symmetric for some $\epsilon \geq \delta$.
To quantify the resulting uncertainty about $F$, we study the set of possible matrices one could recover given this prior knowledge.
We denote this set
\begin{equation} \label{def_set_omega}
    \Omega_{F, X}^\epsilon = \{ A \in M_n(\C)\colon \rank(A) = k,\,  AX = FX,\, \exists Q\in O(k),\, \| U_A^\ast V_A - Q \|_2 \leq \epsilon\},
\end{equation}
where $A=U_A S_A V_A^\ast$ is the singular value decomposition of $A$, $O(k)$ is the group of $k\times k$ orthogonal matrices, and $\|\cdot\|_2$ denotes the spectral norm. This set might be nonempty even when $\epsilon < \delta$, but to ensure that $F \in \Omega_{F,X}^{\epsilon}$, we must have $\epsilon \geq \min_{Q\in O(k)} \| U_F^* V_F - Q \|_2$. The minimum exists because $O(k)$ is compact. 

\begin{remark}[Low-rank recovery algorithms and $\Omega_{F, X}^\epsilon$]
Given some tolerance $\epsilon$, $\Omega_{F,X}^\epsilon$ is the set of $\epsilon$-near-symmetric matrices that can be returned by any low-rank recovery algorithm when approximating $F$, such as the randomized SVD~\citep{halko2011finding,martinsson2020randomized} or the Nystr\"om method~\citep{nystrom1930praktische}. One can find a symmetric approximation in the set using Nystr\"om method by querying $A$ in place of $A^*$.
\end{remark}

The size of $\Omega_{F,X}^\epsilon$ is measured by its diameter in the spectral norm and determines the maximum accuracy of any reasonable reconstruction. If the diameter is large, one cannot estimate $F$ accurately, as one cannot distinguish between any candidate matrix in $\Omega_{F, X}^\epsilon$. This is because any matrix in $\Omega_{F, X}^\epsilon$ satisfies the sketching constraint and is near-symmetric. On the other hand, a small diameter guarantees the fidelity of the reconstruction. We aim to bound the size of $\Omega_{F, X}^\epsilon$, i.e., determine how far apart any two matrices in $\Omega_{F, X}^\epsilon$ can be from each other, with respect to $\epsilon$, which measures our prior knowledge of $F$'s symmetry. We first provide an upper bound on the diameter of the set $\Omega_{F, X}^\epsilon$.

The upper bound relies on a preliminary lemma. \cref{lem:V_minus_UQ}  provides an orthogonal change of basis, bringing a matrix $U$ with orthonormal columns close to another matrix $V$ in the spectral norm sense. The difference is bounded by the proximity of $U^* V$ to an orthogonal matrix. For the rest of this section, $\sigma_{\max}$ and $\sigma_{\min}$ respectively denote the largest and smallest nonzero singular values of a matrix.

\begin{lemma} \label{lem:V_minus_UQ}
    Let $U$ and $V$ be two $n\times k$ matrices with orthonormal columns and $U^\ast V = Q_l \Sigma Q_r^\ast\in \C^{k\times k}$ be a complete SVD. Then, the orthonormal matrix $Q_0 = Q_l Q_r^\ast$ satisfies
    \[
        \| V - U Q_0 \|_2^2
        = 2 \left(\min_{Q:\ Q^\top Q = I} \| Q - U^\ast V \|_2\right)
        = 2\left(1 - \sigma_{\min}(U^\ast V)\right).
    \]
\end{lemma}

\begin{proof}
    We first consider the SVD of the matrix $U^*V$ as $U^*V=Q_l\Sigma Q_r^*$ and introduce the orthonormal matrix $Q_0=Q_l Q_r^\top$. Let $d = \min_{Q:\ Q^\top Q = I} \| Q - U^*V\|_2$ denote the distance between the $k\times k$ matrix $U^*V$ and the set of orthogonal matrices.

    In general, from a result for unitarily invariant norms~\citep[Thm.~1]{Fan1955some}, if $A$ is a square matrix with $\| A \|_2 \leq 1$ and has complete SVD $A = U \Sigma V^*$, then
    \begin{equation}\label{eq:minimizer_of_A_minus_Q}
        \min_{Q:\ Q^\top Q = I} \| Q - A \|_2 = 1 - \sigma_{\min}(A)
    \end{equation}
    is achieved by $Q = U V^*$. Applying \cref{eq:minimizer_of_A_minus_Q} to $U^*V$ yields a characterization of $d$ with the smallest singular value of $U^*V$ as $\sigma_{\min}(U^* V) = 1 - d$. Let $x\in \R^k$ be a unit vector, then $v = V x$ and $u = U Q_0 x$ have norm $1$ because $U$ and $V$ have orthonormal columns. Moreover, $
        u^\top v
        = x^\top Q_0^\top U^\top V x
        = x^\top Q_r Q_l^\top Q_l \Sigma Q_r^\top x
        = x^\top Q_r \Sigma Q_r^\top x
        \geq \sigma_{\min}(U^\top V) = 1 - d$,
    with equality when $x = Q_r e_k$. Using similar triangles in the $(u,v)$-plane, we readily obtain
    \[
        \frac{1 - u^\top v}{\| u - v \|} = \frac{\| u - v \|}{2},
    \]
    which implies that$
        \| V x - U Q_0 x \|^2
        = \| v - u \|^2
        = 2 (1 - u^\top v) \leq 2 d,
    $
    with equality when $x = Q_r e_k$. Finally, taking the supremum over $x$,  $
        \| V - U Q_0 \|_2^2
        = \sup_{x: \ \| x \| = 1} \| V x - U Q_0 x \|^2
        = 2 d,
    $
    which concludes the proof.
\end{proof}

We are now ready to state \cref{thm:upper_bound_on_Omega_eps}, which provides an upper bound on the diameter of the set $\Omega_{F,X}^\epsilon$ defined in \cref{def_set_omega}.

\begin{theorem}[Upper bound] \label{thm:upper_bound_on_Omega_eps}

    Let $0\leq \delta \leq \epsilon < 1$, $F\in M_n(\C)$ be a $\delta$-near-symmetric rank-$k$ matrix, and $X \in \C^{n \times s}$ be a test matrix with $s\geq k$  orthonormal columns such that $\rank(FX) = k$. Let $F = U_0 \Sigma_0 V_0^*$ be a slim SVD of $F$, and introduce the constant $c = \sigma_{\max}(X^* V_0) / \sigma_{\min}(X^* V_0)^2$. If $c \big( \sqrt{2\epsilon} + \sqrt{2\delta} \big) < 1$, then
    \[
        \sup_{A, B \in\Omega_{F,X}^{\epsilon}} \| A - B \|_2
        \leq 4 \| F X \|_2 \left[ \frac{c^2 \big( \sqrt{2\epsilon} + \sqrt{2\delta} \big) }{1 - c \big( \sqrt{2\epsilon} + \sqrt{2\delta} \big) } \right].
    \]
\end{theorem}

\begin{proof}
    Recalling the $X$ has orthonormal columns and letting $\Phi = X X^* V_0$, we observe that
    \[
        F = F \Phi (V_0^* \Phi)^{-1} V_0^*.
    \]
    This is easily verified using the SVD of $F$. We note that $V_0^* \Phi = V_0^* X X^* V_0$ is invertible because $\rank(F X) = \rank(F) = k$.

    Suppose that $A \in \Omega_{F,X}^{\epsilon}$ and let $A = U_A \Sigma_A V_A^*$ be a slim SVD.
    Then $
        A = F \Phi (V_A^* \Phi)^{-1} V_A^*$,
    as one can verify using the SVD of $A$ and the identity $F \Phi = A \Phi$.
    For any invertible $k\times k$ matrix $Q$ (later on $Q$ will be orthogonal), $A = F \Phi (Q V_A^* \Phi)^{-1} Q V_A^*$. By the triangle inequality,
    \begin{equation} \label{eqn:triangle_estimate_for_A0_minus_A}
        \| F - A \|_2
        \leq \| F \Phi \|_2 \| (V_0^* \Phi)^{-1} \|_2 \| V_A Q^* - V_0 \|_2
        + \| F \Phi \|_2 \| (V_0^* \Phi)^{-1} - (Q V_A^* \Phi)^{-1} \|_2.
    \end{equation}
    We use a perturbation argument to show that $V_A^* \Phi$ is invertible, to choose $Q$, and to bound each term in the above equation.

    By \cref{lem:V_minus_UQ} and \cref{eq:minimizer_of_A_minus_Q}, there exist orthogonal matrices $Q_0$ and $Q_A$ satisfying $\| V_0 - U_0 Q_0 \| \leq \sqrt{2\delta}$ and $\| V_A - U_A Q_A \| \leq \sqrt{2\epsilon}$.
    Since $\Range(U_A) = \Range(U_0)$ there is an orthogonal matrix $\tilde{Q}$ such that $U_A Q_A = U_0 \tilde{Q}$.
    Letting $Q = Q_0 \tilde{Q}$ we obtain
    \begin{equation} \label{eq_diff_pert}
        \| V_A Q^* - V_0 \|_2 =
        \| V_A - V_0 Q \|_2
        \leq \| V_A - U_A Q_A \| + \| U_0 \tilde{Q} - V_0 Q_0 \tilde{Q} \|_2
        \leq \sqrt{2\epsilon} + \sqrt{2\delta}
    \end{equation}
    A classical perturbation bound for the difference between the inverse of two matrices $S$ and $T$ is~\citep[Ch.~1, Eq.~4.24]{Kato1980perturbation}
    \[\|S^{-1}-T^{-1}\|_2\leq \frac{\|S-T\|_2\|T^{-1}\|_2^2}{1-\|S-T\|_2\|T^{-1}\|_2}.\]
    Applying this with $S = Q V_A^* \Phi$ and $T = V_0^* \Phi$ gives
    \[
        \| (Q V_A^* \Phi)^{-1} - (V_0^* \Phi)^{-1} \|_2
        \leq \frac{\big(\sqrt{2\epsilon} + \sqrt{2\delta}\big)\| \Phi \|_2 \| (V_0^* \Phi)^{-1} \|_2^2}{1 - \big(\sqrt{2\epsilon} + \sqrt{2\delta}\big)\| \Phi \|_2 \| (V_0^* \Phi)^{-1} \|_2 },
    \]
    where
    \[\|(Q V_A^* \Phi) - (V_0^* \Phi)\|_2\leq \|V_A Q^* - V_0\|_2 \|\Phi\|_2\leq \big(\sqrt{2\epsilon} + \sqrt{2\delta}\big)\|\Phi\|_2,\]
    by \cref{eq_diff_pert}. A quick computation shows that
    \[
        \| \Phi \|_2 \| (V_0^* \Phi)^{-1} \|_2
        = \| X^* V_0 \|_2 \| (V_0^* X X^* V_0)^{-1} \|_2
        = \frac{\sigma_{\max}(X^* V_0)}{\sigma_{\min}(X^* V_0)^2}
        = c,
    \]
    and that $\| F \Phi \|_2 \leq \| F X \|_2 \| X^* V_0 \|_2 = \| F X \|_2 \| \Phi \|_2$. The condition that $c \big( \sqrt{2\epsilon} + \sqrt{2\delta} \big) < 1$ ensures that both $Q V_A^* \Phi$ and $V_A^* \Phi$ are invertible.
    Using these results in \cref{eqn:triangle_estimate_for_A0_minus_A} and collecting terms yields
    \[
        \| F - A \|_2
        \leq \| F X \|_2 \big( \sqrt{2\epsilon} + \sqrt{2\delta} \big) \left[ c + \frac{c^2}{1 - c \big(\sqrt{2\epsilon} + \sqrt{2\delta}\big) } \right]
        \leq 2 \| F X \|_2 \left[\frac{ c^2\big( \sqrt{2\epsilon} + \sqrt{2 \delta}\big) }{1 - c \big( \sqrt{2\epsilon} + \sqrt{2\delta}\big) }\right].
    \]
    Applying the triangle inequality $\| A - B \|_2 \leq \| F - A \|_2 + \| F - B \|_2$ with $A, B \in \Omega_{F,X}^{\epsilon}$ completes the proof.
\end{proof}

To lower bound the diameter of $\Omega_{F, X}^{\epsilon}$, we must show the existence of two matrices in the set that are at least some distance apart. We generate these matrices by perturbing $F$.
Our argument makes use of the gap between $\epsilon$ and $\delta$, reflecting the gap in our prior knowledge about the near-symmetry of $F$.

\begin{theorem}[Lower bound] \label{thm:lower_bound_on_Omega_eps}
    Let $0\leq \delta \leq \epsilon <1$, $F\in M_n(\C)$ be a $\delta$-near-symmetric rank-$k$ matrix, and $X \in \C^{n \times s}$ be a test matrix with $k \leq s < n$ orthonormal columns such that $\rank(FX) = k$.
    Then, the diameter of $\Omega_{F, X}^\epsilon$ is lower bounded as follows:
    \begin{equation} \label{eq_lower_bound}
        \sup_{A,B \in \Omega_{F,X}^{\epsilon}} \| A - B \|_2 \geq
        2 \left( \frac{\sigma_{\min}(F)^2}{\sigma_{\max}(F)} \right) \frac{\arccos(1-\epsilon) - \arccos(1-\delta)}{\pi/2 + \arccos(1-\epsilon) - \arccos(1-\delta)}.
    \end{equation}
\end{theorem}

\begin{proof}
    We begin the proof by selecting a matrix $E \in M_n(\C)$, satisfying $E X = 0$ and
    \begin{equation} \label{eqn:choice_of_E}
        \| E \|_2 = \left(\frac{\sigma_{\min}(F)}{\sigma_{\max}(F)}\right) \frac{\arccos(1-\epsilon) - \arccos(1-\delta)}{\pi/2 + \arccos(1-\epsilon) - \arccos(1-\delta)},
    \end{equation}
    which is well defined since $\epsilon \geq \delta$.
    The constraint $E X = 0$ is satisfied by choosing the rows of $E$ in $\Range(X)^{\perp}$, which is nontrivial because $s < n$. Letting $B = F(I + E) = U_B \Sigma_B V_B^\top$, we aim to show that $B\in \Omega_{F,X}^{\epsilon}$. First, we observe that $\Range(B) = \Range(F)$, that is $\Range(U_B) = \Range(U_F)$, because $\Range(B) \subset \Range(F)$ and $k = \rank(F) \geq \rank(B) \geq \rank(B X) = \rank(F X) = k$. Following \cref{eq:minimizer_of_A_minus_Q}, we must show that
    \[
        1 - \sigma_{\min}(U_B^\top V_B) \leq \epsilon.
    \]
    The identity $\sigma_i(U_B^\top V_B) = \cos(\theta_i(U_B, V_B))$, where $\theta_i(U_B, V_B)$ denotes the $i$th principal angle between subspaces $\Range(U_B)$ and $\Range(V_B)$ (see \cite{Bjorck1973numerical}), yields
    \[
        1 - \sigma_{\min}(U_B^\top V_B) = 1 - \cos(\theta_{\max}(U_B, V_B)),
    \]
    because $x\mapsto\cos(x)$ is a decreasing function over the interval $[0, \pi/2]$. Therefore, it suffices to show that $\theta_{\max}(U_B, V_B) \leq \arccos(1-\epsilon)$.
    Thanks to the main result of \cite{Qui2005unitarily}, the largest principal angle $\theta_{\max}$ is a unitarily-invariant metric on the Grassmannian consisting of $k$-dimensional subspaces of $\R^n$.
    In particular, it satisfies the triangle inequality:
    \[
        \theta_{\max}(U_B, V_B)
        = \theta_{\max}(U_F, V_B)
        \leq \theta_{\max}(U_F, V_F) + \theta_{\max}(V_F, V_B).
    \]
    Combining the assumption on $F$ and~\cref{eq:minimizer_of_A_minus_Q}, we have
    \[
        1 - \delta
        \leq \sigma_{\min}(U_F^\top V_F)
        = \cos(\theta_{\max}(U_F, V_F)) \implies
        \theta_{\max}(U_F, V_F)\leq \arccos(1-\delta),
    \]
    as $\theta\mapsto \arccos(\theta)$ is a decreasing function. Therefore,
    \[
        \theta_{\max}(U_B, V_B)
        \leq \arccos(1-\delta) + \theta_{\max}(V_F, V_B).
    \]
    Using Wedin's theorem~\citep{wedin1972perturbation}, we obtain
    \[
        \theta_{\max}(V_F, V_B)
        \leq \frac{\pi}{2} \sin(\theta_{\max}(V_F, V_B))
        \leq \frac{\pi}{2} \| \sin (\Theta(V_F, V_B)) \|_\F
        \leq \frac{(\pi/2) \sigma_{\max}(F) \|E\|_2}{\sigma_{\min}(F) - \sigma_{\max}(F)\|E\|_2},
    \]
    which means that
    \[
        \theta_{\max}(U_B, V_B)
        \leq \arccos(1-\delta) + \frac{(\pi/2) \sigma_{\max}(F) \|E\|_2}{\sigma_{\min}(F) - \sigma_{\max}(F)\|E\|_2}.
    \]
    Inserting the expression for $\| E \|_2$ given by \cref{eqn:choice_of_E} yields $\theta_{\max}(U_B, V_B) \leq \arccos(1-\epsilon)$, which shows that $B \in \Omega_{F,X}^{\epsilon}$. Since the same argument shows that $F(I - E) \in \Omega_{F, X}^\epsilon$, we obtain a lower bound on the diameter as follows:
    \[
        \diam(\Omega_{F,X}^{\epsilon})
        \geq 2 \left( \frac{\sigma_{\min}(F)^2}{\sigma_{\max}(F)} \right) \frac{\arccos(1-\epsilon) - \arccos(1-\delta)}{\pi/2 + \arccos(1-\epsilon) - \arccos(1-\delta)},
    \]
    which concludes the proof.
\end{proof}

\begin{remark}[Orthonormal test matrices]
    Almost every test matrix $X \in \C^{n \times s}$ (with respect to the Lebesgue measure) has linearly independent columns. Thus, these columns can be orthonormalized using the Gram--Schmidt process. Therefore, without loss of generality, we can assume that the test matrix in~\cref{thm:upper_bound_on_Omega_eps,thm:lower_bound_on_Omega_eps} has orthonormal columns. More formally, if $\tilde X  \in \C^{n \times s}$ is a matrix with linearly independent columns, recovering a matrix $A$ from the matrix-vector products $A \tilde X = Y$ is equivalent to recovering $AQ = YR^{\dagger} $, where $\tilde X = QR$ is a QR factorization. Thus, our results extend to the general matrix recovery model from matrix-vector products.
\end{remark}

\begin{remark}(Sharpness for symmetric $F$ and areas for improvement)
    There are limitations to our upper and lower bounds on the diameter of $\Omega_{F,X}^{\epsilon}$ that we hope will be resolved by future works.
    First, our upper bound becomes infinite as $\epsilon$ increases to a finite value, whereas our lower bound saturates as $\epsilon$ is increased.
    On the other hand, our lower bound vanishes when $\epsilon = \delta$, while the upper bound takes a positive value.
    Despite these issues, when $F$ is symmetric, i.e., when $\delta = 0$, our bounds yield
    \[
        \Theta(\sqrt{\epsilon})
        = \frac{2 c^{-1} \arccos(1-\epsilon)}{\pi/2 + \arccos(1-\epsilon)}
        \leq \sup_{A,B \in \Omega_{F,X}^{\epsilon}} \| A - B \|_2
        \leq \frac{4 \| F X \| c^2 \sqrt{2\epsilon}}{1 - c\sqrt{2\epsilon}}
        = \mcal{O}(\sqrt{\epsilon}),
    \]
    as $\epsilon \to 0$, meaning that they are asymptotically sharp in this regime.
    Sharpening our bounds when $F$ is asymmetric is particularly challenging because little is known about the properties of $\Omega_{F,X}^{\epsilon}$. To our knowledge, our work is the first to define such a set, and we hope that future works will investigate this set more thoroughly.
\end{remark}

The upper and lower bounds reveal that the uncertainty about $F$ given queries of its action is directly related to the uncertainty about the symmetry of its left and right singular subspaces. For example, our ability to recover a symmetric rank-$k$ matrix using $k \leq s < n$ queries is fundamentally limited by our prior knowledge about the proximity of $\Range(F)$ and $\Range(F^*)$ because there are many asymmetric matrices with the same rank that satisfy the same sketching constraints. As described above, generic $n\times s$ test matrices $X$ are capable of revealing the range of $F$, meaning that the uncertainty about $F$ really comes from a lack of prior knowledge about the range of $F^*$. This highlights why the action of the adjoint is essential for efficient matrix recovery without any prior information about an operator's symmetry, or more generally about the range of its adjoint.
In the following section we turn to the study of PDEs, where we show that regularity estimates provide useful prior information about the adjoint that can be leveraged to provide convergent adjoint-free approximation methods.

\section{Fourier Sampling for Differential Operators}\label{sec:FourierSampling}

This section shows how without querying the adjoint, one can construct finite-dimensional approximations of certain non-self-adjoint infinite-dimensional compact operators with error bounds. We develop our main results in a very general abstract setting before applying them to approximate solution operators for uniformly elliptic PDEs, among other operators with regularity (i.e., ``smoothing'') properties.
This leads to practical algorithms that can be applied in various settings. The main idea is to leverage known regularity properties of the adjoint operator along with guaranteed approximation properties of smooth functions in Fourier bases, or more generally, in eigenfunction bases of a suitably chosen self-adjoint operator, e.g.~the Laplace--Beltrami operator (LBO). This self-adjoint operator, therefore, serves as a prior to approximate the non-self-adjoint operator. We obtain finite-rank approximations with guaranteed error rates in the operator norm by querying the forward action of a non-self-adjoint operator with the leading eigenfunctions of the prior operator. 
Our approach is closely related to spectral Galerkin methods for PDEs \cite{Canuto2006spectral}, except that we project the solution operator from one side, rather than the PDE from both sides.
The key function approximation results we develop for these purposes extend standard Fourier approximation results \cite{Canuto1982approximation} to non-rectangular domains, and they extend results by \cite{Aflalo2013spectral, Aflalo2015optimality} to higher degrees of regularity and domains with boundaries.
Similar techniques have also been used by \cite{Friz1999smooth, Robinson2008topological} to study the ``thickness exponents'' of subsets of Sobolev spaces, arising as attractors for solutions of nonlinear spatiotemporal PDEs.

Consider a compact operator $A: \H \to \H'$ between Hilbert spaces $\H$ and $\H'$, and define $P_n: \H \to \H$ as the orthogonal projection onto the subspace spanned by a set of $n$ orthonormal vectors $\varphi_1, \ldots, \varphi_n\in \H $.
Letting $F_n : \R^n \to \H$ be the map given by
\[
    F_n : (x_1, \ldots, x_n) \mapsto \sum_{k=1}^n x_k \varphi_k,
\]
we obtain the operator $Y:\R^n \to \H$ given by $Y = A F_n$ by evaluating $A$ at $\varphi_1, \ldots, \varphi_n$. Our rank-$n$ approximation of the operator $A$ takes the form
\[
    A P_n = Y F_n^*.
\]
This approximation of $A$ can be formed without querying the adjoint operator $A^*$ and comes with error bounds provided we have some prior information about $\Range(A^*)$.

We select $\{\varphi_k\}_{k=1}^n$ to be the first $n$ eigenfunctions of an unbounded self-adjoint operator $L: D(L) \subset \H \to \H$, which serves as a prior operator for approximating $A$. Hence, the projection $P_n$ enjoys guaranteed approximation properties for functions in $D(L)$. Here, $D(L)$ is the domain of the operator $L$, satisfying
\[D(L)=\{f\in \H,\quad L f \in \H\}.\]
Examples of such operators are provided later in \cref{sec:manifold_without_boundary_case,sec:domain_with_boundary_case}. We begin with the following lemma, which can be viewed as an abstract generalization of a result by~\cite{Aflalo2013spectral}.
A similar approach has been used to obtain approximation results for orthogonal polynomials (cf. proof of Theorem~2.3 in \cite{Canuto1982approximation} and equation~5.4.11 in \cite{Canuto2006spectral}).

\begin{lemma} \label{lem:function_projection}
    Let $\H$ be a separable Hilbert space and $L: D(L)\subset \H \to \H$ be a self-adjoint operator whose domain $D(L)$, endowed with the graph norm, is compactly embedded in $\H$.
    $\H$ admits an orthonormal basis of eigenfunctions $\{ \varphi_k \}_{k=1}^{\infty}$ of $L$ with eigenvalues $\lambda_k \in \R$ ordered in increasing magnitude as $| \lambda_1 | \leq | \lambda_2 |\leq \cdots$, with $| \lambda_k | \to \infty$ as $k \to \infty$.
    Let $n\geq 1$ and $P_n : \H \to \H$ denote the orthogonal projection onto $\vspan\{ \varphi_k \}_{k=1}^n$.
    If $\lambda_{n+1} \neq 0$, then
    \begin{equation} \label{eqn:function_projection_bound}
        \| f - P_n f \|_{\mcal{H}} \leq \frac{1}{| \lambda_{n+1} |} \| L f \|_{\mcal{H}}, \quad f \in D(L),
    \end{equation}
    and equality is achieved with $f = \varphi_{n+1}$.
\end{lemma}

\begin{proof}
    For any $\lambda \in \C$ belonging to the resolvent set $\rho(L)$~\citep[Def.~1.16]{conway2019course}, we first show the resolvent operator $R_L(\lambda) = (\lambda I - L)^{-1}$ is compact, as it is a bounded operator into $D(L)$, which is compactly embedded in $\H$.
    Specifically, let $\imath_{D(L)}: D(L) \hookrightarrow \H$ denote the inclusion map, which is compact by assumption. Since $\Range(R_L(\lambda)) \subset D(L)$ we may define $\tilde{R}_L(\lambda) : \H \to D(L)$ so that $R_L(\lambda) = \imath_{D(L)} \circ \tilde{R}_L(\lambda)$.
    To prove that $R_L(\lambda)$ is compact, it suffices to show that $\tilde{R}_L(\lambda)$ is bounded. For any $f \in \H$, we have $\|f\|_{L}^2=\|f\|_{\mcal{H}}^2+\|L f\|_{\mcal{H}}^2$ by definition of the graph norm, which implies that
    \begin{align*}
        \| \tilde{R}_L(\lambda) f \|_{L}^2 &
        =  \| R_L(\lambda) f \|_{\mcal{H}}^2 + \| L R_L(\lambda) f \|_{\mcal{H}}^2                                                                                                                \\
                                           & \leq \| R_L(\lambda) f \|_{\mcal{H}}^2 + \big(\| \underbrace{L R_L(\lambda) f - \lambda f}_{-f} \|_{\mcal{H}} + |\lambda| \| f \|_{\mcal{H}} \big)^2 \\
                                           & \leq \left[ \| R_L(\lambda) \|_{\mcal{H}\to\mcal{H}}^2 + (1 + |\lambda|)^2 \right] \| f \|_{\mcal{H}}^2,
    \end{align*}
    meaning that $\tilde{R}_L(\lambda)$ is bounded, hence $R_L(\lambda)$ is compact.

    By \citep[Thm.~VIII.3]{Reed1980functional}, the imaginary unit $i$ lies in the resolvent set $\rho(L)$, and so $R_L(i)$ is compact. Following the proof of \citep[Thm.~VIII.4]{Reed1980functional}, $R_L(i)$ commutes with its adjoint $R_L(i)^* = (-iI - L^*)^{-1} = (-iI - L)^{-1} = R_L(-i)$ and is therefore a normal operator.
    By the spectral theorem for compact normal operators, $\H$ admits an orthonormal basis of eigenfunctions $\{ \varphi_k \}_{k=1}^{\infty}$ of $R_L(i)$ with eigenvalues $\mu_k\in\C$. Since $\Null R_L(i) = \{ 0 \}$, all of the $\mu_k$ are nonzero. The identity $(iI-L)^{-1} \varphi_k = \mu_k \varphi_k$ yields
    \[
        L \varphi_k = \left( \frac{i\mu_k - 1}{\mu_k} \right) \varphi_k,
    \]
    which shows that $\varphi_k$ is an eigenfunctions of $L$ with eigenvalue $\lambda_k = (i\mu_k - 1)/\mu_k$. Since $R_L(i)$ is compact, its eigenvalues satisfy $\mu_k \to 0$ as $k\to\infty$, meaning that $|\lambda_k| \to \infty$. Moreover, the eigenvalues $\lambda_k$ are real because $L$ is self-adjoint.

    Now consider a function $f \in D(L)$. By Parseval's theorem, we have
    \[
        \| L f \|_{\mcal{H}}^2
        = \sum_{k=1}^{\infty} | \langle \varphi_k, \ L f \rangle_{\mcal{H}} |^2
        = \sum_{k=1}^{\infty} | \lambda_k |^2 | \langle \varphi_k, \ f \rangle_{\mcal{H}}|^2.
    \]
    We follow an argument similar to \citep{Aflalo2013spectral} to obtain
    \begin{align*}
        \| L (f - P_n f) \|_{\mcal{H}}^2 & = \sum_{k=1}^{\infty} | \lambda_k |^2 | \langle \varphi_k, \ (I-P_n) f \rangle_{\mcal{H}} |^2 = \sum_{k=n+1}^{\infty} | \lambda_k |^2 | \langle \varphi_k, \ (I-P_n) f \rangle_{\mcal{H}} |^2 \\
                                         & \geq |\lambda_{n+1}|^2 \sum_{k=n+1}^{\infty} | \langle \varphi_k, \ (I-P_n) f \rangle_{\mcal{H}} |^2
        = |\lambda_{n+1}|^2 \| f - P_n f \|_{\mcal{H}}^2,
    \end{align*}
    and
    \[
        \| L (f - P_n f) \|_{\mcal{H}}^2 = \sum_{k=n+1}^{\infty} | \lambda_k |^2 | \langle \varphi_k, \ f \rangle_{\mcal{H}} |^2 \leq \| L f \|_{\mcal{H}}^2.
    \]
    Combining these inequalities yields \cref{eqn:function_projection_bound}.
\end{proof}

Thanks to the preliminary approximation result in~\cref{lem:function_projection}, we can derive an error bound between $A$ and its finite-rank approximation, $A P_n$, by choosing the operator $L$ so that $\Range(A^*) \subset D(L)$. In this case, the operators $P_n$ and $A P_n$ can be written as
\[P_n f = \sum_{k=1}^n \varphi_k\langle \varphi_k,\ f\rangle_{\mcal{H}}, \quad  A P_n f = \sum_{k=1}^n A\varphi_k\langle \varphi_k,\ f\rangle_{\mcal{H}}, \quad f\in \H.\]

\begin{theorem} \label{thm:operator_projection}
    Let $\H'$ be a Hilbert space. Under the same assumptions as \cref{lem:function_projection}, if $A: \H \to \H'$ is a bounded linear operator with $\Range(A^*) \subset D(L)$, then $LA^*$ is bounded. Moreover, for any $n\geq 1$ satisfying $\lambda_{n+1} \neq 0$, we have
    \begin{equation} \label{eqn:operator_projection_bound}
        \| A - A P_n \|_{\mcal{H}\to\mcal{H}'} \leq \frac{1}{| \lambda_{n+1} |} \| L A^* \|_{\mcal{H}'\to\mcal{H}},
    \end{equation}
    with respect to the induced norms of operators $\mcal{H}\to\mcal{H}'$ and $\mcal{H}'\to\mcal{H}$.
    Equality for all such $n$ is achieved by the compact self-adjoint operator $A = L^{\dagger} : \H \to \H$ defined by
    \begin{equation} \label{eqn:operator_achieving_equality_in_projection_bound}
        L^{\dagger}: f \mapsto \sum_{\substack{k=1\\\lambda_k \neq 0}}^\infty \frac{1}{\lambda_k} \varphi_k \langle \varphi_k,\ f \rangle_{\mcal{H}}.
    \end{equation}
\end{theorem}
\noindent

\begin{proof}
    First, $\Range(A^*)\subset D(L)$, so the operator $L A^*$ is well-defined  on $\H'$. To show that $L A^*$ is bounded, we show that $LA^*$ and the graph of $LA^*$ are closed. Let  $f_k \to f$ in $\H'$ and $L A^* f_k \to g$ in $\H$. By the continuity of $A^*$, $A^*f_k \to A^*f$, so $LA^*$ is closed. Additionally, the closedness of $L$ implies that $A^* f \in D(L)$ and $L A^* f = g$, so the graph of $L A^*$ is closed. Then, by the closed graph theorem, $L A^*$ is bounded. 
    
    Then, we have
    \begin{align*}
        \| A - A P_n \|_{\mcal{H}\to\mcal{H}'}
         & = \|(I-P_n)A^*\|_{\mcal{H}'\to\mcal{H}}
        =\sup_{\substack{f\in\H'                   \\ \| f \|_{\mcal{H}'} \leq 1 }} \| (I-P_n) A^* f \|_{\mcal{H}} \\
         & \leq \sup_{\substack{f\in\H'            \\ \| f \|_{\mcal{H}'} \leq 1 }} \frac{1}{| \lambda_{n+1} |} \| L A^* f \|_{\mcal{H}} = \frac{1}{| \lambda_{n+1} |} \| L A^* \|_{\mcal{H}'\to\mcal{H}},
    \end{align*}
    where the inequality follows from~\cref{lem:function_projection} and the final equality is due to the fact that $\Range(A^*)\subset D(L)$. This proves the desired bound.

    To show that the equality in the bound is achieved for the operator $L^{\dagger}$ defined by \cref{eqn:operator_achieving_equality_in_projection_bound}, we must show that $\Range\big((L^{\dagger})^*\big)\subset D(L)$ and that $L$ acts element-wise on the series defining $(L^{\dagger})^*$. We first establish that $L^{\dagger}$ is well-defined and compact and denote the operators defined by the partial sums in \cref{eqn:operator_achieving_equality_in_projection_bound} as
    \[
        A_n : f \mapsto \sum_{\substack{k=1 \\ \lambda_k \neq 0}}^n \lambda_k^{-1} \varphi_k \langle \varphi_k,\ f \rangle_{\mcal{H}}, \quad n\geq 1.
    \]
    Let $1\leq m \leq n$ be sufficiently large such that $\lambda_{m+1}\neq 0$ and $f\in \H$, we have
    \[
        \| A_n f - A_m f \|_{\mcal{H}}^2
        = \sum_{k=m+1}^n | \lambda_k |^{-2} | \langle \varphi_k,\ f \rangle_{\mcal{H}} |^2
        \leq | \lambda_{m+1} |^{-2} \sum_{k=m+1}^{\infty} | \langle \varphi_k,\ f \rangle |_{\mcal{H}}^2
        \leq | \lambda_{m+1} |^{-2} \| f \|_{\mcal{H}}^2.
    \]
    Hence, $\{ A_n f \}_{n=1}^{\infty}$ is a Cauchy sequence converging to $L^{\dagger} f$ in $\mcal{H}$ because
    $\vert \lambda_m \vert \to \infty$ as $m\to\infty$.
    Passing to the limit, we obtain
    \[
        \| L^{\dagger} f - A_m f \|_{\mcal{H}}^2
        \leq | \lambda_{m+1} |^{-2} \| f \|_{\mcal{H}}^2,
    \]
    meaning that $\| L^{\dagger} - A_m \|_{\mcal{H} \to \mcal{H}} \leq | \lambda_{m+1} |^{-1} \to 0$ as $m\to\infty$.
    Since $L^{\dagger}$ is the limit with respect to the operator norm of a sequence of finite rank operators, it follows from \citet[Corollary~6.2]{Brezis2010functional} that $L^{\dagger}$ is compact.
    Moreover, since each $A_n$ is self-adjoint, it follows that $L^{\dagger}$ is self-adjoint by continuity as
    \[
        \langle L^{\dagger} f,\ g \rangle_{\mcal{H}}
        = \lim_{n\to\infty} \langle A_n f,\ g \rangle_{\mcal{H}}
        = \lim_{n\to\infty} \langle f,\ A_n g \rangle_{\mcal{H}}
        = \langle f,\ L^{\dagger} g \rangle_{\mcal{H}}.
    \]

    Next, we show that $A_n f$ converges in the graph norm for $L$, i.e., in the norm defined by $\| g \|_{L}^2 := \| g \|_{\mcal{H}}^2 + \| L g \|_{\mcal{H}}^2$ for $g \in D(L)$. When $m \leq n$ is sufficiently large such that $\lambda_{m+1}\neq 0$, we have
    \[
        \| A_n f - A_m f \|_{L}^2
        = \sum_{k=m+1}^n \left( | \lambda_k |^{-2} + 1 \right) |\langle \varphi_k,\ f \rangle_{\mcal{H}} |^2
        \leq \left( | \lambda_{m+1} |^{-2} + 1 \right) \sum_{k=m+1}^{\infty} | \langle \varphi_k,\ f \rangle_{\mcal{H}} |^2,
    \]
    which converges to $0$ as $m\to\infty$, i.e., $A_n f$ is a Cauchy sequence in $D(L)$ with the graph norm. Since the graph of $L$ is closed, it follows that the limit $L^{\dagger} f$ is in $D(L)$ and $L A_n f \to L L^{\dagger} f$ in $\mcal{H}$ as $n\to\infty$. In other words, we have
    \[
        L L^{\dagger} f
        = \sum_{\substack{k=1\\ \lambda_k \neq 0}}^\infty \lambda_k^{-1} L \varphi_k \langle \varphi_k,\ f \rangle_{\mcal{H}}
        = \sum_{\substack{k=1\\ \lambda_k \neq 0}}^\infty \varphi_k \langle \varphi_k,\ f \rangle_{\mcal{H}},
    \]
    where the sum converges in $\H$. From this expression, one can see that if there exists a nonzero $\lambda_k$, then $\| L (L^{\dagger})^* \|_{\mcal{H}\to\mcal{H}} = \| L L^{\dagger} \|_{\mcal{H}\to\mcal{H}} = 1$.
    By Parseval's identity, we also have
    \[
        \| L^{\dagger} f - L^{\dagger} P_n f \|_{\mcal{H}}^2
        = \sum_{\substack{k =n+1\\ \lambda_k \neq 0}}^\infty | \lambda_k |^{-2} |\langle \varphi_k,\ f \rangle_{\mcal{H}} |^2
        \leq | \lambda_{n+1}|^{-2} \| f \|_{\mcal{H}}^2,
    \]
    and equality is achieved by $f = \varphi_{n+1}$.
    Then, $\| L^{\dagger} - L^{\dagger} P_n \|_{\mcal{H}\to\mcal{H}} = | \lambda_{n+1} |^{-1}$ when $\lambda_{n+1} \neq 0$.
    Finally, \cref{eqn:operator_projection_bound} holds with equality for the operator defined by \cref{eqn:operator_achieving_equality_in_projection_bound}.
\end{proof}

\begin{remark}[Role of $\|LA^*\|$]
The constant $\|LA^*\|$ quantifies how well $L$ captures information about the range of $A^*$. This is related to the definition of near-symmetry in the finite-dimensional case in \cref{def_near_sym}.
\end{remark}

\cref{thm:operator_projection} shows that the approximation error between $A$ and $A P_n$ decays at a rate determined by the eigenvalues of $L$. A practical way to compute the constant $\|L A^*\|_{\mcal{H}'\to\mcal{H}}$ in \cref{thm:operator_projection} is given by the following lemma.

\begin{lemma} \label{lem:LAstar_norm}
    The bounded operator $(L A^*)^*$ extends the operator $A L : D(L) \subset \H \to \H'$ that is densely-defined. Let $\{\tilde{P}_n\}_{n=1}^{\infty}$ be a sequence of orthogonal projections in $\H$ with $\Range(\tilde{P}_n) \subset D(L)$ and $\tilde{P}_n$ converging strongly to the identity.
    Then, we have
    \begin{equation} \label{eqn:norm_of_LAstar}
        \| L A^* \|_{\mcal{H}'\to\mcal{H}}
        = \sup_{\substack{f\in D(L)\\ \| f\|_{\H} \leq 1}} \| A L f \|_{\mcal{H}'}
        = \lim_{n\to\infty} \| A L \tilde{P}_n \|_{\mcal{H}\to\mcal{H}'}.
    \end{equation}
\end{lemma}

We emphasize that one can choose $\tilde{P}_n = P_n$ to estimate the constant $\|LA^*\|_{\mcal{H}'\to\mcal{H}}$ numerically in \cref{eqn:norm_of_LAstar} as in \cref{alg_constant_estimate}. To understand \cref{lem:function_projection} and \cref{thm:operator_projection} more concretely, we first consider the case where $L$ is constructed from the Laplace--Beltrami operator on a compact Riemannian manifold without boundary in~\cref{sec:manifold_without_boundary_case}. Then, we extend this analysis to domains with boundaries in \cref{sec:domain_with_boundary_case}.

\begin{algorithm}[htbp]
\caption{Estimation of $\|LA^*\|$}\label{alg_constant_estimate}
\begin{algorithmic}[1]
\Require Bounded linear operator $A:\mathcal{H}\to \mathcal{H}'$, self-adjoint operator $L:D(L)\subset \mathcal{H}\to\mathcal{H}$ such that $\Range(A^*) \subset D(L)$, integer $n\geq 1$.
\State Compute the first $n$ eigenfunctions $\varphi_1,\ldots,\varphi_n$ of $L$ and eigenvalues $\lambda_1,\ldots,\lambda_n$, with $|\lambda_1|\leq |\lambda_2|\leq \cdots |\lambda_n|$.
\State Sample the operator $A$ $n$ times at the eigenfunctions of $L$ to obtain \[u_1=A(\varphi_1),\quad \ldots,\quad u_n=A(\varphi_n).\]
\State Define the matrix $M_n$ as
\[M_n\coloneqq\begin{bmatrix}\lambda_1 u_1&\ldots&\lambda_n u_n\end{bmatrix}.\]
\Ensure Approximation $\|M_n\|_2$ of the constant $\|LA^*\|$.
\end{algorithmic}
\end{algorithm}

\subsection{Compact Manifolds without Boundaries} \label{sec:manifold_without_boundary_case}

Let $\M$ be a smooth, compact, $d$-dimensional manifold without boundary, and let $\langle \cdot, \cdot \rangle_g$ be a Riemannian metric on $\M$. An example of such as manifold is the periodic box, also known as the torus $\mathbb{T}^d$ formed by identifying opposite faces of the cube $[0,1]^d$ in $d$-dimensional Euclidean space $\R^d$. In this section, we consider the case where $A: \H \to \H$ is an operator on $\H = L^2(\M)$, while the prior operator $L$ in \cref{thm:operator_projection} is constructed from the Laplace--Beltrami operator on the manifold. The classical LBO $\Delta_g:D(\Delta_g)\subset L^2(\M) \to L^2(\M)$ is defined on a domain of smooth functions $D(\Delta_g) = C^{\infty}(\M)$ using the standard formula
\begin{equation} \label{eqn:classical_LBO}
    \Delta_g f
    = - \vdiv(\grad f)
    = - \frac{1}{\sqrt{\det(g_{\cdot,\cdot})}} \frac{\partial}{\partial x^j}\left( \sqrt{\det(g_{\cdot,\cdot})} g^{j,k} \frac{\partial f}{ \partial x^k} \right),
\end{equation}
where the overall sign is a matter of convention.
Here, $g_{i,j} = \langle \partial/\partial x^i, \partial/\partial x^j \rangle_g$ is the metric tensor in a coordinate chart $(x^1, \ldots, x^d): \mcal{U} \to \R^d$ on an open subset $\mcal{U}\subset \mcal{M}$, and $g^{i,j}$ is the inverse metric tensor defined by the relation $g^{i,j} g_{j,k} = \delta^{i}_{k}$, with $\delta^{i}_{k}$ being the Kronecker delta. On the torus $\mathbb{T}^d$ with Euclidean coordinates $x^i$, the LBO is the standard Laplacian operator with flipped sign $\Delta_g = - \frac{\partial}{\partial x^1} \frac{\partial}{\partial x^1} - \cdots - \frac{\partial}{\partial x^d} \frac{\partial}{\partial x^d}$ and periodic boundary conditions. Using our choice of sign, the LBO is nonnegative and symmetric due to Green's identity (see~\cref{app:Fourier_proj_proofs}):
\begin{equation}    \label{eqn:Green_without_boundary}
    \int_{\M} f_1 \Delta_g f_2 \d \mu_g
    = \int_{\M} \langle \grad f_1, \grad f_2 \rangle_g \d \mu_g,
\end{equation}
which holds for every $f_1, f_2 \in C^{\infty}(\M)$. Here, $\d\mu_g$ is the Riemannian density defined in~\citep[Prop.~16.45]{Lee2013introduction}. Moreover, the closure of the classical LBO $\Delta = \overline{\Delta_g}$ is self-adjoint (see~\cref{lem:boundaryless_LBO_domain} in \cref{app:Fourier_proj_proofs}) and the domains of its powers coincide with the Sobolev spaces
\[
    D(\Delta^{k/2}) = H^k(\M) = W^{k,2}(\M),
\]
consisting of square-integrable functions with $k\geq 1$ weak derivatives in $L^2(\M)$. Then, choosing the operator $L = \Delta^{k/2}$ in \cref{lem:function_projection,thm:operator_projection} yields the following result, whose proof is available in \cref{app:Fourier_proj_proofs}.

\begin{theorem} \label{thm:manifold_without_boundary_case}
    Let $\M$ be a smooth, compact Riemannian manifold without boundary.
    Then $L^2(\M)$ admits an orthonormal basis consisting of eigenfunctions $\{ \varphi_j \}_{j=1}^{\infty}$ of $\Delta = \overline{\Delta_g}$ with eigenvalues $0 \leq \lambda_1 \leq \lambda_2 \leq \cdots$, and $\lambda_n\to\infty$.
    Let $P_n:L^2(\M) \to L^2(\M)$ denote the orthogonal projection onto $\vspan\{ \varphi_j \}_{j=1}^n$.
    If $\lambda_{n+1} \neq 0$ then for every integer $k\geq 1$ we have
    \begin{equation} \label{eqn:function_approx_boundaryless_case}
        \| f - P_n f \|_{L^2(\mcal{M})}
        \leq \frac{1}{\lambda_{n+1}^{k/2}} \| \Delta^{k/2} f \|_{L^2(\mcal{M})},
        \quad f \in H^k(\M),
    \end{equation}
    and equality is achieved by $f = \varphi_{n+1}$.
    Every bounded operator $A:L^2(\M) \to \H'$ with $\Range(A^*) \subset H^k(\M)$ satisfies $\| \Delta^{k/2} A^* \|_{\mcal{H}' \to L^2(\mcal{M})} < \infty$ and
    \[
        \| A - A P_n \|_{L^2(\mcal{M}) \to \mcal{H}'}
        \leq \frac{1}{\lambda_{n+1}^{k/2}} \| \Delta^{k/2} A^* \|_{\mcal{H}' \to L^2(\mcal{M})},
    \]
    where equality is achieved by the operator $A = (\Delta^{k/2})^{\dagger}$.
\end{theorem}

The $k=1$ case in \cref{eqn:function_approx_boundaryless_case} was proven by \cite{Aflalo2013spectral, Aflalo2015optimality}. Our generalization is useful when considering solution operators $A: L^2(\M) \to L^2(\M)$ associated with uniformly elliptic PDEs, which benefit from higher degrees of regularity. This allows us to apply \cref{thm:manifold_without_boundary_case} for $k > 1$ to achieve faster convergence rates.

\subsection{Application to Differential Operators}
A concrete application of \cref{thm:manifold_without_boundary_case} arises when approximating the solution operator $A$ associated with differential equations. Here, we consider a $k$th-order smooth scalar differential operator $\mcal{L}$ on $\mcal{M}$.
The formal adjoint of $\mcal{L}$ is the unique differential operator, denoted $\mcal{L}^{\top}$, which satisfies the integration-by-parts formula
\[
    \int_{\mcal{M}} v \mcal{L}(u) \d\mu_g = \int_{\mcal{M}} \mcal{L}^\top(v) u \d\mu_g,
\]
for all test functions $u,v \in \mcal{D}(\mcal{M})$, that is, infinitely differentiable functions with compact support in the interior of $\mcal{M}$. Here, since $\mcal{M}$ is assumed to be compact and boundaryless, we have $\mcal{D}(\mcal{M}) = C^{\infty}(\mcal{M})$.
The formal adjoint can be constructed by observing that $\mcal{L}$ can always be written as
\[
    \mcal{L}(u) = \sum_{\| \alpha \| \leq k} a_{\alpha} X_{1}^{\alpha_1} \circ X_{2}^{\alpha_2} \circ \cdots \circ X_{\tilde{m}}^{\alpha_{\tilde{m}}} (u),
\]
where $X_j$ are smooth vector fields on $\mcal{M}$, $\alpha = (\alpha_1, \ldots, \alpha_{\tilde{m}})$ is a multi-index, and $a_{\alpha} \in C^{\infty}(\mcal{M})$ are smooth coefficient functions. The formal adjoint of a smooth vector field $X$ is the differential operator
\begin{equation} \label{eqn:adjoint_of_vector_field}
    X^\top(v) = - X(v) - \vdiv(X) v,
\end{equation}
thanks to the relation $\vdiv(u v X) = v X(u) +  \left[ X(v) + \vdiv(X) v \right] u$ and the divergence theorem \citep[Thm.~16.48]{Lee2013introduction}.
Applying \cref{eqn:adjoint_of_vector_field} recursively allows us to express the formal adjoint of the differential operator $\mcal{L}$ as
\[
    \mcal{L}^\top(v) = \sum_{\| \alpha \| \leq k} (X_{\tilde{m}}^{\alpha_{\tilde{m}}})^\top \circ \cdots \circ (X_2^{\alpha_2})^\top \circ (X_1^{\alpha_1})^\top (a_{\alpha} v).
\]
This expression satisfies the integration-by-parts formula.

By examining the coefficient functions in local charts on $\mcal{M}$, one can see that $\mcal{L}$ is elliptic if and only if $\mcal{L}^\top$ is elliptic. Supposing that $\mcal{L}$ is elliptic and $u$ is a distribution solving the partial differential equation
\begin{equation} \label{eqn:PDE_on_boundaryless_mfd}
    \mcal{L}(u) = f, \quad f \in H^{s}(\mcal{M}),
\end{equation}
where $s\geq 0$. Then, the interior elliptic regularity theorem \citep[Ch.~5, Thm.~11.1]{Taylor2011PDEI} says that $u \in H^{k + s}(\mcal{M})$ and satisfies the following regularity estimate
\[
    \| u \|_{H^{k+s}(\mcal{M})} \leq C(\mcal{M}, \mcal{L}, s, \sigma) \left( \| f \|_{H^s(\mcal{M})} + \| u \|_{H^{\sigma}(\mcal{M})} \right), \quad \forall \sigma < k + s.
\]
Let us assume that for every $f \in L^2(\mcal{M})$, \cref{eqn:PDE_on_boundaryless_mfd} has a solution $u = A f$ given by the operator $A: L^2(\mcal{M}) \to L^2(\mcal{M})$.
Elliptic regularity ensures that $A$ is bounded, and in fact, that $\Range(A) \subset H^{k}(\mcal{M})$.
Integration by parts shows that $v = A^* g$ solves the adjoint PDE
\[
    \mcal{L}^\top(v) = g, \quad g \in L^2(\mcal{M})
\]
in the sense of distributions, and is thus also a strong solution by the interior elliptic regularity theorem.
Thus, we have $\Range(A^*) \subset H^{k}(\mcal{M})$ and \cref{thm:manifold_without_boundary_case} applies.

The asymptotic behavior of the approximation bound in \cref{thm:manifold_without_boundary_case} is determined by the growth of the eigenvalues $\lambda_n$ of the LBO.
Weyl's law \citep{Weyl1911asymptotische} characterizes the asymptotic distribution of the eigenvalues as $\lambda_n \sim c n^{2/d}$ for a constant $c>0$ (see also~\citealt{Canzani2013analysis, Minakshisundaram1949some}). The asymptotic distribution of eigenvalues yields an approximation error that is asymptotically
\begin{equation} \label{decay_eig_law}
    \| A - A P_n \| = \mcal{O}\big( n^{-k/d} \big)
    \quad \mbox{as} \quad n \to \infty,
\end{equation}
where $k\geq 1$ is the order of the differential operator and $d$ is the spatial dimension. Hence, the rate in \cref{decay_eig_law} degrades in high dimensions but improves as the regularity increases. However, for problems involving solution operators of PDEs motivated by physics applications, we often have $d \leq 3$.

\begin{remark}
    We can also construct the prior operator $L$ using powers of other uniformly elliptic differential operators besides the LBO.
    Similar asymptotic laws also hold for the eigenvalues of these operators \citep{beals1970asymptotic,clark1967asymptotic}.
    By choosing $L$ based on additional prior information about $A$, it may be possible to reduce the magnitude of the constant $\| L A^* \|_{\mcal{H}' \to L^2(\mcal{M})}$ in the approximation bound.
\end{remark}

\subsection{Domains with Boundaries} \label{sec:domain_with_boundary_case}

Many practical applications involve solving PDEs on spatial domains with boundaries and specified boundary conditions (BCs) rather than the compact manifolds without boundaries as described in \cref{sec:manifold_without_boundary_case}. With a spatial domain consisting of a smooth Riemannian manifold $(\Omega, \langle \cdot, \cdot\rangle_g)$ with smooth boundary $\partial \Omega$, we consider a $k$th-order elliptic partial differential equation of the form:
\begin{equation} \label{eqn:PDE_on_domain_with_boundary}
    \mbox{$\mcal{L}(u) = f$ in $\Omega$}, \quad \mbox{and} \quad \mbox{$\mcal{B}_i(u) = 0$, $1\leq i\leq l$, on $\partial \Omega$}.
\end{equation}
We assume that this equation has a solution for every $f \in H^s(\Omega)$ and that these solutions satisfy the global elliptic regularity estimate
\begin{equation}\label{eqn:global_regularity}
    \| u \|_{H^{k+s}(\Omega)} \leq C(\Omega, \mcal{L}, s, \sigma) \left( \| f \|_{H^s(\Omega)} + \| u \|_{H^{\sigma}(\Omega)} \right), \quad \forall \sigma < k + s.
\end{equation}
Unlike in the boundaryless case, the conditions for such an estimate to hold depend delicately on the operator $\mcal{L}$, the boundary $\partial \Omega$, and boundary conditions specified by the differential operators $\mcal{B}_j$~\citep[Ch.~4, Sec.~11]{Taylor2011PDEI}.

An adjoint problem associated with \cref{eqn:PDE_on_domain_with_boundary} can be defined as
\begin{equation} \label{eqn:adjoint_PDE_on_domain_with_boundary}
    \mbox{$\mcal{L}^\top(v) = g$ in $\Omega$}, \quad \mbox{and} \quad \mbox{$\mcal{B}_i'(u) = 0$, $1\leq i\leq l'$, on $\partial \Omega$},
\end{equation}
using the formal adjoint operator $\mcal{L}^\top$ and boundary conditions for which
\[
    \int_{\Omega} v \mcal{L}(u) \d \mu_g = \int_{\Omega} \mcal{L}^\top(v) u \d \mu_g
\]
holds for every $u \in D$ and $v \in D'$, where $D$ and $D'$ are defined as
\begin{equation}\label{eqn:domains_for_PDE_and_adjoint_with_boundary}
    \begin{aligned}
        D  & = \{ u \in H^k(\Omega) \colon \mcal{B}_i(u) = 0, \ 1\leq i\leq l, \ \mbox{on} \ \partial \Omega \}, \ \mbox{and} \\
        D' & = \{ v \in H^k(\Omega) \colon \mcal{B}_i'(v) = 0, \ 1\leq i\leq l', \ \mbox{on} \ \partial \Omega \}.
    \end{aligned}
\end{equation}
Here, the boundary conditions are understood in the sense of trace.
We assume the adjoint problem also has a solution satisfying a corresponding global regularity estimate in the form of \cref{eqn:global_regularity} for each $g \in H^s(\Omega)$.

\begin{remark}
    Conditions ensuring the existence of adjoint boundary operators and regular solutions of the adjoint boundary value problem are discussed in \citep[Thm.~12.7]{Taylor2011PDEI} and \citep[Thms.~8.37~and~8.41]{Renardy2004introduction}.
    As an example, suppose that the boundary value problem given in \cref{eqn:PDE_on_domain_with_boundary} is defined in a planar region $\Omega \subset \R^d$, has even order $k = 2m$, satisfies the ``complementing conditions'' described by \citet[Def.~8.28]{Renardy2004introduction}, and has a unique regular solution for each $f \in L^2(\Omega)$.
    Then, it follows immediately from \citep[Thms.~8.37~and~8.41]{Renardy2004introduction} that adjoint boundary operators exist and the adjoint problem in \cref{eqn:adjoint_PDE_on_domain_with_boundary} has regular solutions.
\end{remark}

If $A: L^2(\Omega) \to L^2(\Omega)$ is a solution operator associated with \cref{eqn:PDE_on_domain_with_boundary}, then its adjoint, $A^*$, is the solution operator for the adjoint problem defined in \cref{eqn:adjoint_PDE_on_domain_with_boundary}.
Additionally, the global regularity estimate~\eqref{eqn:domains_for_PDE_and_adjoint_with_boundary} ensures that $\Range(A) \subset D \subset H^k(\Omega)$, and, more importantly, that $\Range(A^*) \subset D' \subset H^k(\Omega)$.
To verify that $A^*$ is the solution operator associated with \cref{eqn:adjoint_PDE_on_domain_with_boundary},  suppose that $v$ is a solution to \cref{eqn:adjoint_PDE_on_domain_with_boundary} for a given $g\in L^2(\Omega)$, and thus $v \in D'$ due to global regularity. If $f \in L^2(\Omega)$, we also have $A f \in D$ due to global regularity.
Therefore,
\[
    \langle f,\ A^* g \rangle_{L^2(\Omega)}
    = \langle A f,\ \mcal{L}^\top (v) \rangle_{L^2(\Omega)}
    = \langle \mcal{L} A f, \ v \rangle_{L^2(\Omega)}
    = \langle f, \ v \rangle_{L^2(\Omega)},
\]
which implies that $v = A^* g$, as desired.

We provide two methods for approximating the solution operator $A$ that do not require solving the adjoint problem in \cref{sec_matching_adjoint_boundary_conditions_method,sec:extension_method}:
\begin{enumerate}
    \item The matching adjoint boundary conditions method. It can be applied for even order ($k=2m$) elliptic operators and whenever the adjoint boundary conditions $\mcal{B}_j'$ can be used to define a self-adjoint extension $L: D(L) \subset L^2(\Omega) \to L^2(\Omega)$ of $\Delta_g^m$. Here, $\Delta_g^m$ denotes the $m$th power of Laplace--Beltrami operator on $\Omega$. The key idea is to apply \cref{thm:operator_projection} directly under sufficient assumptions, ensuring that $\Range(A^*) \subset D' \subset D(L)$.
    \item The extension method. It is more general and can always be applied, but it might lead to a worse constant factor in the resulting approximation bound. The key idea is to embed the domain $\Omega$ in a compact manifold $\mcal{M}$ without boundary, on which a self-adjoint extension of the LBO can be defined (see the discussion in \cref{sec:manifold_without_boundary_case}) and apply \cref{thm:manifold_without_boundary_case}.
\end{enumerate}

\subsubsection{Matching adjoint boundary conditions method} \label{sec_matching_adjoint_boundary_conditions_method}

In this section, we assume that the differential operator $\mcal{L}$ has even order $k = 2m$, and aim to construct a self-adjoint extension $L: D(L)\subset L^2(\Omega) \to L^2(\Omega)$ of the $m$th power of the Laplace--Beltrami operator, $\Delta_g^m$, satisfying $D' \subset D(L)$.
If this can be accomplished, then our main approximation result in \cref{thm:operator_projection} can be applied because $\Range(A^*) \subset D' \subset D(L)$.
To do this, we consider a set of differential operators $\{\mcal{B}''_i\}_{i=1}^{l''}$ acting on the boundary of $\Omega$, so that the space of functions $D'$, defined by \cref{eqn:domains_for_PDE_and_adjoint_with_boundary}, is included in the following function space:
\[
    D'' = \{ u \in H^{2m}(\Omega)\colon \mcal{B}_i''(u) = 0, \ 1\leq i\leq l'', \ \mbox{on} \ \partial \Omega \}.
\]
We assume that the $m$th power of the LBO, understood (with a slight abuse of notation) as a differential operator acting on distributions, satisfies the symmetry condition:
\begin{equation}\label{eqn:LBO_power_symmetry_condition}
    \langle \Delta_g^m u, \ v \rangle_{L^2(\Omega)} = \langle u, \ \Delta_g^m v \rangle_{L^2(\Omega)}, \quad \forall u,v \in D''.
\end{equation}
Furthermore, we assume that there exists a constant $c \in \R$ such that the boundary value problem
\begin{equation}\label{eqn:BVP_for_LBO_power}
    \mbox{$c u + \Delta_g^m u = f$ in $\Omega$},
    \quad \mbox{and} \quad
    \mbox{$\mcal{B}''_i(u) = 0$, $1\leq i\leq l''$, on $\partial \Omega$}
\end{equation}
has a solution $u \in H^{2m}(\Omega)$ satisfying a global regularity estimate in the form of \cref{eqn:global_regularity} for each $f \in L^2(\Omega)$.
Typical examples where symmetry and regularity hold in the $m=1$ case include Dirichlet and Neumann boundary conditions corresponding to operators $\mcal{B}''_1(u) = u$ and $\mcal{B}''_1(u) = \vec{n} \cdot \grad u$, respectively. Here, $\vec{n}$ denotes the unit outward normal vector field along $\partial \Omega$. We refer the reader to \citep{Ibort2015self} for more details on the choice of boundary conditions leading to self-adjoint extensions of the Laplace--Beltrami operator. Boundary conditions, self-adjointness, and eigenvalue estimates for biharmonic ($m=2$) and poly-harmonic ($m \geq 2$) operators are discussed by \cite{colbois2022neumann, ilias2010universal}. The following lemma shows that our assumptions yield a self-adjoint power of the LBO with domain $D(L) = D''$.

\begin{lemma}\label{lem:self_adjoint_BCs_for_LBO}
    Let $L = \left.\Delta_g^m\right|_{D''} : D(L) = D'' \subset L^2(\mcal{M}) \to L^2(\mcal{M})$ denote the restriction of $\Delta_g^m$ to $D''$, where $\Delta_g^m$ is understood as a differential operator acting on distributions. Suppose that \cref{eqn:LBO_power_symmetry_condition} holds and \cref{eqn:BVP_for_LBO_power} has a solution $u \in H^{2m}(\Omega)$ satisfying an estimate in the form of \cref{eqn:global_regularity} for each $f \in L^2(\Omega)$. Then $L$ is self-adjoint and $D(L) = D''$, endowed with the graph norm, is compactly embedded in $L^2(\Omega)$.
\end{lemma}

Then, the eigenfunctions $\{\varphi_j\}_{j\geq 1}$ and eigenvalues $\{\lambda_j\}_{j\geq 1}$ of $L$ are obtained by solving the following eigenvalue problem:
\[
    \mbox{$\Delta_g^m \varphi_j = \lambda_j \varphi_j$ in $\Omega$},
    \quad \mbox{and} \quad
    \mbox{$\mcal{B}''_i(\varphi_j) = 0$, $1\leq i\leq  l''$, on $\partial \Omega$}.
\]
After ordering the eigenvalues by increasing magnitude, that is $| \lambda_1 | \leq | \lambda_2 | \leq \cdots$, we can form the orthogonal projection $P_n$ onto $\vspan\{ \varphi_j \}_{j=1}^n$ in $L^2(\Omega)$.
Finally, the combination of \cref{thm:operator_projection} and the fact that $\Range(A^*) \subset D' \subset D'' = D(L)$ allows us to derive the following error bound between $A$ and our approximation $A P_n$:
\begin{equation} \label{eqn:matching_based_operator_approx_with_boundary}
    \| A - A P_n \|_{L^2(\mcal{M})\to L^2(\mcal{M})} \leq \frac{1}{| \lambda_{n+1} |} \| \Delta_g^m A^* \|_{L^2(\mcal{M})\to L^2(\mcal{M})}.
\end{equation}
The right-hand side of \cref{eqn:matching_based_operator_approx_with_boundary} is bounded when $n$ is chosen large enough so that $\lambda_{n+1} \neq 0$. Here, we emphasize that $\lambda_{j}$ are the eigenvalues of the $m$th power of the LBO so that the approximation error converges to zero as $n$ increases.

A significant limitation of the ``matching adjoint boundary conditions'' method described in this section is that appropriate boundary conditions, i.e., the operators $\mcal{B}_i''$, might not exist or be known in advance. Fortunately, the method described in the next section can be applied even when such boundary conditions cannot be found at the cost of a larger constant in the approximation bound.

\subsubsection{Extension method} \label{sec:extension_method}

The second option presented in \cref{alg_extension} can be used even when the adjoint boundary conditions satisfied by $A^*$ do not give rise to a self-adjoint power of the LBO.
In fact, it can be used whenever $A: L^2(\Omega) \to \mcal{H}'$ and $\Range(A^*) \subset H^k(\Omega)$. We assume that the spatial domain $\Omega$ is a subset of a compact, $d$-dimensional Riemannian manifold $\M$ without boundary. For example, a bounded region $\Omega \subset \R^d$ can be embedded into the torus $\M = \mathbb{T}^d$ by constructing a cube around $\Omega$ and identifying opposite faces
of the cube. More generally, any smooth compact manifold $\Omega$ with smooth boundary $\partial \Omega$ is embedded in its ``double'' $\mcal{M} = \Omega \# \Omega$, a smooth compact manifold without boundary obtained by attaching $\Omega$ to a copy of itself along $\partial \Omega$~\citep[Ex.~9.32]{Lee2013introduction}.

\renewcommand{\algorithmicrequire}{\textbf{Input:}}
\renewcommand{\algorithmicensure}{\textbf{Output:}}
\begin{algorithm}[htbp]
\caption{Extension-based adjoint-free approximation algorithm}\label{alg_extension}
\begin{algorithmic}[1]
\Require Smooth, compact Riemannian manifold $\mcal{M}$ without boundary, subset $\Omega \subset \mcal{M}$ satisfying the modified $k$-extension property (\cref{def:modified_p_extension}), bounded linear operator $A:L^2(\Omega) \to \mathcal{H}'$ satisfying $\Range(A^*) \subset H^k(\Omega)$, and an integer $n\geq 1$.
\State Compute the first $n$ eigenfunctions $\varphi_1,\ldots,\varphi_n$ of the Laplace-Beltrami operator $\Delta_g$ on $\mcal{M}$ given by \cref{eqn:classical_LBO} and eigenvalues $\lambda_1,\ldots,\lambda_n$, with $|\lambda_1|\leq |\lambda_2|\leq \cdots |\lambda_n|$.
\State Form an orthonormal basis $\psi_1, \ldots, \psi_{m}$ for $\vspan\{ \left.\varphi_1\right|_{\Omega}, \ldots, \left.\varphi_n\right|_{\Omega} \}$ in $L^2(\Omega)$.
\State Sample the operator $A$ on the basis functions to obtain \[u_1=A(\psi_1),\quad \ldots,\quad u_{m}=A(\psi_{m}).\]
\State Define the rank $\leq n$ projected operator $A P_n:L^2(\Omega)\to \mathcal{H}'$ as
\[A P_n(f) \coloneqq \sum_{k=1}^m u_k\langle \psi_k,\ f\rangle_{L^2(\Omega)} = \sum_{k=1}^m A(\psi_k)\langle \psi_k,\ f\rangle_{L^2(\Omega)}, \quad f\in L^2(\Omega)\]
\Ensure Approximation $A P_n$ of $A$ satisfying
\[\| A - A P_n \|_{L^2(\Omega)\to\mcal{H}'} \leq \frac{C(\Omega, \mcal{M}, k)}{| \lambda_{n+1}^{k/2} |} \| A^* \|_{\mcal{H}'\to H^k(\Omega)}.\]
\end{algorithmic}
\end{algorithm}

\Cref{alg_extension} for approximating the operator $A$ proceeds by first computing the eigenfunctions $\{\varphi_j\}$ with eigenvalues $\{\lambda_j\}$ of the self-adjoint LBO $\Delta$ defined on $\M$, following the discussion of \cref{sec:manifold_without_boundary_case}. Note that by regularity, these can be computed using the classical LBO given by \cref{eqn:classical_LBO}.
We then restrict these eigenfunctions to the domain $\Omega$ and use the orthogonal projection $P_n$ onto $\vspan\{ \left.\varphi_1\right|_{\Omega}, \ldots, \left.\varphi_n\right|_{\Omega} \}$ in $L^2(\Omega)$ to form the approximant $A P_n$. When the domain $\Omega$ has the following Sobolev extension property (see~\cref{def:modified_p_extension}), this algorithm has an error bound similar to \cref{thm:manifold_without_boundary_case}, but with an additional constant factor introduced by Sobolev extension.

\begin{definition}[Modified $k$-extension property]
    \label{def:modified_p_extension}
    A subset $\Omega \subset \M$ is said to have the modified $k$-extension property if there exists a bounded linear operator $E: L^2(\Omega) \to L^2(\M)$, satisfying $ \left. (E f) \right|_{\Omega} = f$ for every $f\in L^2(\Omega)$, such that
    \[
        \| E f \|_{H^k(\M)} \leq C(\Omega, \M, k) \| f \|_{H^k(\Omega)}, \quad f \in H^k(\Omega),
    \]
    where $C(\Omega, \M, k)>0$ is a constant depending only on $\Omega$, $\M$, and $k$.
\end{definition}

The Sobolev extension method described by \citet[Section~4.4]{Taylor2011PDEI} for smooth, compact manifolds $\Omega$ with smooth boundary provides a modified $k$-extension. The only difference is that a modified $k$-extension must be defined and bounded as an operator $E: L^2(\Omega) \to L^2(\M)$, restricted to a bounded operator $H^k(\Omega) \to H^k(\M)$, rather than being defined only as an operator $H^k(\Omega) \to H^k(\M)$. This modification is unnecessary for \cref{eqn:extension_based_function_approx_with_boundary} in \cref{thm:domain_with_boundary_case}. We obtain the approximation bound stated below for domains with the modified $k$-extension property.

\begin{theorem} \label{thm:domain_with_boundary_case}
    Let $\Omega \subset \M$ be a domain satisfying the modified $k$-extension property, and $P_n: L^2(\Omega) \to L^2(\Omega)$ denote the orthogonal projection onto $\vspan\{ \left. \varphi_j\right|_{\Omega} \}_{j=1}^n$. There exists a constant $C(\Omega, \M, k)$, depending only on $\Omega$, $\M$, and $k$, so that for every $n\geq 1$ with $\lambda_{n+1}\neq 0$, we have
    \begin{equation}
        \| f - P_n f \|_{L^2(\Omega)}
        \leq \frac{C(\Omega, \M, k)}{\lambda_{n+1}^{k/2}} \| f \|_{H^k(\Omega)},
        \quad f\in H^k(\Omega).
        \label{eqn:extension_based_function_approx_with_boundary}
    \end{equation}
    Moreover, every bounded operator $A:L^2(\Omega) \to \H'$ into a Hilbert space $\H'$ with $\Range(A^*) \subset H^k(\Omega)$ satisfies $\| A^* \|_{\H' \to H^k(\Omega)} < \infty$ and
    \[
        \| A - A P_n \|_{L^2(\mcal{M})\to\mcal{H}'} \leq \frac{C(\Omega, \M, k)}{\lambda_{n+1}^{k/2}}\| A^* \|_{\H' \to H^k(\Omega)}.
    \]
\end{theorem}
\noindent
When the norm on $H^k(\mcal{M})$ is defined so that $\| \Delta^{k/2} f \|_{L^2(\mcal{M})} \leq \| f \|_{H^k(\mcal{M})}$, then the constant in the theorem is the same as the Sobolev extension constant in \cref{def:modified_p_extension}.
In general, this constant is greater than one and increases exponentially fast with $k$~\citep[Thm.~5.21]{adams2003sobolev}. However, we are mostly interested in the decay rate with respect to $n$ for a fixed $k$.

To understand this result in the setting described at the beginning of \cref{sec:domain_with_boundary_case}, suppose that $A: L^2(\Omega) \to L^2(\Omega)$ and $A^*$ are the solution operators for the boundary value problems respectively defined in \cref{eqn:PDE_on_domain_with_boundary,eqn:adjoint_PDE_on_domain_with_boundary}. By assumption, these solutions satisfy the global elliptic regularity estimate of \cref{eqn:global_regularity}, which implies that $\Range(A^*) \subset H^k(\Omega)$, where $k$ is the order of the elliptic differential operator $\mcal{L}$ in \cref{eqn:PDE_on_domain_with_boundary}.
In this case, \cref{thm:domain_with_boundary_case} yields the bound
\begin{equation}\label{eqn:extension_based_operator_approx_with_boundary}
    \| A - A P_n \|_{L^2(\mcal{M})\to L^2(\mcal{M})} \leq \frac{C(\Omega, \M, k)}{\lambda_{n+1}^{k/2}} \| A^* \|_{L^2(\mcal{M}) \to H^k(\Omega)},
\end{equation}
where $\| A^* \|_{L^2(\mcal{M}) \to H^k(\Omega)}$ is determined by the elliptic regularity constant in \cref{eqn:global_regularity} for the adjoint solution operator.

While it is applicable in a broader range of settings, the bound in \cref{eqn:extension_based_operator_approx_with_boundary} is typically worse than the one given by \cref{eqn:matching_based_operator_approx_with_boundary}, obtained using the method described in \cref{sec_matching_adjoint_boundary_conditions_method}.
The two main reasons for this are the introduction of the Sobolev extension constant, which is greater than one, and the observation that the eigenvalues of the LBO generally decrease when the spatial domain is enlarged from $\Omega$ to $\mcal{M}$. For the Laplacian operator with Dirichlet or Neumann boundary conditions, the monotonicity property of the eigenvalues when the domain is enlarged follows from the min-max principle~\citep[Thms.~62 and 63]{Canzani2013analysis}. This is also reflected in Weyl's law, where the asymptotic magnitudes of the eigenvalues are inverse to the volume of the manifold \citep[Thm.~72]{Canzani2013analysis}.

\subsection{Numerical Examples}

This section contains numerical examples to illustrate the approximation of the solution operator of non-self-adjoint PDEs by solving the PDEs with right-hand sides corresponding to eigenfunctions of the Laplace--Beltrami operator (see~\cref{thm:operator_projection,lem:LAstar_norm}).

\subsubsection{One-dimensional advection-diffusion equation}

We first consider the following one-dimensional (1D) advection-diffusion equation defined on the unit interval $\Omega=[0,1]$:
\begin{equation} \label{eq_advection_diffusion}
    \frac{1}{4}\frac{d^2 u}{d x^2}+5\frac{du}{dx}+u=f,
\end{equation}
with homogeneous Dirichlet boundary conditions. We aim to approximate the solution operator $A:f\to u$ associated with \cref{eq_advection_diffusion} by solving \cref{eq_advection_diffusion} with right-hand sides given by eigenfunctions of the Laplacian operator $L=-d^2/dx^2$ with homogeneous Dirichlet boundary conditions. To do so, we compute the first $N=601$ eigenvalues and eigenfunctions of the Laplacian operator $L=-d^2/dx^2$ satisfying
\[L \varphi_k = \lambda_k \varphi_k, \quad \varphi_k(0)=\varphi_k(1)=0,\quad 1\leq k\leq n,\]
using the Firedrake finite element software~\citep{rathgeber2016firedrake} and the Scalable Library for Eigenvalue Problem Computations (SLEPc)~\citep{hernandez2005slepc}. In this particular case, one could use the analytical expressions for the Laplacian eigenfunctions $\varphi_k(x) = \sin(k \pi x)$ and eigenvalues $\lambda_k = \pi^2 k^2$, but they may not be known in higher dimensions. We partition the interval $\Omega$ into $1000$ cells and discretize the functions with continuous piecewise-cubic polynomials. We then solve \cref{eq_advection_diffusion} with right-hand side $\varphi_k$ for $1\leq k\leq N$ and compute the corresponding solution $u_k$.

\begin{figure}[htbp]
    \centering
    \begin{overpic}[width=\textwidth]{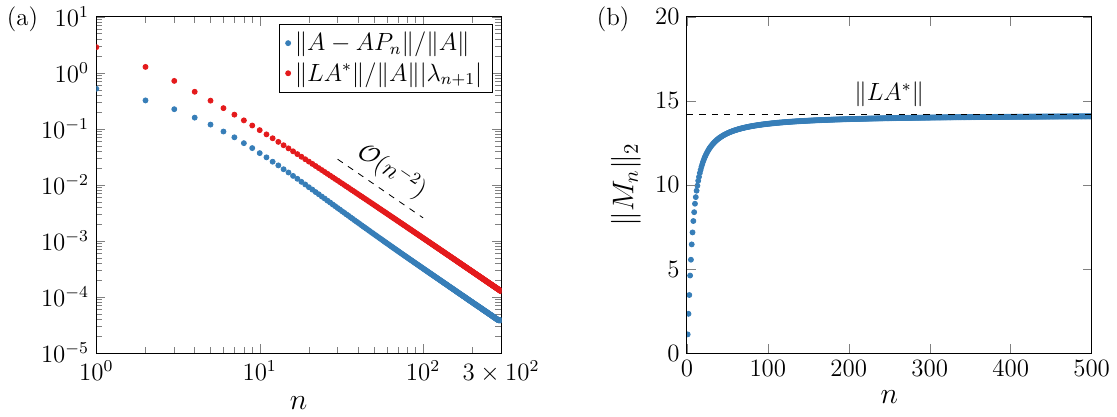}
    \end{overpic}
    \caption{(a) Left and right-hand side of the bound given by \cref{thm:operator_projection} for the approximation of the solution operator of the 1D advection-diffusion equation \cref{eq_advection_diffusion} by the eigenfunctions of the Laplacian operator. (b) Convergence of the spectral norm of the matrix $M_n$ to $\|LA^*\|$ as $n$ increases, illustrating \cref{lem:LAstar_norm}.}
    \label{fig_exp_adv_diff_1d}
\end{figure}

Letting $1\leq n<N$, we estimate the spectral norm $\|A-A P_n\|$ of $A-A P_n$ by restricting this operator to the subspace $\text{span}\{\phi_1,\ldots,\phi_N\}$. This is equivalent to computing the spectral norm of the dense matrix $[u_{n+1}, \ldots, u_N]$. We then approximate the largest singular value of the resulting matrix using the randomized singular value decomposition (SVD)~\citep{halko2011finding} with some power iteration to avoid the expensive computation of a full SVD and estimate $\|A-A P_n\|$. We note that alternative algorithms, such as randomized Lanczos iterations might be more accurate~\citep{kuczynski1992estimating}. Finally, we approximate the constant $\|L A^*\|$ in the right-hand side of \cref{eqn:operator_projection_bound} using \cref{lem:LAstar_norm}. Hence, we construct a matrix $M_n$ whose columns consist of the functions $\lambda_k u_k$ for $1\leq k\leq n$ and compute its spectral norm using the randomized SVD. Following \cref{lem:LAstar_norm}, $\|M_n\|_2$ converges to $\|LA^*\|$ as $n$ increases.

We plot the normalized left-hand side and right-hand side of \cref{thm:operator_projection} corresponding to \cref{eq_advection_diffusion} in \cref{fig_exp_adv_diff_1d}(a) and observe that the approximation error $\|A-A P_n\|$ decays as $n$ increases and is controlled by $\|LA^*\|/|\lambda_{n+1}|$, as guaranteed by \cref{thm:operator_projection}. Furthermore, the decay rate of $\|A-A P_n\|$ behaves asymptotically like $\mathcal{O}(n^{-2})$, which is the same as the decay rate of the eigenvalues $\lambda_n$ of the Laplacian operator $L$ given by Weyl's law~\citep{Weyl1911asymptotische}. In \cref{fig_exp_adv_diff_1d}(b), we display the convergence of $\|M_n\|_2$ as $n$ increases to the approximate value of $\|LA^*\|\approx 14.1$.

\subsubsection{Structural mechanic in two and three dimensions}

As a second numerical example, we consider the following linear elasticity (Navier--Cauchy) equations modeling the small elastic deformations of a body $\Omega\subset \R^d$, where $d\in\{2,3\}$, under the action of an external force $f:\Omega\to\R^d$:
\begin{equation} \label{eq_linear_elasticity}
    \begin{aligned}
        -\nabla \cdot \sigma & = f, \quad \text{in }\Omega,                  \\
        \sigma               & = \lambda \Tr(\epsilon(u))I+2\mu \epsilon(u), \\
        \epsilon(u)          & = \frac{1}{2}(\nabla u + \nabla u^\top),
    \end{aligned}
\end{equation}
where $\sigma$ is the stress tensor, $\lambda=1.25$ and $\mu=1$ are the Lam\'e parameters for the material considered, $\epsilon$ is the symmetric gradient, and $u:\Omega\to \R^d$ is the displacement vector field and solution to \cref{eq_linear_elasticity}~\citep[Sec.~3.3]{langtangen2017solving}. In two dimensions (resp.~3D), we consider the domain $\Omega = [0,1]\times [0,2]$ (resp.~$\Omega = [0,1]\times [0,2]^2$) and enforce a zero displacement boundary condition at the left extremity $x=0$ and traction free conditions on the other boundaries. We discretize the domain $\Omega$ uniformly into $500$ quadrilateral cells and $5000$ hexahedral cells and use continuous piecewise-quadratic and piecewise-linear polynomials to discretize the functions in 2D and 3D, respectively, using the Firedrake finite element software~\citep{rathgeber2016firedrake}. \cref{fig_beam_3d} displays the deformation of the beam under its weight, obtained by solving \cref{eq_linear_elasticity} using the right-hand side $f=(0,0,-1.6\times 10^{-2})$.

\begin{figure}[htbp]
    \centering
    \begin{overpic}[width=0.6\textwidth]{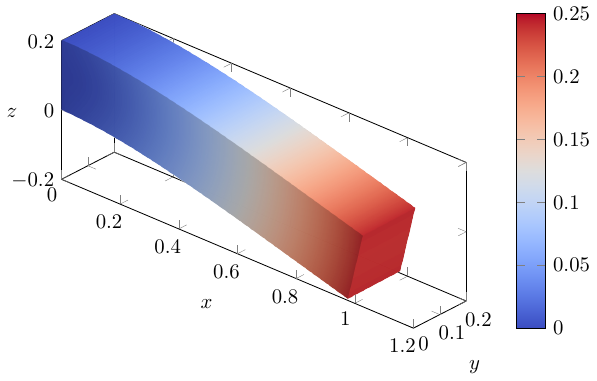}
    \end{overpic}
    \caption{Solution to the 3D linear elasticity equations~\eqref{eq_linear_elasticity} with right-hand side $f=(0,0,-1.6\times 10^{-2})$, modeling the deformation of the beam under its weight. The beam is deformed according to the displacement field $u$, and the color map represents the magnitude of the displacement.}
    \label{fig_beam_3d}
\end{figure}

We compute the first $N=601$ eigenvalues $\{\lambda_k\}_{k=1}^N$ and eigenfunctions $\{\varphi_k\}_{k=1}^N$ of the Laplacian operator, whose action is given by $L u=-\nabla^2 u = -(\nabla\cdot\nabla) u$, with the same boundary conditions as the linear elasticity equations, and use these eigenfunctions to approximate the solution operator $A$ associated with \cref{eq_linear_elasticity}. Hence, we follow the procedure described in \cref{sec:FourierSampling} and approximate $A$ by the operator $A P_n$, where $P_n$ is the orthogonal projection onto the span of the first $n$ eigenfunctions of the Laplacian operator. We then plot in \cref{fig_mechanics_2d}(a) and (c) the normalized left and right-hand sides in the bound given by \cref{thm:operator_projection} for the two and three-dimensional problems and observe that the upper bound (plotted in red) is relatively tight in estimating the decay rate of the approximation error of the solution operator (blue dots). Hence, the convergence rate bounded by $\mathcal{O}(n^{-1})$ in two dimensions and $\mathcal{O}(n^{-3/2})$ in 3D, following \cref{decay_eig_law}.

\begin{figure}[htbp]
    \centering
    \begin{overpic}[width=\textwidth]{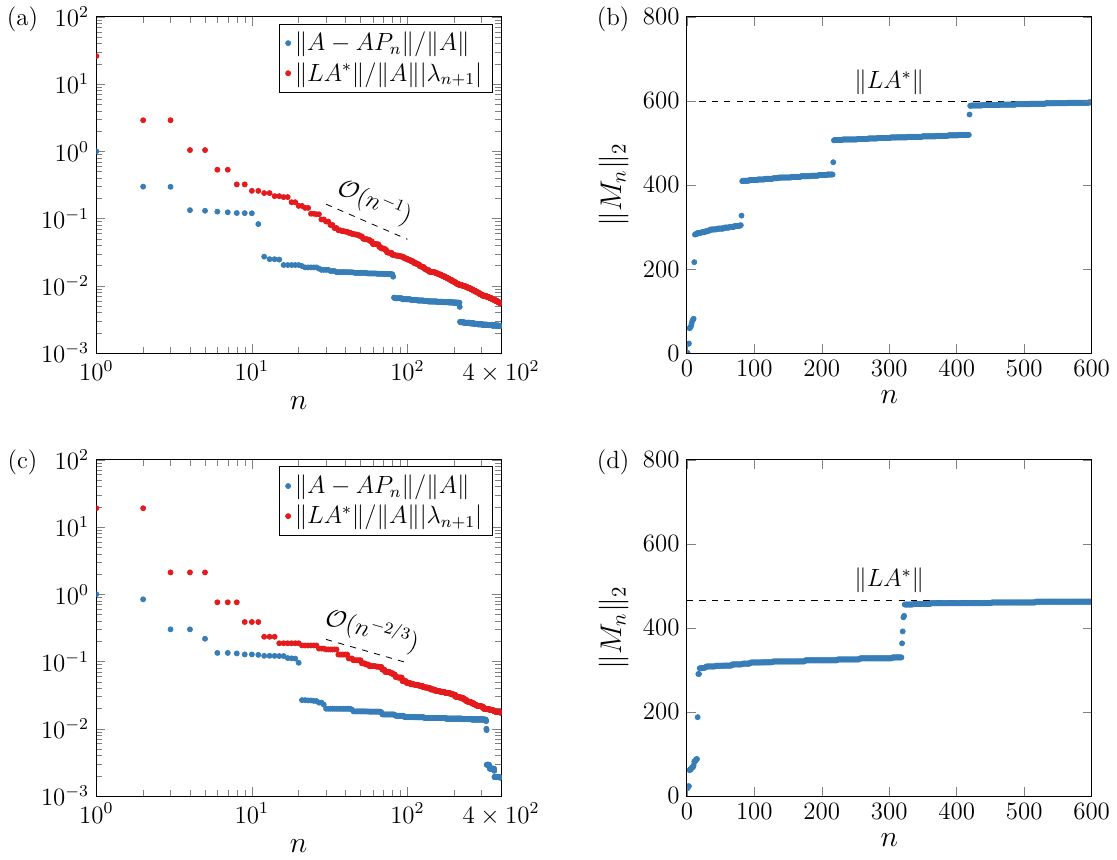}
    \end{overpic}
    \caption{(a) Normalized left and right-hand sides of the bound \cref{thm:operator_projection} for approximating the solution operator of the 2D linear elasticity equations using eigenfunctions of the Laplacian operator. (b) Convergence of the spectral norm of the matrix $M_n$ to the constant $\|LA^*\|$ as $n$ increases. (c)-(d) Same as (a)-(b) but for the 3D linear elasticity equations. The eigenvalues of the solution operator of the Laplacian operator decay as $\mathcal{O}(n^{-1})$ in 2D and $\mathcal{O}(n^{-2/3})$ in 3D, as predicted by \cref{decay_eig_law}.}
    \label{fig_mechanics_2d}
\end{figure}

The constant term $\|LA^*\|$ in the right-hand side of \cref{thm:operator_projection} is estimated using \cref{lem:LAstar_norm} as $\|LA^*\|=\lim_{n\to\infty}\|M_n\|_2$, where the matrix $M_n$ contains the columns $\lambda_k \varphi_k$ for $1\leq k\leq n$. We display in \cref{fig_mechanics_2d}(b) and (d) the spectral norm of $M_n$ as $n$ increases, computed using the randomized SVD, along with our estimate of $\|LA^*\|$.

\section{Convergence Rate for Perturbed Uniformly Elliptic PDEs}\label{sec:Perturbed}

To approximate the solution operator $A$ of a uniformly elliptic partial differential operator with lower order terms, such as \cref{eq_conv_diff_2}, one can exploit the results derived in \cref{sec:FourierSampling} to derive explicit error bounds on the approximant (cf.~\cref{thm:operator_projection}), which involve a constant $\|LA^*\|$ that depends on the magnitude of the convection coefficient $\bm c$. In this section, we study the behavior of the bound as the convection coefficient magnitude increases and derive a convergence rate for the approximation error of the solution operator of \cref{eq_conv_diff_2}. The primary motivation for this analysis originates from the deep learning experiments of \cref{sec_motivation}, which show that the approximation error increases linearly as the convection coefficient magnitude increases (\cref{fig_exp}(d)).

\begin{remark}[Spatial dimension]
    In this section, we assume the spatial dimension satisfies $d \geq 3$ because we rely on a result regarding the norm of the adjoint solution operator~\cite[Eq.~(4.6)]{kim2019green} formulated under this assumption. If this result can be established for $d = 1,2$, then our propositions will also hold, provided we assume that $p \geq 2$.
\end{remark}

Let $\Omega$ be a bounded domain with Lipschitz smooth boundary in $\R^d$ with $d\geq 3$. We consider the Sobolev space $H^1(\Omega)$ consisting of square-integrable functions $u\in L^2(\Omega)$ with first weak derivative in $L^2(\Omega)$, with norm
\[\|u\|_{H^1(\Omega)}=\|u\|_{L^2(\Omega)}+\|\nabla u\|_{L^2(\Omega)}.\]
In addition, we denote $H^1_0(\Omega)$ to be the closure of $C_c^{\infty}(\Omega)$ in $H^1(\Omega)$, where $C_c^{\infty}(\Omega)$ is the set of infinitely differentiable functions that are compactly supported in $\Omega$. If $u\in H^1_0(\Omega)$, then the Sobolev inequality~\citep[Sec.~5.6, Thm.~3]{evans10} states that, for each $q\in [1,2^*]$,
\begin{equation} \label{eq_sobolev}
    \|u\|_{L^{q}(\Omega)}\leq C_{d,q}\|\nabla u\|_{L^2(\Omega)},
\end{equation}
where $2^* = 2d/(d-2)$ is the Sobolev conjugate of $2$ and the constant $C_{d,q}$ depends only on $d$, $q$, and $\Omega$. Let $H^{-1}(\Omega)$ be the dual space of $H^1_0(\Omega)$ such that the norm of a function $F\in H^{-1}(\Omega)$ is given by
\[\|F\|_{H^{-1}(\Omega)}=\sup \{\langle F, u\rangle\colon u\in H^1_0(\Omega),\,\|u\|_{H^1_0(\Omega)}\leq 1\}.\]

We are interested in recovering the solution operator $A$ associated with a second order elliptic operator $\L:H^1_0(\Omega)\to H^{-1}(\Omega)$ of the form
\[
    \L u = -\sum_{i,j=1}^d D_i \left(a^{ij}(x) D_j u\right) + \sum_{i=1}^d c^i(x)D_i u,
\]
which is abbreviated as
\[
    \L u = -\vdiv(\bm A\nabla u)+\bm c \cdot \nabla u.
\]
The matrix $\bm A=(a^{ij})$ is symmetric, consists of measurable bounded coefficient functions and satisfies the uniformly elliptic condition, i.e., there exists a constant $\lambda>0$ such that
\begin{equation} \label{eq_unif_ellipt}
    \lambda|\xi|^2\leq \bm A(x)\xi\cdot\xi,\quad \forall x\in \Omega,\, \xi\in \R^d.
\end{equation}
Following the analysis of~\citet{kim2019green}, we assume that the lower order coefficient vector $\bm c=(c^1,\ldots,c^d)\in L^p(\Omega)$ for some $p>d$. Under these conditions, for every right-hand side $F\in H^{-1}(\Omega)$, there exists a unique $u\in H^1_0(\Omega)$ such that~\citep[Lem.~4.2]{kim2019green}
\begin{equation} \label{eq_L_tilde_elliptic}
    \L u = -\div(\bm A\nabla u)+\bm c \cdot \nabla u = F,\quad \text{in }\Omega.
\end{equation}
Here, \cref{eq_L_tilde_elliptic} holds weakly in the sense that for all $\phi\in C_c^\infty(\Omega)$, we have
\[\int_\Omega \bm A\nabla u\cdot \nabla \phi+\bm c \cdot \nabla u\phi\d x = \int_\Omega F \phi \d x.\]
Then, the solution operator $A:H^{-1}(\Omega)\to H^1_0(\Omega)$, defined as $A(F) = u$, is a bijection. Moreover, the $H^{-1}(\Omega)\to H^1_0(\Omega)$ adjoint of $A$ is given by the operator $A^*:H^{-1}(\Omega)\to H^1_0(\Omega)$, which is the solution operator of the adjoint problem of \cref{eq_L_tilde_elliptic}~\cite[Eq.~(4.6)]{kim2019green}:
\[
    \L^\top=-\div(\bm{A}\nabla u+\bm{c} u)=F,\quad \text{in }\Omega.
\]
When the coefficient vector $\bm{c}$ is sufficiently small, one can view the operator $\L$ as a perturbation of the uniformly elliptic operator $L:H_0^1(\Omega)\to H^{-1}(\Omega)$, defined as
\begin{equation} \label{eq_L_elliptic}
    L u = -\vdiv(\bm A\nabla u), \quad u\in H_0^1(\Omega).
\end{equation}
We denote the solution operator associated with $L$ as $T:H^{-1}(\Omega)\to H^1_0(\Omega)$, which satisfies $T(F) = u$, where $u$ is a solution to $L u = -\div(\bm A\nabla u) = F$.

The following result constructs a low-rank approximation to the solution operator $A$ that is controlled by the eigenvalue decay of the prior uniformly elliptic solution operator $L$. The proof of \cref{sec_approx_H} is deferred to \cref{sec_proof_section_perturbed} and combines \cref{thm:operator_projection} with perturbation theory results of linear operators~\citep{Kato1980perturbation}.

\begin{theorem}[Approximation in $H^{-1}$] \label{sec_approx_H}
    Let $A$ and $T:H^{-1}(\Omega)\to H^1_0(\Omega)$ denote the solution operators associated with the elliptic operators $\L$ and $L$, defined respectively in \cref{eq_L_tilde_elliptic,eq_L_elliptic}. Let $n\geq 1$ and $P_n:H^{-1}(\Omega)\to H^{-1}(\Omega)$ denote the $H^{-1}$-projection onto the space spanned by the first $n$ eigenfunctions of $L$. There exists a constant $C(\Omega,p)$, depending only on $\Omega$ and $p$ such that if $\bm c\in L^p(\Omega)$ satisfies $\|\bm c\|_{L^p(\Omega)}<\lambda/C(\Omega,p)$, then the operator $A$ can be approximated by the operator $A P_n$ in the $H^{-1}(\Omega)$-norm as
    \[
        \|A-A P_n\|_{H^{-1}(\Omega)\to H^{-1}(\Omega)}\leq \frac{1}{\lambda_{n+1}}\frac{\tilde{C}(\bm A,\Omega,p)}{\lambda-C(\Omega,p)\|\bm c\|_{L^p(\Omega)}}\| T \|_{H^{-1}(\Omega) \to H^1_0(\Omega)},
    \]
    where $\lambda_{n+1}$ is the $(n+1)$-th eigenvalue of the operator $L$, and $\tilde{C}(\bm A,\Omega,p)$ is a constant independent of $\bm c$.
\end{theorem}

For a fixed $n\geq 1$, the right-hand side of the approximation bound in \cref{sec_approx_H} depends on the magnitude of the convection coefficient as $\mathcal{O}(1/(\lambda-C(\Omega,p)\|\bm c\|_{L^p(\Omega)}))$. When the perturbation is sufficiently small, such that $C \|\bm c\|_{L^p(\Omega)}\leq \epsilon$, with $0<\epsilon<1$, then an asymptotic estimate of the bound shows that
\[\|A-A P_n\|_{H^{-1}(\Omega)\to H^{-1}(\Omega)} \leq M\frac{\| T \|_{H^{-1}(\Omega) \to H^1_0(\Omega)}}{\lambda_{n+1}}\left(1+\mathcal{O}\left(\|\bm c\|_{L^p(\Omega)}\right)\right),\]
for some constant $M>0$ independent of $\bm c$. This result is in agreement with the numerical experiments of \cref{sec_motivation}, which revealed a linear growth of the error between the learned and exact solution operator of a 1D convection-diffusion equation as the magnitude of the convection term increases (see \cref{fig_exp}(c)). Intuitively, this suggests that \cref{sec_approx_H} might explain why deep learning models do not require the adjoint when learning solution operators associated with PDEs.

\begin{remark}[Asymptotic growth of the eigenvalues]
    The eigenvalues $\{\lambda_j\}_{j\geq 1}$ in \cref{sec_approx_H} follows from applying the spectral theorem for compact self-adjoint operators~\citep[App.~D.4, Thm.~7]{evans10} to the operator $L_{-1}:D(L)=H_0^1(\Omega) \subset H^{-1}(\Omega) \to H^{-1}(\Omega)$. Then, there exists an orthonormal basis $\{\varphi_j\}_{j\geq 1}$ of $H^{-1}(\Omega)$ such that
    \begin{align*}
        L(\varphi_j) & = \lambda_j \varphi_j, \quad \text{in }\Omega, \\
        \varphi_j    & = 0, \quad \text{on }\partial\Omega,
    \end{align*}
    where $\varphi_j\in H_0^1(\Omega)$. In particular, the eigenvalues agree with the standard $L^2$-eigenvalues obtained by restricting $L$ to a domain included in $H^2(\Omega)$, and enjoy the same asymptotic estimates~\citep{beals1970asymptotic}. Moreover, following elliptic regularity results, additional regularity assumptions on the coefficients of $L$ imply that the eigenfunctions $\{\varphi_j\}_{j\geq 1}$ are smooth~\citep[Sec.~6.3]{evans10}.
\end{remark}

While the approximation bound of \cref{sec_approx_H} is stated in the $H^{-1}(\Omega)$-norm and can be difficult to interpret, the result generalizes to higher-order Sobolev spaces, such as $L^2(\Omega)$ and $H^m(\Omega)$ by assuming higher regularity on the coefficients of the PDE, as explained in \cref{thm_approx_L2,rem_higher}.

\begin{theorem}[Approximation in $L^2$] \label{thm_approx_L2}
    Let $A$ and $T:L^2(\Omega)\to H^2_0(\Omega)$ denote the solution operators associated with the elliptic operators $\L$ and $L$, defined respectively in \cref{eq_L_tilde_elliptic,eq_L_elliptic}. Assume that the coefficient matrix $\bm A = (a^{ij})$ satisfies $a^{ij}\in C^1(\Omega)$, the lower order coefficient vector $\bm c=(c^1,\ldots,c^d)\in L^\infty(\Omega)$, and that $\partial \Omega\in C^2$. Let $n\geq 1$ and $P_n:L^2(\Omega)\to L^2(\Omega)$ denote the $L^2$-projection onto the space spanned by the first $n$ eigenfunctions of $L$. There exists a constant $C(\Omega)$, depending only on $\Omega$, such that if $\bm c$ satisfies $\|\bm c\|_{L^\infty(\Omega)}<\lambda/C(\Omega)$, then the operator $A$ can be approximated by the operator $A P_n$ in the $L^2(\Omega)$-norm as
    \[
        \|A-A P_n\|_{L^2(\Omega)\to L^2(\Omega)}\leq \frac{1}{\lambda_{n+1}}\frac{\tilde{C}(\bm A,\Omega,p)}{\lambda-C(\Omega)\|\bm c\|_{L^\infty(\Omega)}}\| T \|_{L^2(\Omega) \to H^2_0(\Omega)},
    \]
    where $\lambda_{n+1}$ is the $(n+1)$-th eigenvalue of the operator $L$, and $\tilde{C}(\bm A,\Omega)$ is a constant independent of $\bm c$.
\end{theorem}

\begin{remark}[Higher-order regularity approximation bounds] \label{rem_higher}
    When the coefficient matrix $\bm{A}$ and the lower order coefficient vector $\bm{c}$ are sufficiently smooth and satisfy the regularity assumptions of \cref{prop_higher_regularity} in \cref{app_perturbation}, we can derive higher-order regularity estimates for the low-rank approximation of the solution operator $A$. This follows naturally from the proof of \cref{sec_approx_H} by combining the perturbation bounds of \cref{prop_higher_regularity} along with the introduction of Sobolev spaces $\mathcal{H}_m$ on which the operator $L$ is self-adjoint (see~\cref{lem_self_adjoint_uniform_2}) for $m\geq 1$. In particular, a compactly supported function $f\in H^m(\Omega)$ on $\Omega$ is in $\mathcal{H}_m$ and the approximation bound holds in the $H^m(\Omega)$-norm.
\end{remark}

\section{Summary and Discussion} \label{sec_conclusions}

We introduced the problem of approximating operators from queries of their action without access to the corresponding adjoint operator. In finite dimensions, we showed that no low-rank recovery algorithm can approximate a matrix efficiently from matrix-vector products unless one has prior information on the angle between its left and right singular subspaces.

Interestingly, compact non-self-adjoint infinite-dimensional operators $A:\mathcal{H}\to\mathcal{H}'$ between Hilbert spaces $\mathcal{H}$ and $\mathcal{H}'$ can be approximated efficiently from operator evaluations without the adjoint if one has prior knowledge about the regularity of the adjoint operator. In this case, we derived low-rank approximation bounds that depend on the eigenvalue decay of a self-adjoint operator $L:\mathcal{H}\to\mathcal{H}$, which captures this regularity information and is used as a prior for approximating the operator $A$. Our result applies naturally to non-self-adjoint differential operators, such as the steady-state advection-diffusion equation, whose adjoint solution operators have Sobolev regularity.
For such operators, powers of the Laplace-Beltrami operator can be used universally for the prior $L$, leading to explicit error bounds and convergence rates.

Our numerical experiments on learning Green's functions associated with advection-diffusion equations revealed that the non-self-adjointness of the operator impacts the performance of deep learning models, which has not been analyzed previously in the literature. Hence, we observed that the error between the learned and exact Green's functions increases linearly with the magnitude of the convection coefficient, which matches the behavior of our theoretical error bound.

This paper is a first step towards understanding non-self-adjoint operator learning and opens up several interesting directions for future research in preconditioning operator learning. Future interesting directions include the selection of the prior differential operator $L$ for generating training datasets, the estimation of the constant $\|LA^*\|$ in our approximation bounds for a variety of problems, the derivation of optimal sample complexity results for general elliptic problems, and the extensions of the non-symmetric matrix recovery results to other classes of matrices, beyond low-rank.

\section*{Data Availability}
Codes and data supporting this paper are publicly available on GitHub at \url{https://github.com/NBoulle/OperatorLearningAdjoint}.

\acks{
    We thank Yuji Nakatsukasa for the discussions related to this work. This work was supported by the Office of Naval Research (ONR) under grant N00014-23-1-2729. N.B. was supported by an INI-Simons Postdoctoral Research Fellowship, D.H. was supported by National Science Foundation grant DGE-2139899, and A.T. was supported by National Science Foundation grants DMS-1952757 and DMS-2045646.}

\vskip 0.2in
\bibliography{biblio}

\appendix

\section{Proofs of Section~\ref{sec:FourierSampling}}
\label{app:Fourier_proj_proofs}

This section contains the proofs of the results exposed in \cref{sec:manifold_without_boundary_case,sec:domain_with_boundary_case}.

\begin{replemma}{lem:LAstar_norm}
    The bounded operator $(L A^*)^*$ extends the operator $A L : D(L) \subset \H \to \H'$ that is densely-defined. Let $\{\tilde{P}_n\}_{n=1}^{\infty}$ be a sequence of orthogonal projections in $\H$ with $\Range(\tilde{P}_n) \subset D(L)$ and $\tilde{P}_n$ converging strongly to the identity.
    Then, we have
    \begin{equation} \tag{\ref{eqn:norm_of_LAstar}}
        \| L A^* \|_{\mcal{H}'\to\mcal{H}}
        = \sup_{\substack{f\in D(L)\\ \| f\|_{\H} \leq 1}} \| A L f \|_{\mcal{H}'}
        = \lim_{n\to\infty} \| A L \tilde{P}_n \|_{\mcal{H}\to\mcal{H}'}.
    \end{equation}
\end{replemma}
\begin{proof}
    Since $L A^*$ is bounded, its adjoint $(L A^*)^*$ is also bounded.
    Choosing $f\in D(L)$ and $v \in \H'$ we have,
    \[
        \big\langle (L A^*)^* f, \ v \big\rangle_{\H'} = \big\langle f, \ L A^* v \big\rangle_{\H} = \big\langle L f, \ A^*v \big\rangle_{\H} = \big\langle A L f, \ v \big\rangle_{\H'}.
    \]
    Since this holds for all $v \in \H'$, we obtain $(L A^*)^* f = A L f$ for all $ f \in D(L)$, which proves the first claim. We now show that \cref{eqn:norm_of_LAstar} holds. Since $D(L)$ is dense in $\H$, we have,
    \[
        \| L A^* \|
        = \| (L A^*)^* \|
        = \sup_{\substack{f\in \H\\ \| f\|_{\H} \leq 1}} \| (L A^*)^* f \|
        = \sup_{\substack{f\in D(L)\\ \| f\|_{\H} \leq 1}} \| A L f \|,
    \]
    The second equality in~\cref{eqn:norm_of_LAstar} follows from the observation that
    \[
        \| A L \tilde{P}_n \|
        = \| (L A^*)^* \tilde{P}_n \|
        = \| \tilde{P}_n  L A^* \|
        \leq \| L A^* \|.
    \]
    Let $\epsilon > 0$, we select $v\in\H'$ satisfying $\| v \| \leq 1$ and $\| L A^* \| \leq \| L A^* v \| + \epsilon$.
    Since $\tilde{P}_n L A^* v \to L A^* v$ as $n\to\infty$, there exists $N_{\epsilon}\geq 1$ such that for $n \geq N_{\epsilon}$, we have
    \[
        \| L A^* \|
        \leq \| \tilde{P}_n L A^* v \| + 2\epsilon
        \leq \| \tilde{P}_n L A^* \| + 2\epsilon
        = \| A L \tilde{P}_n \| + 2 \epsilon.
    \]
    Therefore, $\| A L \tilde{P}_n \| \to \| L A^* \|$ as $n \to \infty$.
\end{proof}

Before proving \cref{thm:manifold_without_boundary_case}, we require some technical results about the Laplace--Beltrami operator on a smooth, compact Riemannian manifold without boundary. First, we establish the Green's identity in \cref{eqn:Green_without_boundary} for the classical LBO in \cref{eqn:classical_LBO} on a (possibly non-orientable) Riemannian manifold by observing that the following equation,
\[
    \div (f_1 \grad f_2)
    = \langle \grad f_1, \grad f_2 \rangle_g - f_1 \Delta_g f_2,
\]
holds for every $f_1, f_2 \in C^{\infty}(\M)$ following \citep[Prob.~16-12(a)]{Lee2013introduction}. Integrating over $\M$ with respect to the Riemannian density $\d \mu_g$ and applying the divergence theorem for non-orientable Riemannian manifolds given by \citep[Thm.~16.48]{Lee2013introduction} yields \cref{eqn:Green_without_boundary}. Green's identity for the classical LBO yields a closable nonnegative quadratic form
\begin{equation}\label{eqn:closable_form_for_LBO}
    \tilde{q} (f_1, f_2) = \int_{\M} \langle \grad f_1, \grad f_2 \rangle_g \d \mu_g,
    \quad f_1, f_2 \in Q(\tilde{q}) = C^{\infty}(\M).
\end{equation}
Its closure $q:Q(q)\times Q(q) \subset L^2(\M)\times L^2(\M) \to \R$ corresponds to a self-adjoint extension $\Delta:D(\Delta)\subset L^2(\M) \to L^2(\M)$ of the classical LBO called the Friedrichs extension~\citep[Thm.~X.23]{Reed1975Fourier}. The Friedrichs extension $\Delta$ is the only self-adjoint extension of $\Delta_g$ whose domain is contained in the form domain $Q(q)$. The following lemma characterizes the form domain. We note that \cref{lem:boundaryless_form_domain,lem:boundaryless_LBO_domain} might already exist in the literature (though we cannot find precise references).

\begin{lemma}
    \label{lem:boundaryless_form_domain}
    Let $\M$ be a smooth, compact Riemannian manifold without boundary.
    Let $q$ be the closure of the quadratic form $\tilde{q}$ defined by \cref{eqn:closable_form_for_LBO} on the domain $Q(\tilde{q}) = C^{\infty}(\M)$. Then, $Q(q) = H^1(\M)$.
\end{lemma}

\begin{proof}
    Recall that $C^{\infty}(\M)$ is dense in $H^1(\M)$~\citep[Sec.~4.3-4.4]{Taylor2011PDEI}, which has norm
    \[
        \| f \|_{H^1}^2
        = \| f \|_{L^2}^2 + \| \grad f \|_{L^2}^2
        = \| f \|_{L^2}^2 + \tilde{q}(f,f).
    \]
    Therefore $C^{\infty}(\M)$ is a form core \citep[p.~277]{Reed1980functional} for $q$ and $Q(q) = H^1(\M)$.
\end{proof}

The following lemma characterizes the domains of some relevant powers of the LBO.

\begin{lemma} \label{lem:boundaryless_LBO_domain}
    Let $\M$ be a smooth, compact Riemannian manifold without boundary and let $\Delta:D(\Delta)\subset L^2(\M) \to L^2(\M)$ be the Friedrichs extension of the classical LBO. Then $\Delta = \overline{\Delta_g}$ and $D(\Delta^{p/2}) = H^p(\M)$ for every $p\geq 1$.
\end{lemma}

\begin{proof}
    Following the discussion in \citep[p.~277-278]{Reed1980functional}, the form domain $Q(q)$ coincides with the domain $D(\Delta^{1/2})$ given by the spectral theorem in multiplication operator form~\citep[Thm.~VIII.4]{Reed1980functional}. Together with \cref{lem:boundaryless_form_domain}, this proves the case $p=1$ of \cref{lem:boundaryless_LBO_domain}.

    Let $k$ be a nonnegative integer and $f \in H^{k+2}(\M)$. By density, there exists a sequence of functions $\{ f_n \}_{n=1}^{\infty} \subset C^{\infty}(\M)$ such that $f_n \to f$ in $H^{k+2}(\M)$. Then, using the classical LBO, we have,
    \[
        \| f_n - f_m \|_{H^k}^2 + \| \Delta_g(f_n - f_m) \|_{H^k}^2 \leq \| f_n - f_m \|_{H^{k+1}}^2 \to 0,
    \]
    as $m,n \to \infty$.
    In particular, this shows that $f \in \D(\overline{\Delta_g}) \subset D(\Delta)$ and $\Delta_g f_n \to \Delta f$ in $H^k(\M)$.
    Therefore, $H^{k+2}(\M) \subset \D(\overline{\Delta_g}) \subset D(\Delta)$ and
    \[
        \| f \|_{H^k}^2 + \| \overline{\Delta_g} f \|_{H^k}^2 \leq \| f \|_{H^{k+2}}^2, \quad f\in H^{k+2}(\M).
    \]
    By induction, for every $q\geq 1$, $H^{2q}(\M) \subset D\big((\overline{\Delta_g})^q\big) \subset D(\Delta^q)$ and
    \[
        \| f \|_{H^k}^2 + \| (\overline{\Delta_g})^q f \|_{H^k}^2 \leq \| f \|_{H^{k+2q}}^2, \quad f\in H^{k+2q}(\M).
    \]
    In particular, taking $k=1$, we have $\Delta^q\big( H^{2q + 1}(\M) \big) \subset H^1(\M) = D(\Delta^{1/2})$. Hence, $H^{2q+1}(\M) \subset D(\Delta^{q + 1/2})$, and $H^{p}(\M) \subset D(\Delta^{p/2})$ for every $p\geq 1$.

    The reverse inclusions are proved using elliptic regularity. The quadratic form $q$ may be viewed as corresponding to an operator $\Delta_q : H^1(\M) \to H^{-1}(\M)$, where $H^{-1}(\M)$ is the dual of $H^1(\M)$ as a subspace of $L^2(\M)$.
    By definition, $\Delta_q$ agrees with $\Delta$ on $D(\Delta)$. Following the elliptic regularity of $\Delta_q$ on compact manifolds without boundary~\citep[p.362-363]{Taylor2011PDEI}, for every $s \geq -1$ the operator $(I + \Delta_q)^{-1}$ is bounded from $H^{s}(\M)$ to $H^{s+2}(\M)$. Taking $f \in D(\Delta)$, we have $(I + \Delta) f \in L^2$ and it follows that $f = (I + \Delta_q)^{-1} (I + \Delta) f \in H^2(\M)$, meaning that $D(\Delta) \subset H^2(\M)$.
    Combining with the earlier result that $H^2(\M) \subset D(\overline{\Delta_g}) \subset D(\Delta)$, we obtain
    \[
        D(\Delta) = D(\overline{\Delta_g}) = H^2(\M),
    \]
    which proves the case $p=2$ of \cref{lem:boundaryless_LBO_domain}.
    Since $\Delta$ is an extension of  $\overline{\Delta_g}$, we have $\Delta = \overline{\Delta_g}$.

    Now, suppose that $D(\Delta^q) = H^{2q}(\M)$ for a positive integer $q$.
    Choosing $f \in D(\Delta^{q+1})$ yields $(I + \Delta) f \in D(\Delta^q) = H^{2q}(\M)$. By elliptic regularity, $f = (I + \Delta_q)^{-1} (I + \Delta) f \in H^{2q + 2}(\M)$.
    Therefore, $D(\Delta^q) = H^{2q}(\M)$ for every $q \geq 1$ by induction, which shows that \cref{lem:boundaryless_LBO_domain} holds for even $p$. Considering the odd $p$ case, we suppose that $D(\Delta^{q-1/2}) = H^{2q-1}(\M)$ for a positive integer $q$. Then, choosing $f \in D(\Delta^{q+1/2})$ yields $(I + \Delta) f \in D(\Delta^{q-1/2}) = H^{2q-1}(\M)$. Moreover, by elliptic regularity, we have $f = (I + \Delta_q)^{-1} (I + \Delta) f \in H^{2q + 1}(\M)$.
    Therefore, $D(\Delta^{q-1/2}) = H^{2q-1}(\M)$ for every $q \geq 1$ by induction.
\end{proof}

Using the above results, we can prove \cref{thm:manifold_without_boundary_case}.
\begin{reptheorem}{thm:manifold_without_boundary_case}
    Let $\M$ be a smooth, compact Riemannian manifold without boundary.
    Then $L^2(\M)$ admits an orthonormal basis consisting of eigenfunctions $\{ \varphi_j \}_{j=1}^{\infty}$ of $\Delta = \overline{\Delta_g}$ with eigenvalues $0 \leq \lambda_1 \leq \lambda_2 \leq \cdots$, and $\lambda_n\to\infty$.
    Let $P_n:L^2(\M) \to L^2(\M)$ denote the orthogonal projection onto $\vspan\{ \varphi_j \}_{j=1}^n$.
    If $\lambda_{n+1} \neq 0$ then for every integer $k\geq 1$ we have
    \begin{equation} \tag{\ref{eqn:function_approx_boundaryless_case}}
        \| f - P_n f \|_{L^2(\mcal{M})}
        \leq \frac{1}{\lambda_{n+1}^{k/2}} \| \Delta^{k/2} f \|_{L^2(\mcal{M})},
        \quad f \in H^k(\M),
    \end{equation}
    and equality is achieved by $f = \varphi_{n+1}$.
    Every bounded operator $A:L^2(\M) \to \H'$ with $\Range(A^*) \subset H^k(\M)$ satisfies $\| \Delta^{k/2} A^* \|_{\mcal{H}' \to L^2(\mcal{M})} < \infty$ and
    \[
        \| A - A P_n \|_{L^2(\mcal{M}) \to \mcal{H}'}
        \leq \frac{1}{\lambda_{n+1}^{k/2}} \| \Delta^{k/2} A^* \|_{\mcal{H}' \to L^2(\mcal{M})},
    \]
    where equality is achieved by the operator $A = (\Delta^{k/2})^{\dagger}$.
\end{reptheorem}
\begin{proof}
    The domain of $\Delta^{k/2}$ is characterized by \cref{lem:boundaryless_LBO_domain}. The stated results are then obtained by directly applying \cref{lem:function_projection} and \cref{thm:operator_projection} with the self-adjoint operator $L = \Delta^{k/2}$.
\end{proof}

\begin{replemma}{lem:self_adjoint_BCs_for_LBO}
    Let $L = \left.\Delta_g^m\right|_{D''} : D(L) = D'' \subset L^2(\mcal{M}) \to L^2(\mcal{M})$ denote the restriction of $\Delta_g^m$ to $D''$, where $\Delta_g^m$ is understood as a differential operator acting on distributions. Suppose that 
    \begin{equation}\tag{\ref{eqn:LBO_power_symmetry_condition}}
        \langle \Delta_g^m u, \ v \rangle_{L^2(\Omega)} = \langle u, \ \Delta_g^m v \rangle_{L^2(\Omega)}, \quad \forall u,v \in D''.
    \end{equation}
    holds and 
    \begin{equation}\tag{\ref{eqn:BVP_for_LBO_power}}
        \mbox{$c u + \Delta_g^m u = f$ in $\Omega$},
        \quad \mbox{and} \quad
        \mbox{$\mcal{B}''_i(u) = 0$, $1\leq i\leq l''$ on $\partial \Omega$}
    \end{equation}
    has a solution $u \in H^{2m}(\Omega)$ satisfying an estimate in the form of
    \begin{equation}\tag{\ref{eqn:global_regularity}}
        \| u \|_{H^{k+s}(\Omega)} \leq C(\Omega, \mcal{L}, s, \sigma) \left( \| f \|_{H^s(\Omega)} + \| u \|_{H^{\sigma}(\Omega)} \right), \quad \forall \sigma < k + s.
    \end{equation}
    for each $f \in L^2(\Omega)$. Then $L$ is self-adjoint and $D(L) = D''$, endowed with the graph norm, is compactly embedded in $L^2(\Omega)$.
\end{replemma}
\begin{proof}
    Since $L = \left.\Delta_g^m\right|_{D''} : D'' \subset L^2(\Omega) \to L^2(\Omega)$ is a differential operator of order $k = 2m$, then it is well-defined. Moreover, the symmetry condition~\eqref{eqn:LBO_power_symmetry_condition} ensures that $D(L) := D'' \subset D(L^*)$. It remains to show that $D(L^*) \subset D''$ to prove that $L$ is self-adjoint. Let $v \in D(L^*)$ and $f \in L^2(\Omega)$.
    Following the assumption that \cref{eqn:BVP_for_LBO_power} has regular solutions, there exist $u\in D''$ satisfying
    \[
        f = c u + L u.
    \]
    Since $v\in D(L^*)$ and $L^*:D(L^*)\subset L^2\to L^2$, then $g\coloneqq c v + L^* v\in L^2$. Therefore, by \cref{eqn:BVP_for_LBO_power}, there exists $\tilde{v}\in H^{2m}\subset L^2$ such that $L\tilde{v}+c\tilde{v}=g$, which implies that
    \[    c v + L^* v = c \tilde{v} + L \tilde{v}.\]
    After multiplying the two equalities above by $v$ and $u$, we obtain
    \[
        \langle f,v \rangle = \langle c u + L u, v \rangle = \langle u, c v + L^* v \rangle = \langle u, c \tilde{v} + L \tilde{v} \rangle = \langle c u + L u, \tilde{v} \rangle
        = \langle f, \tilde{v} \rangle,
    \]
    where the second equality is by definition of the adjoint $L^*$ and the fourth equality follows from the symmetry condition~\eqref{eqn:LBO_power_symmetry_condition}. Therefore, $v = \tilde{v} \in D''$, which implies that $D(L^*) \subset D''$, and shows that $L$ is self-adjoint. The compactness of $D(L) = D''$ in $L^2(\Omega)$ follows from the global regularity estimate~\eqref{eqn:global_regularity}:
    \[
        \| u \|_{H^{2m}(\Omega)}
        \leq C \left( \| L u \|_{L^2(\Omega)} + \| u \|_{L^2(\Omega)} \right)
        \leq C' \| u \|_{H^{2m}(\Omega)},
    \]
    which shows that the graph norm on $D(L)$ is equivalent to the $H^{2m}(\Omega)$-norm. Finally, the Rellich--Kondrachov theorem~\cite[Thm.~6.3]{adams2003sobolev} states that $H^{2m}(\Omega)$, and thus $D(L)$, is compactly embedded in $L^2(\Omega)$, which concludes the proof.
\end{proof}

\begin{reptheorem}{thm:domain_with_boundary_case}
    Let $\Omega \subset \M$ be a domain satisfying the modified $k$-extension property, and $P_n: L^2(\Omega) \to L^2(\Omega)$ denote the orthogonal projection onto $\vspan\{ \left. \varphi_j\right|_{\Omega} \}_{j=1}^n$. There exists a constant $C(\Omega, \M, k)$, depending only on $\Omega$, $\M$, and $k$, so that for every $n\geq 1$ with $\lambda_{n+1}\neq 0$, we have
    \begin{equation}
        \| f - P_n f \|_{L^2(\Omega)}
        \leq \frac{C(\Omega, \M, k)}{\lambda_{n+1}^{k/2}} \| f \|_{H^k(\Omega)},
        \quad f\in H^k(\Omega).
        \tag{\ref{eqn:extension_based_function_approx_with_boundary}}
    \end{equation}
    Moreover, every bounded operator $A:L^2(\Omega) \to \H'$ into a Hilbert space $\H'$ with $\Range(A^*) \subset H^k(\Omega)$ satisfies $\| A^* \|_{\H' \to H^k(\Omega)} < \infty$ and
    \[
        \| A - A P_n \|_{L^2(\mcal{M})\to\mcal{H}'} \leq \frac{C(\Omega, \M, k)}{\lambda_{n+1}^{k/2}}\| A^* \|_{\H' \to H^k(\Omega)}.
    \]
\end{reptheorem}
\begin{proof}
    Let $E:L^2(\Omega) \to L^2(\M)$ be an extension operator satisfying the conditions in \cref{def:modified_p_extension}
    and $R: L^2(\mcal{M}) \to L^2(\Omega)$ be the restriction operator defined as $f \mapsto \left. f \right\vert_{\Omega}$.
    Let $\tilde{P}_n:L^2(\M) \to L^2(\M)$ be the orthogonal projection onto $\vspan\{ \varphi_j \}_{j=1}^n$ and $f \in H^k(\Omega)$. Since $\Range(R \tilde{P}_n E) \subset \vspan\{ \left. \varphi_j\right|_{\Omega} \}_{j=1}^n$, it follows from the projection theorem~\cite[Cor.~5.4 \& Thm.~5.2]{Brezis2010functional} that
    \[
        \| f - P_n f \|_{L^2(\Omega)}
        = \min_{g \in \vspan\{ \left. \varphi_j \right|_{\Omega} \}_{j=1}^n} \| f - g \|_{L^2(\Omega)}
        \leq \| f - R \tilde{P}_n E f \|_{L^2(\Omega)}.
    \]
    We then obtain
    \[
        \| f - R \tilde{P}_n E f \|_{L^2(\Omega)}
        = \| R E f - R \tilde{P}_n E f \|_{L^2(\Omega)} \leq \| E f - \tilde{P}_n E f \|_{L^2(\M)} \leq \frac{1}{\lambda_{n+1}^{k/2}}\| \Delta^{k/2} E f\|_{L^2(\M)},
    \]
    thanks to \cref{thm:manifold_without_boundary_case}.
    Moreover, by \cref{lem:boundaryless_LBO_domain}, the graph norm on $D(\Delta^{k/2})$ is equivalent to the $H^k(\M)$-norm. Hence, there exists a constant $C'(\M,k)$, depending only on $\M$ and $k$, such that
    \[
        \| f - R \tilde{P}_n E f \|_{L^2(\Omega)} \leq \frac{C'(\M,k)}{\lambda_{n+1}^{k/2}} \| E f \|_{H^k(\M)}.
    \]
    Finally, by definition of the modified $k$-extension operator $E$, there exists a constant $C''(\Omega, \M, k)$, depending only on $\Omega$, $\M$, and $k$, such that $\| E f \|_{H^k(\M)} \leq C''(\Omega, \M, k) \| f \|_{H^k(\Omega)}$, meaning that \cref{eqn:extension_based_function_approx_with_boundary} holds with $C(\Omega, \M, k) = C'(\M, k) C''(\Omega, \M, k)$.

    Applying \cref{thm:operator_projection} to the operator $A E^*$ with the operator $L = \Delta^{k/2}$ shows that $\| \Delta^{k/2} E A^* \|_{\mcal{H}'\to L^2(\mcal{M})} < \infty$.
    Using the fact that the graph norm on $D(\Delta^{k/2})$ is equivalent to the $H^k(\M)$ norm again, there exists a constant $C'''(\M,k)$ such that
    \[
        \| f \|_{H^k(\M)} \leq C'''(\M,k) \big(\| \Delta^{k/2} f \|_{L^2(\M)} + \| f \|_{L^2(\M)}\big),
        \quad \forall f \in H^k(\M).
    \]
    Moreover, because $\| f \|_{H^k(\Omega)} \leq \| E f \|_{H^k(\M)}$, for every $f \in H^k(\Omega)$, we obtain
    \begin{align*}
        \| A^* \|_{\H' \to H^k(\Omega)}
         & = \sup_{v\in\H',\ \| v \| \leq 1} \| A^* v \|_{H^k(\Omega)} \leq \sup_{v\in\H',\ \| v \| \leq 1} \| E A^* v \|_{H^k(\M)}             \\
         & \leq \sup_{v\in\H',\ \| v \| \leq 1} C'''(\M,k) \big(\| \Delta^{k/2} E A^* v \|_{L^2(\M)} + \| E A^* v \|_{L^2(\M)}\big)             \\
         & \leq C'''(\M,k) \big(\| \Delta^{k/2} E A^* \|_{\mcal{H}'\to L^2(\mcal{M})} + \| E A^* \|_{\mcal{H}'\to L^2(\mcal{M})}\big) < \infty.
    \end{align*}
    Finally, using \cref{eqn:extension_based_function_approx_with_boundary}, we compute
    \begin{align*}
        \| A - A P_n \|_{L^2(\mcal{M})\to\mcal{H}'}
         & = \sup_{v\in\H',\ \| v \| \leq 1} \| A^*v - P_n A^*v \|_{L^2(\Omega)}                                                                                                                      \\
         & \leq \frac{C(\Omega, \M, k)}{\lambda_{n+1}^{k/2}} \sup_{v\in\H',\ \| v \| \leq 1} \| A^*v \|_{H^k(\Omega)} = \frac{C(\Omega, \M, k)}{\lambda_{n+1}^{k/2}} \| A^* \|_{\H' \to H^k(\Omega)}.
    \end{align*}
\end{proof}

\section{Proofs of Section~\ref{sec:Perturbed}} \label{sec_proof_section_perturbed}

\subsection{Relatively bounded perturbation of elliptic operators} \label{app_perturbation}

We start by controlling the norm of the solution operator of a uniformly elliptic operator with lower order terms with respect to the norm of the convection coefficient using perturbation theory of linear operators~\citep{Kato1980perturbation}.

\begin{lemma}[Relatively bounded perturbation of $L$] \label{lem_rel_bounded}
    Let $\L:H_0^1(\Omega)\to H^{-1}(\Omega)$ and $L:H_0^1(\Omega)\to H^{-1}(\Omega)$ denote the elliptic operators $\L : u \mapsto -\div(\bm A\nabla u)+\bm c \cdot \nabla u$ and $L : u \mapsto -\div(\bm A\nabla u)$ defined respectively in \cref{eq_L_tilde_elliptic,eq_L_elliptic}. Then, the operator $\L-L$ is $L$-bounded and
    \[\|\L u- L u\|_{H^{-1}(\Omega)}=\|\bm c\cdot \nabla u\|_{H^{-1}(\Omega)}\leq a\|u\|_{H_0^1(\Omega)}+b\|L u\|_{H^{-1}(\Omega)},\quad u\in H_0^1(\Omega),\]
    with $a=0$ and $b=C(\Omega, p)\|\bm c\|_{L^p(\Omega)}/\lambda$. Here, $C(\Omega, p)$ is a constant depending only on $\Omega$ and $p$.
\end{lemma}

\begin{proof}
    Let $u\in H_0^1(\Omega)$, we begin the proof by showing that $\bm{c}\cdot \nabla u\in H^{-1}(\Omega)$. By definition, we have
    \[\|\bm c\cdot \nabla u\|_{H^{-1}(\Omega)} = \sup_{\substack{v\in H^1_0(\Omega)\\ \|v\|_{H^1_0(\Omega)}\leq 1}}|\langle \bm c\cdot\nabla u, \, v\rangle|.\]
    Let $v \in H_0^1(\Omega)$ with $\|v\|_{H^1_0}\leq 1$.
    First considering the case when $d \geq 2$, we apply H\"older's inequality with conjugate exponents $q = 2p / (p-2)$ and $q^* = 2p /(p+2)$ to obtain
    \[
        | \langle \bm c \cdot \grad u,\ v \rangle |
        \leq \| \bm c \cdot \grad u \|_{L^{q^*}(\Omega)} \| v \|_{L^q(\Omega)}.
    \]
    Since $2 < q < 2d / (d-2)$ the Sobolev space $H_0^1(\Omega)$ is continuously embedded in $L^q(\Omega)$ thanks to \cite[Thm.~4.2, Part~I, Cases~B~and~C]{adams2003sobolev}. With the embedding constant being $C(\Omega, p)$, this yields
    \[
        | \langle \bm c \cdot \grad u,\ v \rangle |
        \leq C(\Omega, p) \| \bm c \cdot \grad u \|_{L^{q^*}(\Omega)} \| v \|_{H_0^1(\Omega)}.
    \]
    A second application of H\"older's inequality noting that $1/q^* = 1/p + 1/2 \leq 1$ gives
    \[
        \| \bm c \cdot \grad u \|_{L^{q^*}(\Omega)}
        \leq \| \bm c \|_{L^p(\Omega)} \| \grad u \|_{L^2(\Omega)}.
    \]
    Combining these results gives
    \begin{equation}\label{eq_bound_c_nabla_u}
        \| \bm c \cdot \grad u \|_{H^{-1}}
        \leq C(\Omega, p) \| \bm c \|_{L^p(\Omega)} \| \grad u \|_{L^2(\Omega)}.
    \end{equation}

    We will now express $\|\nabla u\|_{L^2(\Omega)}$ in terms of $\|L u\|_{H^{-1}(\Omega)}$ using the ellipticity of $L$. The ellipticity condition in \cref{eq_unif_ellipt} implies that
    \[
        \|Lu\|_{H^{-1}(\Omega)}=\sup_{\substack{v\in H_0^1(\Omega),\\v\neq 0}}\frac{\langle\bm A \nabla u, \nabla v\rangle}{\|v\|_{H_0^1(\Omega)}}\geq \frac{\langle\bm A \nabla u, \nabla u\rangle}{\|u\|_{H_0^1(\Omega)}}\geq \lambda \frac{\|\nabla u\|_{L^2(\Omega)}^2}{\|u\|_{H_0^1(\Omega)}}.
    \]
    However, following Sobolev inequality's (\cref{eq_sobolev}), we have
    \[
        \|u\|_{H_0^1(\Omega)}\leq (1+C_{d,2})\|\nabla u\|_{L^2(\Omega)},
    \]
    which implies that
    \[
        \|\nabla u\|_{L^2(\Omega)}\leq \frac{1+C_{d,2}}{\lambda}\|L u\|_{H^{-1}(\Omega)}.
    \]
    Finally, we combine this inequality with \cref{eq_bound_c_nabla_u} to obtain
    \[
        \|\bm c\cdot \nabla u\|_{H^{-1}(\Omega)}\leq \frac{C(\Omega,p)(1+C_{d,2})}{\lambda}\|\bm c\|_{L^p(\Omega)} \|L u\|_{H^{-1}(\Omega)},
    \]
    which concludes the proof.
\end{proof}

\begin{remark}[Extension to dimensions $d=1,2$]
    The proof of \cref{lem_rel_bounded} remains valid in dimension $d=2$. In dimension $d=1$, one requires $p\geq 2$ and the following modification to derive \cref{eq_bound_c_nabla_u}. When $d = 1$, $H_0^1(\Omega)$ is continuously embedded in $L^{\infty}(\Omega)$ thanks to \cite[Thm.~4.2, Part~I, Case~A]{adams2003sobolev}.
    Using this and H\"older's inequality yields
    \[
        | \langle \bm c \cdot \grad u,\ v \rangle |
        \leq \| \bm c \cdot \grad u \|_{L^{1}(\Omega)} \| v \|_{L^{\infty}(\Omega)}
        \leq C'(\Omega, p) \| \bm c \cdot \grad u \|_{L^{1}(\Omega)} \| v \|_{H_0^1(\Omega)}.
    \]
    Another application of H\"older's inequality noting that $1 = (p-2)/(2p) + 1/p + 1/2$ gives
    \[
        \| \bm c \cdot \grad u \|_{L^{1}(\Omega)}
        \leq | \Omega |^{\frac{p-2}{2p}} \| \bm c \|_{L^p(\Omega)} \| \grad u \|_{L^2(\Omega)}.
    \]
    Combining these results yields \cref{eq_bound_c_nabla_u} for the $d=1$ case with the constant $C(\Omega, p) = C'(\Omega,p) | \Omega |^{\frac{p-2}{2p}}$. The proof of \cref{lem_rel_bounded} then proceeds as before.
\end{remark}

The following proposition provides an explicit bound on the norm of the solution operator $A$ in terms of the norm of the solution operator $T$ that depends explicitly on the convection coefficient $\bm c$. In particular, an asymptotic equivalent of the bound indicates that the right-hand side grows linearly with $\bm c$ in the regime of small perturbations.

\begin{proposition}[$H^{-1}$-perturbation bound] \label{prop_perturbed_inverse}
    Let $A$ and $T:H^{-1}(\Omega)\to H^1_0(\Omega)$ denote the solution operators associated with the elliptic operators $\L : u \mapsto -\div(\bm A\nabla u)+\bm c \cdot \nabla u$ and $L : u \mapsto -\div(\bm A\nabla u)$, defined respectively in \cref{eq_L_tilde_elliptic,eq_L_elliptic}. There exists a constant $C(\Omega,p)$, depending only on $\Omega$ and $p$, such that if $\bm c\in L^p(\Omega)$ satisfies $\|\bm c\|_{L^p(\Omega)}<\lambda/C(\Omega,p)$, then the operator norms of $A$ and $A^*$ are bounded as
    \[\|A^*\|_{H^{-1}(\Omega)\to H^1_0(\Omega)}=\|A\|_{H^{-1}(\Omega)\to H^1_0(\Omega)}\leq \frac{\lambda}{\lambda-C(\Omega,p)\|\bm c\|_{L^p(\Omega)}}\|T\|_{H^{-1}(\Omega)\to H^1_0(\Omega)}.\]
\end{proposition}

\begin{proof}
    Following the Lax-Milgram theorem~\cite[Sec.~6.2.1]{evans10}, $L$ is invertible, and its inverse is given by the solution operator $T:H^{-1}(\Omega)\to H_0^1(\Omega)$. Moreover, following \cref{lem_rel_bounded}, the operator $\L-L:H_0^1(\Omega)\to H^{-1}(\Omega)$ is $L$-bounded with constants $a=0$ and $b=C(\Omega, p)\|\bm c\|_{L^p(\Omega)}/\lambda$. We now apply the stability of bounded invertibility of linear operators~\cite[Sec.~4, Thm.~1.16]{Kato1980perturbation}, which states that if $b<1$, then $\L$ is invertible and its inverse $A$ satisfies
    \[\|A\|_{H^{-1}(\Omega)\to H_0^1(\Omega)}\leq \frac{\|T\|_{H^{-1}(\Omega)\to H_0^1(\Omega)}}{1-a\|T\|_{H^{-1}(\Omega)\to H_0^1(\Omega)}-b}.\]
    Finally, the observation that $A^*$ is the $H^{-1}(\Omega)\to H_0^1(\Omega)$ adjoint of $A$~\cite[Eq.~(4.6)]{kim2019green} ensures that $\|A^*\|_{H^{-1}(\Omega)\to H_0^1(\Omega)}=\|A\|_{H^{-1}(\Omega)\to H_0^1(\Omega)}$.
\end{proof}

The proof of \cref{prop_perturbed_inverse} exploits the stability of perturbed bounded inverse of linear operators~\citep[Chap.~4, Thm.~1.16]{Kato1980perturbation}. Interestingly, \citet[Prop.~6.14]{kim2019green} prove a stronger result than \cref{prop_perturbed_inverse}, which is not based on a perturbation argument, by showing that the norm of the solution operator $A$ is bounded unconditionally on the magnitude of the convection coefficient $\bm c$. However, the resulting bound depends implicitly on the $L^p$-norm of $\bm c$, and making this dependence explicit might be challenging.

When the coefficient functions $\bm A$, $\bm c$, and the boundary of the domain $\partial \Omega$ are smooth, we can derive higher-order regularity estimates for the norm of the solution operator $A$ and its adjoint $A^*$ (see~\cref{prop_higher_regularity}). We first prove an intermediate result.

\begin{lemma}\label{lem_mult_in_Hm_by_Cm_fun}
    Let $\phi \in L^{\infty}(\Omega)$ and $u \in L^2(\Omega)$. Then $\phi u \in L^2(\Omega)$ and
    \[
        \| \phi u \|_{L^2(\Omega)} \leq \| \phi \|_{L^{\infty}(\Omega)} \| u \|_{L^2(\Omega)}.
    \]
    For any integer $m \geq 0$, let $\phi \in C^m(\Omega)$ and $u \in H^m(\Omega)$.
    Then $\phi u \in H^m(\Omega)$ and there is a constant $C(m)$ depending only on $m$ such that
    \[
        \| \phi u \|_{H^m(\Omega)} \leq C(m) \| \phi \|_{C^m(\Omega)} \| u \|_{H^m(\Omega)}
    \]
\end{lemma}
\begin{proof}
    The first assertion and the $m=0$ case are readily established by observing that
    \[
        \| \phi u \|_{L^2(\Omega)}^2 = \int_{\Omega} | \phi |^2 | u |^2 \d x \leq \| \phi \|_{L^{\infty}(\Omega)}^2 \| u \|_{L^2(\Omega)}^2.
    \]
    To prove the remaining cases, we proceed by induction on $m$, assuming the stated result holds for every $\phi \in C^m(\Omega)$ and every $u \in H^m(\Omega)$.
    Choose $\phi \in C^{m+1}(\Omega)$ and $u \in H^{m+1}(\Omega)$.
    For any function $f \in C^{\infty}(\Omega)\cap H^{m+1}(\Omega)$ we have
    \begin{align*}
        \| \phi f \|_{H^{m+1}(\Omega)}
         & \leq \| \phi f \|_{H^{m}(\Omega)} + \| \grad (\phi f) \|_{H^{m}(\Omega)}                                                                                                                  \\
         & \leq \| \phi f \|_{H^{m}(\Omega)} + \| f \grad \phi \|_{H^{m}(\Omega)} + \| \phi \grad f \|_{H^{m}(\Omega)}                                                                               \\
         & \leq C(m)\!\left( \| \phi \|_{C^m(\Omega)} \| f \|_{H^m(\Omega)} {+} \| \grad\phi \|_{C^m(\Omega)} \| f \|_{H^m(\Omega)} {+} \| \phi \|_{C^m(\Omega)} \| \grad f \|_{H^m(\Omega)} \right) \\
         & \leq 3 C(m) \| \phi \|_{C^{m+1}(\Omega)} \| f \|_{H^{m+1}(\Omega)}.
    \end{align*}
    By the Meyers--Serrin Theorem~\cite[Thm.~3.17]{adams2003sobolev}, there exists a sequence $\{ u_k \}_{k=1}^{\infty} \subset C^{\infty}(\Omega)\cap H^{m+1}(\Omega)$ converging to $u$ in $H^{m+1}(\Omega)$. By the $m=0$ case $\phi u_k \to \phi u$ in $L^2(\Omega)$. Taking $f = u_k - u_l$ in the above inequality proves that $\phi u_k$ is a Cauchy sequence in $H^{m+1}(\Omega)$, and thus converges to some limit $g \in H^{m+1}(\Omega)$. By uniqueness of the limits in $L^2(\Omega)$, it follows that $\phi u = g$, proving that $\phi u_k \to \phi u$ in $H^{m+1}(\Omega)$. With $f = u_k$ in the inequality we have
    \[
        \| \phi u_k \|_{H^{m+1}(\Omega)} \leq 3 C(m) \| \phi \|_{C^{m+1}(\Omega)} \| u_k \|_{H^{m+1}(\Omega)}.
    \]
    Since both $\phi u_k \to \phi u$ and $u_k \to u$ in $H^{m+1}(\Omega)$ both sides of the inequality converge, yielding the desired result with $C(m) = 3^m$.
\end{proof}

\begin{proposition}[Higher-order regularity perturbation bounds] \label{prop_higher_regularity}
    With the notations of \cref{prop_perturbed_inverse}, assume that the coefficient matrix $\bm A = (a^{ij})$ satisfies $a^{ij}\in C^1(\Omega)$, the lower order coefficient vector $\bm c=(c^1,\ldots,c^d)\in L^\infty(\Omega)$, and that $\partial \Omega\in C^2$. There exists a constant $C(\Omega)$, depending only on $\Omega$, such that if $\bm c$ satisfies $\|\bm c\|_{L^\infty(\Omega)}<\lambda/C(\Omega)$, then the operator norms of $A$ and $A^*$ are bounded as
    \[\|A^*\|_{L^2(\Omega)\to H^2_0(\Omega)}=\|A\|_{L^2(\Omega)\to H^2_0(\Omega)}\leq \frac{\lambda}{\lambda-C(\Omega)\|\bm c\|_{L^\infty(\Omega)}}\|T\|_{L^2(\Omega)\to H_0^2(\Omega)}.\]
    Moreover, let $m\geq 1$ be an integer and assume that $\Omega \subset \R^3$. Suppose that $a^{ij},\bm c\in C^{m+1}(\Omega)$, and $\partial\Omega\in C^{m+2}$. Then, there exists a constant $C(\Omega, \bm A, m)$, depending only on $\Omega$, $\bm A$, and $m$, such that
    \[\|A^*\|_{H^m(\Omega)\to H^{m+2}_0(\Omega)}{=}\|A\|_{H^m(\Omega)\to H^{m+2}_0(\Omega)}{\leq} \frac{1}{1-C(\Omega,\bm A, m)\|\bm c\|_{C^{m+1}(\Omega)}}\|T\|_{H^m(\Omega)\to H^{m+2}_0(\Omega)}.\]
\end{proposition}

\begin{proof}
    The proof is similar to the one of \cref{prop_perturbed_inverse}. We begin by showing that the operator $\L-L:H_0^2(\Omega)\to L^2(\Omega)$ is $L$-bounded (see~\cref{lem_rel_bounded}). Let $u\in H_0^2(\Omega)$, we aim to control the $L^2$-norm of $(\L-L)u = \bm c\cdot \nabla u$ with respect to $\|L u\|_{L^2(\Omega)}$. Combining H\"older's inequality with the ellipticity of $L$ (\cref{eq_unif_ellipt}) yields
    \begin{equation} \label{eq_elliptic_holder}
        \|\bm c\cdot\nabla u\|_{L^2(\Omega)}\leq \|\bm c\|_{L^\infty(\Omega)}\|\nabla u\|_{L^2(\Omega)},\quad \text{and}\quad \|\nabla u\|_{L^2(\Omega)}^2\leq \frac{1}{\lambda}\langle \bm A \nabla u,\nabla u\rangle_{L^2(\Omega)}.
    \end{equation}
    We then apply Cauchy--Schwarz inequality to obtain
    \[
        \langle \bm A \nabla u,\nabla u\rangle_{L^2(\Omega)}=\langle-\vdiv \bm A\nabla u,u\rangle_{L^2(\Omega)}\leq \|L u\|_{L^2(\Omega)}\|u\|_{L^2(\Omega)}.
    \]
    Following Poincar\'e's inequality (a special case of Sobolev's inequality in~\cref{eq_sobolev}), there exists a constant $C(\Omega)$, depending only on $\Omega$, such that
    \[\|u\|_{L^2(\Omega)}\leq C(\Omega)\|\nabla u\|_{L^2(\Omega)}.\]
    Therefore, after dividing the right inequality in \cref{eq_elliptic_holder} by $\|\nabla u\|_{L^2(\Omega)}$, we have
    \[\|\nabla u\|_{L_2(\Omega)}\leq \frac{C(\Omega)}{\lambda}\|L u\|_{L^2(\Omega)}.\]
    Finally, combining this inequality with \cref{eq_elliptic_holder} shows that
    \[\|\bm c \cdot \nabla u\|_{L^2(\Omega)}\leq \frac{C(\Omega)}{\lambda}\|\bm c\|_{L^\infty(\Omega)}\|L u\|_{L^2(\Omega)}.\]
    Hence, $\L-L$ is $L$-bounded with $a=0$ and $b=C(\Omega)\|\bm c\|_{L^\infty}/\lambda$. We can now apply the stability result for bounded linear operators~\citep[Chap.~4, Thm.~1.16]{Kato1980perturbation} to obtain
    \[\|A\|_{L^2(\Omega)\to H^2(\Omega)}\leq \frac{\lambda}{\lambda-C(\Omega)\|\bm c\|_{L^\infty}}\|T\|_{L^2(\Omega)\to H_0^2(\Omega)}.\]

    We now prove the higher-order estimate. Let $m\geq 1$ be an integer, $u\in H_0^{m+2}(\Omega)$, and $\bm c\in C^{m+1}$. Let $f=-\vdiv\bm a\nabla u\in H^m(\Omega)$, following regularity estimates for the solution to second order elliptic equations~\citep[Chap.~6.3, Thm.~5]{evans10}, there exists a constant $C(\Omega,\bm A, m)$ such that
    \begin{equation} \label{eq_regularity_hm}
        \|u\|_{H_0^{m+2}(\Omega)}\leq C(\Omega,\bm A, m)\|f\|_{H^m(\Omega)}.
    \end{equation}
    Moreover, we estimate the $H^m$-norm as $(\L-L)(u)=\bm c \cdot \nabla u$ using \cref{lem_mult_in_Hm_by_Cm_fun}.
    Hence, there exists a constant $C'(m)$, depending only on $m$, such that
    \[\|\bm c \cdot \nabla u\|_{H^{m+1}(\Omega)}\leq C'(m)\|\bm c\|_{C^{m+1}(\Omega)}\|\nabla u\|_{H^{m+1}(\Omega)},\]
    as $\nabla u\in H^{m+1}(\Omega)$ since $u\in H_0^{m+2}(\Omega)$. Therefore,
    \begin{align*}
        \|\bm c \cdot \nabla u\|_{H^m(\Omega)} & \leq \|\bm c \cdot \nabla u\|_{H^{m+1}(\Omega)}\leq C'( m)\|\bm c\|_{C^{m+1}(\Omega)}\|\nabla u\|_{H^{m+1}(\Omega)} \\
                                               & \leq C'( m)\|\bm c\|_{C^{m+1}(\Omega)}\|u\|_{H_0^{m+2}(\Omega)}.
    \end{align*}
    We combine this inequality with \cref{eq_regularity_hm} to obtain
    \[\|\bm c \cdot \nabla u\|_{H^m(\Omega)}\leq C''(\Omega,\bm A, m)\|\bm c\|_{C^{m+1}(\Omega)}\|f\|_{H^m(\Omega)}=C''(\Omega,\bm A, m)\|\bm c\|_{C^{m+1}(\Omega)}\|L u\|_{H^m(\Omega)},\]
    where $C''(\Omega,\bm A, m)=C(\Omega,\bm A, m)C'(m)$. This shows that $\L-L$ is $L$-bounded with $a=0$ and $b=C''(\Omega,\bm A, m)\|\bm c\|_{C^{m+1}(\Omega)}$. After applying the stability result for bounded linear operators~\citep[Chap.~4, Thm.~1.16]{Kato1980perturbation}, we obtain
    \[\|A\|_{H^m(\Omega)\to H^{m+2}_0(\Omega)}\leq \frac{1}{1-C''(\Omega,\bm A, m)\|\bm c\|_{C^{m+1}(\Omega)}}\|T\|_{H^m(\Omega)\to H^{m+2}_0(\Omega)},\]
    which concludes the proof.
\end{proof}

\subsection{Self-adjointness of $L$ on Sobolev spaces}

We begin by introducing an inner product on $H^{-1}(\Omega)$, that is equivalent to the usual inner product on $H^{-1}(\Omega)$ and is compatible with the inner product on $H^1_0(\Omega)$ induced by the operator $L$. The aim is to show that the operator $L$ is self-adjoint on $H^{-1}(\Omega)$ with respect to this inner product so that \cref{thm:operator_projection} applies.

\begin{lemma} \label{lem_self_adjoint_uniform}
    Let $L:H_0^1(\Omega)\to H^{-1}(\Omega)$ be the uniformly elliptic operator $u \mapsto -\div(\bm A\nabla u)$ defined by \cref{eq_L_elliptic}, and introduce the inner product induced by $L$ on $H_0^1(\Omega)$ as
    \[
        \innerii{u, v}_{H_0^{1}(\Omega)}
        = \langle L u, v \rangle
        = \langle \bm A \grad u, \grad v \rangle_{L^2(\Omega)}.
    \]
    Then, the corresponding norm $\|\cdot\|_{H_0^1(\Omega),L}$ is equivalent to $\|\cdot\|_{H_0^1(\Omega)}$, and the dual space, $H^{-1}(\Omega)$, of $H_0^1(\Omega)$ equipped with $\|\cdot\|_{H_0^1(\Omega),L}$ has inner product given by
    \begin{equation}\label{eq_Hm1_inner_prod}
        \innerii{f, g}_{H^{-1}(\Omega)} = \langle f, T g \rangle.
    \end{equation}
    Moreover, the operator $L_{-1}:D(L_{-1})=H_0^1(\Omega) \subset H^{-1}(\Omega) \to H^{-1}(\Omega)$ defined by $u \mapsto L u$ is self-adjoint with respect to the inner product \cref{eq_Hm1_inner_prod}.
    Here, the inclusion of $H_0^1(\Omega)$ in $H^{-1}(\Omega)$ is understood in the sense of the Gelfand triple $H_0^1(\Omega) \subset L^2(\Omega) \cong L^2(\Omega)^* \subset H^{-1}(\Omega)$ and is given explicitly by
    $u \mapsto \langle u, \cdot \rangle_{L^2(\Omega)}$.
\end{lemma}

\begin{proof}
    By symmetry and positive-definiteness of the coefficient matrix $\bm A$, it follows that $\innerii{u, v}_{H_0^1(\Omega)}$ is symmetric and nonnegative-definite.
    Let $\| u \|_{H_0^1(\Omega),L} := \sqrt{\innerii{u, u}_{H_0^1(\Omega)}}$.
    For $u\in H_0^1(\Omega)$, combining the Poincar\'{e} inequality and the uniform ellipticity of $L$ yields
    \[
        \Vert u \Vert_{H_0^1(\Omega)}^2
        \leq C(\Omega) \int_{\Omega} \vert \grad u \vert^2 \d x
        \leq \frac{C(\Omega)}{\lambda} \int_{\Omega} \bm A \grad u \cdot \grad u \d x
        = \frac{C(\Omega)}{\lambda} \|u\|_{H_0^1(\Omega),L}^2,
    \]
    where $C(\Omega)$ is the Poincar\'{e} constant and $\lambda$ is the uniform ellipticity constant in \cref{eq_unif_ellipt}. Thus, we obtain
    \[\frac{\lambda^{1/2}}{C(\Omega)^{1/2}}\Vert u \Vert_{H_0^1(\Omega)}\leq \|u\|_{H_0^1(\Omega),L}\leq \|\sigma_1(\bm A)\|_{L^\infty(\Omega)}\Vert u \Vert_{H_0^1(\Omega)},\]
    where $\sigma_1$ denotes the largest singular value, and the last inequality follows from the boundedness of the coefficients of $\bm A$.
    Hence, $\|\cdot\|_{H_0^1(\Omega),L}$ is a norm is equivalent to the usual $H_0^1$-norm.
    
    To determine the inner product on the dual space, we observe that the action of $f \in H^{-1}(\Omega)$ on $u \in H_0^1(\Omega)$ can be expressed with respect to the inner product $\innerii{u, v}_{H_0^1(\Omega)}$ as
    \begin{equation}\label{eq_Hm1_action_H1_inner_prod}
        \langle f, u \rangle
        = \langle L T f, u \rangle
        = \innerii{T f, u}_{H_0^1(\Omega)}.
    \end{equation}
    Using this and the symmetry of $\innerii{\cdot, \cdot}_{H_0^1(\Omega)}$, we obtain the symmetry
    \[
        \langle f, T g \rangle
        = \innerii{T f, T g}_{H_0^1(\Omega)}
        = \langle g, T f \rangle \qquad \forall f,g \in H^{-1}(\Omega).
    \]
    Following an approach similar to~\citet[Lem.~2.1]{Bramble1997least}, the corresponding norm on $H^{-1}(\Omega)$ is given by
    \[
        \| f \|_{H^{-1}(\Omega),L}^2
        = \sup_{\substack{u \in H_0^1(\Omega),\\u\neq 0}} \frac{\innerii{T f, u}_{H_0^1(\Omega)}^2}{\|u\|_{H_0^1(\Omega),L}^2}
        = \|T f\|_{H_0^1(\Omega),L}^2
        = \langle f, T f \rangle,
    \]
    where the first equality comes from \cref{eq_Hm1_action_H1_inner_prod} and the second equality follows from the Cauchy--Schwarz inequality. By polarization, the corresponding inner product on $H^{-1}(\Omega)$ is given by \cref{eq_Hm1_inner_prod}.

    To prove self-adjointness, we first show that $L_{-1}: D(L_{-1}) = H_0^1(\Omega) \subset H^{-1}(\Omega) \to H^{-1}(\Omega)$ is symmetric.
    Choosing any $f, g \in D(L_{-1}) = H_0^1(\Omega) \subset H^{-1}(\Omega)$, we have
    \[
        \innerii{f, L_{-1} g}_{H^{-1}(\Omega)}
        = \langle f, T L_{-1} g \rangle
        = \langle f, g \rangle_{L^2(\Omega)}.
    \]
    By symmetry of the inner products, we obtain
    \[
        \innerii{f, L_{-1} g}_{H^{-1}(\Omega)}
        = \innerii{L_{-1} f, g}_{H^{-1}(\Omega)},
    \]
    meaning that $L_{-1}$ is symmetric.

    It remains to show that $D(L_{-1}^*) \subset H_0^1(\Omega)$.
    Let $f \in D(L_{-1}^*) \subset H^{-1}(\Omega)$ and $g \in H^{-1}(\Omega)$.
    We observe that both $u = T L_{-1}^* f$ and $v = T g$ are elements in $H_0^{1}(\Omega)$.
    Therefore, by symmetry, we have
    \begin{align*}
        \innerii{f, g}_{H^{-1}(\Omega)}
         & = \innerii{f, L_{-1} v}_{H^{-1}(\Omega)} = \innerii{L_{-1}^* f, v}_{H^{-1}(\Omega)} = \innerii{L_{-1} u, v}_{H^{-1}(\Omega)} \\
         & = \innerii{u, L_{-1} v}_{H^{-1}(\Omega)} = \innerii{u, g}_{H^{-1}(\Omega)},
    \end{align*}
    which means that $f = u \in H_0^{1}(\Omega)$, and $L_{-1}$ is self-adjoint.
\end{proof}

The following lemma generalizes \cref{lem_self_adjoint_uniform} and defines higher-order Sobolev spaces $\mcal{H}_m$ for $m\geq 0$, equipped with inner products equivalent to the standard $H^m$ inner products, and shows that $L$ is self-adjoint on these spaces under suitable regularity conditions on the coefficients of the operator.

\begin{lemma} \label{lem_self_adjoint_uniform_2}
    Under the same regularity assumptions as in \cref{prop_higher_regularity}, for each integer $m \geq 0$ the space $\mcal{H}_m$, defined by $\mcal{H}_0 = L^2(\Omega)$ and
    \[
        \mcal{H}_m = \left\{ u \in H^m(\Omega) \ : \ L^{k} u = 0 \mbox{ on } \partial \Omega, \ k=0, \ldots, \lceil m/2 \rceil - 1 \right\}, \quad m \geq 1,
    \]
    is a separable Hilbert space with the inner product
    \[
        \langle f, \ g \rangle_{\mcal{H}_m}
        = \begin{cases}
            \langle L^k f, \ L^k g \rangle_{L^2(\Omega)},                   & \text{if } m = 2k \text{ is even},    \\
            \langle \bm A \grad L^k f, \ \grad L^k g \rangle_{L^2(\Omega)}, & \text{if } m = 2k + 1 \text{ is odd},
        \end{cases}
    \]
    which is equivalent to the $H^m(\Omega)$-inner product on $\mcal{H}_m$.

    Moreover, the operator $L_{m} \colon D(L_m) \subset \mcal{H}_m \to \mcal{H}_m$ obtained by restricting $L$ to the domain
    \[
        D(L_m) = \left\{ u \in H^{m+2}(\Omega) \ : \ L^k u = 0 \mbox{ on } \partial \Omega, \ k=0, \ldots, \lceil m/2 \rceil \right\}
    \]
    is self-adjoint.
\end{lemma}

\begin{proof}
    We begin by verifying that the symmetric, nonnegative form $\langle \cdot, \cdot \rangle_{\mcal{H}_m}$ is a valid inner product on $\mcal{H}_m$.
    In the $m=0$ case, there is nothing to prove because $\mcal{H}_0 = L^2(\Omega)$ and the inner products are identical.
    By the trace theorem given by \cite[Chap.~6.4, Thm.~6.108]{Renardy2004introduction}, we observe that $\mcal{H}_m$, $m \geq 1$ is a closed subspace of $H^m(\Omega)$.
    To verify that $\langle \cdot, \cdot \rangle_{\mcal{H}}$ is an inner product on $\mcal{H}_m$, we show that it induces a norm equivalent to the $H^m(\Omega)$ norm on $\mcal{H}_m$.
    The case $m=1$ is evident from \cref{lem_self_adjoint_uniform} and the fact that $\mcal{H}_1 = H_0^1(\Omega)$, see \cite[Chap.~5.5, Thm.~2]{evans10}.
    It is also clear from the boundedness of the coefficients $\bm A$ and their derivatives up to order $m+1$ that
    \[
        \langle f, f \rangle_{\mcal{H}_m} \leq C(\bm A, \Omega, m) \| f \|_{H^m(\Omega)}^2,
    \]
    for a constant $C(\bm A, \Omega, m)$ depending only on the coefficients, the spatial domain, and $m$.
    Hence, it suffices to bound $\langle f, f \rangle_{\mcal{H}_m}$ from below by $\| f \|_{H^m(\Omega)}^2$ times a constant positive factor.

    We proceed by induction separately on even and odd $m$, having already established the base cases $m=0$ and $m=1$.
    Consider $f \in \mcal{H}_{m+2}$.
    By definition, we observe that $L f \in \mcal{H}_m$.
    By the standard elliptic regularity estimate given by \citep[Chap.~6.3, Thm.~5]{evans10} and the induction hypothesis we have
    \[
        \| f \|_{H^{m+2}(\Omega)}
        \leq C'(\bm A, \Omega, m) \| L f \|_{H^{m}(\Omega)}
        \leq C''(\bm A, \Omega, m) \| L f \|_{\mcal{H}_m},
    \]
    for a nonzero constant $C''(\bm A, \Omega, m)$.
    Since
    \[
        \| L f \|_{\mcal{H}_m}^2 = \langle f, f \rangle_{\mcal{H}_{m+2}} =: \| f \|_{\mcal{H}_{m+2}}^2,
    \]
    we obtain the equivalence of the norms $\| \cdot \|_{H^m(\Omega)}$ and $\| \cdot \|_{\mcal{H}_m}$ on every $\mcal{H}_m$ by induction.
    Separability follows from the fact that $H^{m}(\Omega)$ is separable and any subset of a separable metric space is separable~\cite[Prop.~3.25]{Brezis2010functional}.

    To prove that the operator $L_m$ is self-adjoint, it suffices to prove that it is symmetric on $D(L_m)$ and that $D(L_m^*) \subset D(L_m)$, see \cite[Chap.~VIII.2]{Reed1980functional}.
    To prove symmetry, we observe that for every $f \in C_c^{\infty}(\Omega)$ and $g \in C^{\infty}(\Omega)$ we have
    \begin{equation}\label{eqn_H01_H2_L_Green}
        \langle f, \ L g \rangle_{L^2(\Omega)}
        = - \langle \bm A \grad f, \ \grad g \rangle_{L^2(\Omega)}
    \end{equation}
    thanks to the divergence theorem and the identity
    \[
        \div ( f \bm A \grad g)
        = \grad f \cdot \bm A \grad g + f \div (\bm A \grad g).
    \]
    By density of $C_c^{\infty}(\Omega)$ in $H_0^1(\Omega)$ and density of $C^{\infty}(\Omega)$ in $H^2(\Omega)$ by the Meyers--Serrin theorem~\cite[Thm.~3.17]{adams2003sobolev}, it follows that \cref{eqn_H01_H2_L_Green} holds for every $f \in H_0^1(\Omega)$ and every $g \in H^2(\Omega)$.

    Suppose that $m=2k$ is even and choose $u, v \in D(L_m)$.
    This means that $L^k u,\ L^k v \in H_0^1(\Omega) \cap H^2(\Omega)$ and we can apply \cref{eqn_H01_H2_L_Green} to $f=L^k u$ and $g=L^k v$ yielding
    \[
        \langle u,\ L v \rangle_{\mcal{H}_m}
        = \langle L^k u, \ L^{k+1} v \rangle_{L^2(\Omega)}
        = - \langle \bm A \grad L^k u, \ \grad L^k v \rangle_{L^2(\Omega)}
        = \langle L u,\ v \rangle_{\mcal{H}_m}.
    \]
    Now suppose that $m=2k+1$ is odd and choose $u, v \in D(L_m)$.
    This means that $L^{k+1} u,\ L^{k+1} v \in H_0^1(\Omega)$ and that $L^k u,\ L^k v \in H^2(\Omega)$.
    Applying \cref{eqn_H01_H2_L_Green} to $f = L^{k+1} u$ and $g = L^k v$ yields
    \[
        \langle L u,\ v \rangle_{\mcal{H}_m}
        = \langle \bm A \grad L^{k+1} u,\ \grad L^k v \rangle_{L^2(\Omega)}
        = - \langle L^{k+1} u,\ L^{k+1} v \rangle_{L^2(\Omega)}
        = \langle u,\ L v \rangle_{\mcal{H}_m}.
    \]
    Therefore, $L_m$ is a symmetric operator for every integer $m \geq 0$.

    Next we show that $D(L_m^*) \subset D(L_m)$.
    To do this, we choose $u \in D(L_m^*)$ and $f \in \mcal{H}_m$ and we let $u',\ v \in H_0^1$ be the solutions of the elliptic Dirichlet problems
    \[
        L u' = L^* u \quad \mbox{and} \quad L v = f.
    \]
    By the standard elliptic regularity estimate given by \cite[Chap.~6.3, Thm.~5]{evans10}, it follows that $u',\ v \in H^{m+2}(\Omega)$.
    Moreover, since $L u',\ L v \in \mcal{H}_m$ it follows that $L^{k} u' = L^{k} v = 0$ on $\partial \Omega$ for $k=0, \ldots, \lceil m/2 \rceil$, meaning that $u',\ v \in D(L_m)$.
    Using the symmetry of $L_m$, the computation
    \begin{align*}
        \langle u,\ f \rangle_{\mcal{H}_m}
         & = \langle u,\ L_m v \rangle_{\mcal{H}_m}
        = \langle L_m^* u,\ v \rangle_{\mcal{H}_m}
        = \langle L_m u',\ v \rangle_{\mcal{H}_m}
        = \langle u',\ L_m v \rangle_{\mcal{H}_m}   \\
         & = \langle u',\ f \rangle_{\mcal{H}_m},
    \end{align*}
    shows that $u = u' \in D(L_m)$ since $f \in \mcal{H}_m$ is arbitrary.
    Therefore, $D(L_m^*) \subset D(L_m)$, which completes the proof that $L_m$ is self-adjoint.
\end{proof}

We are now ready to apply \cref{thm:operator_projection} to the operator $A$ on the space $\mcal{H} = (H^{-1}(\Omega), \|\cdot\|_{H^{-1}(\Omega),L})$, where $L$ is self-adjoint. Here, $\|\cdot\|_{H^{-1}(\Omega),L}$ is the norm corresponding to the inner product $\innerii{\cdot, \cdot}_{H^{-1}(\Omega)}$ defined in \cref{lem_self_adjoint_uniform}. The equivalence of the norms $\|\cdot\|_{H^{-1}(\Omega),L}$ and $\|\cdot\|_{H^{-1}(\Omega)}$ follows from the uniform ellipticity condition \cref{eq_unif_ellipt} and the Sobolev inequality \cref{eq_sobolev} and allows us to formulate an approximation bound in the usual $H^{-1}$ operator norm.

\begin{reptheorem}{sec_approx_H}[Approximation in $H^{-1}$]
    Let $A$ and $T:H^{-1}(\Omega)\to H^1_0(\Omega)$ denote the solution operators associated with the elliptic operators $\L : u \mapsto -\div(\bm A\nabla u)+\bm c \cdot \nabla u$ and $L : u \mapsto -\div(\bm A\nabla u)$, defined respectively in \cref{eq_L_tilde_elliptic,eq_L_elliptic}. Let $n\geq 1$ and $P_n:H^{-1}(\Omega)\to H^{-1}(\Omega)$ denote the $H^{-1}$-projection onto the space spanned by the first $n$ eigenfunctions of $L$. There exists a constant $C(\Omega,p)$, depending only on $\Omega$ and $p$ such that if $\bm c\in L^p(\Omega)$ satisfies $\|\bm c\|_{L^p(\Omega)}<\lambda/C(\Omega,p)$, then the operator $A$ can be approximated by the operator $A P_n$ in the $H^{-1}(\Omega)$-norm as
    \[
        \|A-A P_n\|_{H^{-1}(\Omega)\to H^{-1}(\Omega)}\leq \frac{1}{\lambda_{n+1}}\frac{\tilde{C}(\bm A,\Omega,p)}{\lambda-C(\Omega,p)\|\bm c\|_{L^p(\Omega)}}\| T \|_{H^{-1}(\Omega) \to H^1_0(\Omega)},
    \]
    where $\lambda_{n+1}$ is the $(n+1)$-th eigenvalue of the operator $L$, and $\tilde{C}(\bm A,\Omega,p)$ is a constant independent of $\bm c$.
\end{reptheorem}
\begin{proof}
    Combining the fact that $H_0^1(\Omega)\subset H^{-1}(\Omega)$ and the equivalence of the norms $\|\cdot\|_{H^{-1}(\Omega),L}$ and $\|\cdot\|_{H^{-1}(\Omega)}$ (see~\cref{lem_self_adjoint_uniform}), we find that $A:\mathcal{H}\to\mathcal{H}$ is a bounded linear operator. Moreover, following \cref{lem_self_adjoint_uniform}, the operator $L$ is self-adjoint on $\mathcal{H}$. Hence, we can apply \cref{thm:operator_projection} on the space $\mcal{H} = (H^{-1}(\Omega), \|\cdot\|_{H^{-1}(\Omega),L})$ to obtain
    \begin{equation} \label{eq_approx_thm_8}
        \| A - A P_n \|_{\mathcal{H} \to \mathcal{H}}
        \leq \frac{1}{\lambda_{n+1}} \| L A^* \|_{\mathcal{H} \to \mathcal{H}}.
    \end{equation}
    However, following the definition of the norm $\|\cdot\|_{H^{-1},L}$, we have
    \[\| L A^* \|_{\mathcal{H} \to \mathcal{H}} = \|A^*\|_{H^{-1},L\to H_{0}^1,L}\leq C(\bm A,\Omega)\|A^*\|_{H^{-1}(\Omega)\to H_0^1(\Omega)},\]
    where the last inequality is due to the equivalence between the norms $\|\cdot\|_{H_0^1(\Omega),L}$, $\|\cdot\|_{H_0^1(\Omega)}$, and $\|\cdot\|_{H^{-1}(\Omega),L}$, $\|\cdot\|_{H_0^1(\Omega)}$. Here, the constant $C(\bm A,\Omega)$ depends only on $\bm A$ and $\Omega$. We then control the term $\|A^*\|_{H^{-1}(\Omega)\to H_0^1(\Omega)}$ using \cref{prop_perturbed_inverse} to obtain
    \[\|A^*\|_{H^{-1}(\Omega)\to H_0^1(\Omega)}\leq \frac{C(\Omega,p)}{\lambda - C(\Omega,p) \|\bm c\|_{L^p(\Omega)}} \| T \|_{H^{-1}(\Omega) \to H^1_0(\Omega)}.\]
    On the other hand, we exploit the equivalence between $\|\cdot\|_{H^{-1}(\Omega),L}$ and $\|\cdot\|_{H^{-1}(\Omega)}$ in \cref{eq_approx_thm_8} as
    \[C'(\bm A, \Omega)\|A-AP_n\|_{H^{-1}(\Omega)\to H^{-1}(\Omega)}\leq \frac{1}{\lambda_{n+1}} \| L A^* \|_{\mathcal{H} \to \mathcal{H}}.\]
    After introducing the constant $\tilde{C}(\bm A,\Omega,p)=C(\Omega,p)C(\bm A,\Omega)/C'(\bm A,\Omega)$, we conclude that
    \[\|A-A P_n\|_{H^{-1}(\Omega)\to H^{-1}(\Omega)}\leq \frac{1}{\lambda_{n+1}}\frac{\tilde{C}(\bm A,\Omega,p)}{\lambda-C(\Omega,p)\|\bm c\|_{L^p(\Omega)}}\| T \|_{H^{-1}(\Omega) \to H^1_0(\Omega)}.\]
\end{proof}

The proof of \cref{thm_approx_L2}, and its analog in higher-order Sobolev spaces, is identical to the proof of \cref{sec_approx_H} using \cref{lem_self_adjoint_uniform_2} instead of \cref{lem_self_adjoint_uniform}.

\end{document}